\documentclass[12pt,reqno,openbib,runningheads,a4paper]{amsart}%
\usepackage{graphicx}
\usepackage{amssymb}
\usepackage{amsfonts}
\usepackage{amsmath}
\usepackage{graphicx}
\usepackage{lineno}
\usepackage[authoryear]{natbib}
\usepackage[dvipsnames]{xcolor}
\usepackage{epsf}
\usepackage{lscape}
\usepackage[legalpaper,bookmarks=true,colorlinks=true,linkcolor=blue,citecolor=blue]%
{hyperref}%
\providecommand{\U}[1]{\protect\rule{.1in}{.1in}}
\providecommand{\U}[1]{\protect\rule{.1in}{.1in}}
\newtheorem{theorem}{Theorem}[section]

\newtheorem{proposition}{Proposition}[section]

\newtheorem{lemma}{Lemma}[section]

\newtheorem{remark}{Remark}[section]

\makeatletter
\renewcommand{\@biblabel}[1]{}
\makeatother
\begin{document}

\begin{center}
{\Large \textbf{Nelson-Aalen tail product-limit process and extreme value
index estimation under random censorship}}

\medskip\medskip

{\large Brahim Brahimi, Djamel Meraghni, Abdelhakim Necir}$^{\ast}$\medskip

{\small \textit{Laboratory of Applied Mathematics, Mohamed Khider University,
Biskra, Algeria}}\medskip\medskip%
\[
\]

\end{center}

\noindent\textbf{Abstract}\medskip

\noindent On the basis of Nelson-Aalen nonparametric estimator of the
cumulative distribution function, we provide a weak approximation to tail
product-limit process for randomly right-censored heavy-tailed data. In this
context, a new consistent reduced-bias estimator of the extreme value index is
introduced and its asymptotic normality is established only by assuming the
second-order regular variation of the underlying distribution function. A
simulation study shows that the newly proposed estimator performs better than
the existing ones.\medskip\medskip

\medskip


\vfill

\vfill

\noindent{\small $^{\text{*}}$Corresponding author:
\texttt{necirabdelhakim@yahoo.fr} \newline\noindent\textit{E-mail
addresses:}\newline\texttt{brah.brahim@gmail.com} (B.~Brahimi)\newline%
\texttt{djmeraghni@yahoo.com} (D.~Meraghni)}

\section{\textbf{Introduction\label{sec1}}}

\noindent Let $X_{1},...,X_{n}$ be $n\geq1$ independent copies of a
non-negative random variable (rv) $X,$ defined over some probability space
$\left(  \Omega,\mathcal{A},\mathbf{P}\right)  ,$ with a cumulative
distribution function (cdf) $F.\ $These rv's are censored to the right by a
sequence of independent copies $Y_{1},...,Y_{n}$ of a non-negative rv $Y,$
independent of $X$ and having a cdf $G.$ At each stage $1\leq j\leq n,$ we can
only observe the rv's $Z_{j}:=\min\left(  X_{j},Y_{j}\right)  $ and
$\delta_{j}:=\mathbf{1}\left\{  X_{j}\leq Y_{j}\right\}  ,$ with
$\mathbf{1}\left\{  \cdot\right\}  $ denoting the indicator function. The
latter rv indicates whether there has been censorship or not. If we denote by
$H$ the cdf of the observed $Z^{\prime}s,$ then, by the independence of $X$
and $Y,$ we have $1-H=\left(  1-F\right)  \left(  1-G\right)  .$ Throughout
the paper, we will use the notation $\overline{\mathcal{S}}(x):=\mathcal{S}%
(\infty)-\mathcal{S}(x),$ for any $\mathcal{S}.$ Assume further that $F$ and
$G$ are heavy-tailed or, in other words, that $\overline{F}$ and $\overline
{G}$ are regularly varying at infinity with negative indices $-1/\gamma_{1}$
and $-1/\gamma_{2}$ respectively. That is%
\begin{equation}
\lim_{t\rightarrow\infty}\frac{\overline{F}\left(  tx\right)  }{\overline
{F}\left(  t\right)  }=x^{-1/\gamma_{1}}\text{ and }\lim_{t\rightarrow\infty
}\frac{\overline{G}\left(  tx\right)  }{\overline{G}\left(  t\right)
}=x^{-1/\gamma_{2}}, \label{R-F}%
\end{equation}
for any $x>0.$\ This class of distributions\textbf{ }includes models such as
Pareto, Burr, Fr\'{e}chet, $\alpha-$stable $\left(  0<\alpha<2\right)  $ and
log-gamma, known to be very appropriate for fitting large insurance claims,
large fluctuations of prices, financial log-returns, ...
\citep[see, e.g.,][]{Res06}. The regular variation of $\overline{F}$ and
$\overline{G}$ implies that $\overline{H}$ is regularly varying as well, with
index $-1/\gamma$ where $\gamma:=\gamma_{1}\gamma_{2}/\left(  \gamma
_{1}+\gamma_{2}\right)  .$ Since weak approximations of extreme value theory
based statistics are achieved in the second-order framework
\citep[see][]{deHS96}, then it seems quite natural to suppose that cdf $F$
satisfies the well-known second-order condition of regular variation. That is,
we assume that for any $x>0$%
\begin{equation}
\underset{t\rightarrow\infty}{\lim}\frac{U_{F}\left(  tx\right)  /U_{F}\left(
t\right)  -x^{\gamma_{1}}}{A_{1}^{\ast}\left(  t\right)  }=x^{\gamma_{1}%
}\dfrac{x^{\tau}-1}{\tau}, \label{second-order}%
\end{equation}
where $\tau\leq0$\ is the second-order parameter and $A_{1}^{\ast}$ is a
function tending to $0,$ not changing sign near infinity and having a
regularly varying absolute value at infinity with index $\tau.$\ If $\tau=0,$
interpret $\left(  x^{\tau}-1\right)  /\tau$ as $\log x.$ In the sequel, the
functions $K^{\leftarrow}\left(  s\right)  :=\inf\left\{  x:K\left(  x\right)
\geq s\right\}  ,$ $0<s<1,$ and $U_{K}\left(  t\right)  :=K^{\leftarrow
}\left(  1-1/t\right)  ,$ $t>1,$ respectively stand for the quantile and tail
quantile functions of any given cdf $K.$ The analysis of extreme values of
randomly censored data, is a new research topic to which \cite{ReTo07} made a
very brief reference, in Section 6.1, as a first step but with no asymptotic
results. \cite{BeGDFF07} proposed estimators for the extreme value index (EVI)
$\gamma_{1}$ and high quantiles and discussed their asymptotic properties,
when the data are censored by a deterministic threshold. For their part,
\cite{EnFG08} adapted various classical EVI estimators to the case where data
are censored, by a random threshold, and proposed a unified method to
establish their asymptotic normality. Their approach is used by \cite{NDD2014}
to address the nonparametric estimation of the conditional EVI and large
quantiles. Based on Kaplan-Meier integration, \cite{WW2014} introduced two new
estimators and proved their consistency. They showed, by simulation, that they
perform better, in terms of bias and mean squared error (MSE) than the adapted
Hill estimator of \cite{EnFG08}, in the case where the tail index of the
censored distribution is less than that of the censoring one. \cite{BMN-2015}
used the empirical process theory to approximate the adapted Hill estimator in
terms of Gaussian processes, then they derived its asymptotic normality only
under the usual second-order condition of regular variation. Their approach
allows to relax the assumptions, made in \cite{EnFG08}, on the heavy-tailed
distribution functions and the sample fraction of upper order statistics used
in estimate computation. Recently, \cite{BBWG-16} developed improved
estimators for the EVI and tail probabilities by reducing their biases which
can be quite substantial. In this paper, we develop a new methodology, for the
estimation of the tail index under random censorship, by considering the
nonparametric estimator of cdf $F,$ based on Nelson-Aalen estimator
\citep[][]{Nelson, Aalen} of the cumulative hazard function%
\[
\Lambda\left(  z\right)  :=\int_{0}^{z}\frac{dF\left(  v\right)  }%
{\overline{F}\left(  v\right)  }=\int_{0}^{z}\frac{dH^{\left(  1\right)
}\left(  v\right)  }{\overline{H}\left(  v\right)  },
\]
where $H^{\left(  1\right)  }\left(  z\right)  :=\mathbf{P}\left(  Z_{1}\leq
z,\delta_{1}=1\right)  =\int_{0}^{z}\overline{G}\left(  y\right)  dF\left(
y\right)  ,$ $z\geq0.$ A natural nonparametric estimator $\Lambda_{n}\left(
z\right)  $ of $\Lambda\left(  z\right)  $ is obtained by replacing $H\left(
v\right)  $ and $H^{\left(  1\right)  }\left(  v\right)  $ by their respective
empirical counterparts $H_{n}\left(  v\right)  :=n^{-1}\sum_{i=1}%
^{n}\mathbf{1}\left(  Z_{i}\leq v\right)  $ and $H_{n}^{\left(  1\right)
}\left(  v\right)  :=n^{-1}\sum_{i=1}^{n}\delta_{i}\mathbf{1}\left(  Z_{i}\leq
v\right)  $ pertaining to the observed $Z$-sample. However, since
$\overline{H}_{n}\left(  Z_{n:n}\right)  =0,$ we use $\overline{H}_{n}\left(
v-\right)  :=\lim_{u\uparrow v}\overline{H}_{n}\left(  u\right)  $ instead of
$\overline{H}_{n}\left(  v\right)  $ \citep[see, e.g.,][page 295]{SW86} to get%
\[
\Lambda_{n}\left(  z\right)  :=%
%TCIMACRO{\dint _{0}^{z}}%
%BeginExpansion
{\displaystyle\int_{0}^{z}}
%EndExpansion
\dfrac{dH_{n}^{\left(  1\right)  }\left(  v\right)  }{\overline{H}_{n}\left(
v-\right)  }=%
%TCIMACRO{\dsum \limits_{i=1}^{n}}%
%BeginExpansion
{\displaystyle\sum\limits_{i=1}^{n}}
%EndExpansion
\dfrac{\delta_{i}\mathbf{1}\left(  Z_{i}\leq z\right)  }{n-i+1},\text{ for
}z\leq Z_{n:n}.
\]
From the definition of $\Lambda\left(  z\right)  ,$ we deduce that
$\overline{F}\left(  z\right)  =\exp\left\{  -\Lambda\left(  z\right)
\right\}  ,$ which by substituting $\Lambda_{n}\left(  z\right)  $ for
$\Lambda\left(  z\right)  ,$ yields Nelson-Aalen estimator of cdf $F,$ given
by%
\[
F_{n}^{\left(  NA\right)  }\left(  z\right)  :=1-%
%TCIMACRO{\dprod \limits_{i:Z_{i:n}\leq z}}%
%BeginExpansion
{\displaystyle\prod\limits_{i:Z_{i:n}\leq z}}
%EndExpansion
\exp\left\{  -\dfrac{\delta_{\left[  i:n\right]  }}{n-i+1}\right\}  ,\text{
for }z\leq Z_{n:n},
\]
where $Z_{1:n}\leq...\leq Z_{n:n}$ denote the order statistics, pertaining to
the sample $\left(  Z_{1},...,Z_{n}\right)  $ and $\delta_{\left[  1:n\right]
},...,\delta_{\left[  n:n\right]  }$ the associated concomitants so that
$\delta_{\left[  j:n\right]  }=\delta_{i}$ if $Z_{j:n}=Z_{i}.$ Note that for
our needs and in the spirit of what \cite{Efron} did to Kaplan-Meier estimator
\citep[][]{KM58} (given in $\left(  \ref{KM}\right)  ),$ we complete
$F_{n}^{\left(  NA\right)  }$ beyond the largest observation by $1.$ In other
words, we define a nonparametric estimator of cdf $F$ by%
\[
F_{n}\left(  z\right)  :=\left\{
\begin{array}
[c]{lcc}%
F_{n}^{\left(  NA\right)  }\left(  z\right)  & \text{for} & z\leq Z_{n:n},\\
1 & \text{for} & z>Z_{n:n}.
\end{array}
\right.
\]
By considering samples of various sizes, \cite{FH84} numerically compared
$F_{n}$ with Kaplan-Meier (nonparametric maximum likelihood) estimator of $F$
\citep[][]{KM58}, given in $\left(  \ref{KM}\right)  ,$ and pointed out that
they are asymptotically equivalent and usually quite close to each other. A
nice discussion on the tight relationship between the two estimators may be
found in \cite{HS6}.\medskip

\noindent In the spirit of the tail product-limit process for randomly
right-truncated data, recently introduced by \cite{BMN2016}, we define
Nelson-Aalen tail product-limit process by%
\begin{equation}
D_{n}\left(  x\right)  :=\sqrt{k}\left(  \frac{\overline{F}_{n}\left(
xZ_{n-k:n}\right)  }{\overline{F}_{n}\left(  Z_{n-k:n}\right)  }%
-x^{-1/\gamma_{1}}\right)  ,\text{ }x>0, \label{KM-P}%
\end{equation}
where $k=k_{n}$ is an integer sequence satisfying suitable assumptions. In the
case of complete data, the process $D_{n}\left(  x\right)  $ reduces to
$\mathbf{D}_{n}\left(  x\right)  :=\sqrt{k}\left(  \dfrac{n}{k}\overline
{\mathbf{F}}_{n}\left(  xX_{n-k:n}\right)  -x^{-1/\gamma_{1}}\right)  ,$
$x>0,$ where $\mathbf{F}_{n}\left(  x\right)  :=n^{-1}\sum_{i=1}^{n}%
\mathbf{1}\left(  X_{i}\leq x\right)  $ is the usual empirical cdf based on
the fully observed sample $\left(  X_{1},...,X_{n}\right)  .$ Combining
Theorems 2.4.8 and 5.1.4 in \cite{deHF06}, we infer that under the
second-order condition of regular variation $\left(  \ref{second-order}%
\right)  ,$ there exists a sequence of standard Weiner processes $\left\{
\mathcal{W}_{n}\left(  s\right)  ;\text{ }0\leq s\leq1\right\}  $ such that,
for any $x_{0}>0$ and $0<\epsilon<1/2,$%
\begin{equation}
\sup_{x\geq x_{0}}x^{\epsilon/\gamma_{1}}\left\vert \mathbf{D}_{n}\left(
x\right)  -\mathcal{J}_{n}\left(  x\right)  -x^{-1/\gamma_{1}}\dfrac
{x^{\tau/\gamma_{1}}-1}{\tau/\gamma_{1}}\sqrt{k}A_{1}\left(  n/k\right)
\right\vert \overset{\mathbf{p}}{\rightarrow}0,\text{ as }n\rightarrow\infty,
\label{TP-C}%
\end{equation}
where $A_{1}\left(  t\right)  :=A_{1}^{\ast}\left(  1/\overline{F}\left(
t\right)  \right)  $ and $\mathcal{J}_{n}\left(  x\right)  :=\mathcal{W}%
_{n}\left(  x^{-1/\gamma_{1}}\right)  -x^{-1/\gamma_{1}}\mathcal{W}_{n}\left(
1\right)  .$ One of the main applications of the this weak approximation is
the asymptotic normality of tail indices. Indeed, let us consider Hill's
estimator \citep[][]{Hill75}%
\[
\widehat{\gamma}_{1}^{\left(  H\right)  }:=k^{-1}\sum_{i=1}^{k}\log\left(
X_{n-i+1:n}/X_{n-k:n}\right)  ,
\]
which may be represented, as a functional of the process $\mathbf{D}%
_{n}\left(  x\right)  ,$ into%
\[
\sqrt{k}\left(  \widehat{\gamma}_{1}^{\left(  H\right)  }-\gamma_{1}\right)
=\int_{1}^{\infty}x^{-1}\mathbf{D}_{n}\left(  x\right)  dx.
\]
It follows, in view of approximation $\left(  \ref{TP-C}\right)  ,$ that
\[
\sqrt{k}\left(  \widehat{\gamma}_{1}^{\left(  H\right)  }-\gamma_{1}\right)
=\int_{1}^{\infty}x^{-1}\mathcal{J}_{n}\left(  x\right)  dx+\dfrac{\sqrt
{k}A_{1}\left(  n/k\right)  }{1-\tau}+o_{\mathbf{p}}\left(  1\right)  ,
\]
leading to $\sqrt{k}\left(  \widehat{\gamma}_{1}^{\left(  H\right)  }%
-\gamma_{1}\right)  \overset{\mathcal{D}}{\rightarrow}\mathcal{N}\left(
\dfrac{\widetilde{\lambda}}{1-\tau},\gamma_{1}^{2}\right)  ,$ provided that
$\sqrt{k}A_{1}\left(  n/k\right)  \rightarrow\widetilde{\lambda}<\infty.$ For
more details on this matter, see for instance, \cite{deHF06} page 76.\medskip\ 

\noindent The major goal of this paper is to provide an analogous result to
(\ref{TP-C}) in the random censoring setting through the tail product-limit
process $D_{n}\left(  x\right)  $ and propose a new asymptotically normal
estimator of the tail index. To the best of our knowledge, this approach has
not been considered yet in the extreme value theory literature.\ Our
methodology is based on the uniform empirical process theory \citep[][]{SW86}
and the related weak approximations \citep[][]{CsCsHM86}. Our main result,
given in Section \ref{sec2} consists in the asymptotic representation of
$D_{n}\left(  x\right)  $ in terms of Weiner processes. As an application, we
introduce, in Section \ref{sec3}, a Hill-type estimator for the tail index
$\gamma_{1}$ based on $F_{n}.$ The asymptotic normality of the newly proposed
estimator is established, in the same section, by means of the aforementioned
Gaussian approximation of $D_{n}\left(  x\right)  $ and its finite sample
behavior is checked by simulation in Section \ref{sec4}. The proofs are
postponed to Section \ref{sec5} and some results that are instrumental to our
needs are gathered in\textbf{ }the Appendix.

\section{\textbf{Main result\label{sec2}}}

\begin{theorem}
\label{Theorem1} Let $F$ and $G$ be two cdf's with regularly varying tails
(\ref{R-F}) and assume that the second-order condition of regular variation
(\ref{second-order}) holds with $\gamma_{1}<\gamma_{2}.\ $Let $k=k_{n}$ be an
integer sequence such that $k\rightarrow\infty,$ $k/n\rightarrow0$ and
$\sqrt{k}A_{1}\left(  h\right)  =O\left(  1\right)  ,$ where $h=h_{n}%
:=U_{H}\left(  n/k\right)  .$ Then, there exists a sequence of standard Weiner
processes $\left\{  W_{n}\left(  s\right)  ;\text{ }0\leq s\leq1\right\}  $
defined on the probability space $\left(  \Omega,\mathcal{A},\mathbf{P}%
\right)  ,$ such that for every $0<\epsilon<1/2,$ we have, as $n\rightarrow
\infty,$%
\begin{equation}
\sup_{x\geq p^{\gamma}}x^{\epsilon/\left(  p\gamma_{1}\right)  }\left\vert
D_{n}\left(  x\right)  -J_{n}\left(  x\right)  -x^{-1/\gamma_{1}}%
\dfrac{x^{\tau/\gamma_{1}}-1}{\tau/\gamma_{1}}\sqrt{k}A_{1}\left(  h\right)
\right\vert \overset{\mathbf{p}}{\rightarrow}0, \label{TP-C1}%
\end{equation}
where $J_{n}\left(  x\right)  =J_{1n}\left(  x\right)  +J_{2n}\left(
x\right)  ,$ with%
\[
J_{1n}\left(  x\right)  :=\sqrt{\frac{n}{k}}\left\{  x^{1/\gamma_{2}%
}\mathbf{W}_{n,1}\left(  \frac{k}{n}x^{-1/\gamma}\right)  -x^{-1/\gamma_{1}%
}\mathbf{W}_{n,1}\left(  \frac{k}{n}\right)  \right\}  ,
\]
and
\[
J_{2n}\left(  x\right)  :=\dfrac{x^{-1/\gamma_{1}}}{\gamma}\sqrt{\dfrac{n}{k}}%
%TCIMACRO{\dint _{1}^{x}}%
%BeginExpansion
{\displaystyle\int_{1}^{x}}
%EndExpansion
u^{1/\gamma-1}\left\{  p\mathbf{W}_{n,2}\left(  \dfrac{k}{n}u^{-1/\gamma
}\right)  -q\mathbf{W}_{n,1}\left(  \dfrac{k}{n}u^{-1/\gamma}\right)
\right\}  du,
\]
where $\mathbf{W}_{n,1}$ and $\mathbf{W}_{n,2}$ are two independent Weiner
processes defined, for $0\leq s\leq1,$ by%
\[
\mathbf{W}_{n,1}\left(  s\right)  :=\left\{  W_{n}\left(  \theta\right)
-W_{n}\left(  \theta-ps\right)  \right\}  \mathbf{1}\left(  \theta
-ps\geq0\right)  \text{ and }\mathbf{W}_{n,2}\left(  s\right)  :=W_{n}\left(
1\right)  -W_{n}\left(  1-qs\right)  ,
\]
with $\theta:=H^{\left(  1\right)  }\left(  \infty\right)  ,$ $p=1-q:=\gamma
/\gamma_{1}.$
\end{theorem}

\begin{remark}
\label{Rmk1}It is noteworthy that the assumption $\gamma_{1}<\gamma_{2}$ is
required to ensure that enough extreme data is available for the inference to
be accurate. In other words, the proportion $p$ of the observed extreme values
has to be greater than $1/2.$ This assumption is already considered by
\cite{WW2014} and, in the random truncation context, by \cite{GS2015} and
\cite{BMN2016}.
\end{remark}

\begin{remark}
\label{Rmk2}In the complete data case, we use $\mathbf{F}_{n}$\ instead of
$F_{n}\ $and we have $p=1=\theta.$ This implies that $q=0,$ $H\equiv F,$
$U_{F}\left(  h\right)  =n/k,$ $J_{1n}\left(  x\right)  \overset{\mathcal{D}%
}{=}\mathbf{W}_{n}\left(  x^{-1/\gamma_{1}}\right)  -x^{-1/\gamma_{1}%
}\mathbf{W}_{n}\left(  1\right)  ,$ $J_{2n}\left(  x\right)  =0$ and
$\mathbf{W}_{n}\left(  s\right)  =W_{n}\left(  1\right)  -W_{n}\left(
1-s\right)  .$ Since
\[
\left\{  W_{n}\left(  s\right)  ,\text{ }0\leq s\leq1\right\}  \overset
{\mathcal{D}}{=}\left\{  W_{n}\left(  1\right)  -W_{n}\left(  1-s\right)
,\text{ }0\leq s\leq1\right\}  ,
\]
it follows that $J_{1n}\left(  x\right)  \overset{\mathcal{D}}{=}W_{n}\left(
x^{-1/\gamma_{1}}\right)  -x^{-1/\gamma_{1}}W_{n}\left(  1\right)  $ and so
approximations $\left(  \ref{TP-C1}\right)  $ and $\left(  \ref{TP-C}\right)
$ agree for $x_{0}=p^{\gamma}.$ The symbol $\overset{\mathcal{D}}{=}$ stands
for equality in distribution.
\end{remark}

\section{\textbf{Tail index estimation \label{sec3}}}

\noindent In the last decade, some authors began to be attracted by the
estimation of the EVI $\gamma_{1}$ when the data are subject to random
censoring. For instance, \cite{EnFG08} adapted the classical Hill estimator
(amongst others) to a censored sample to introduce $\widehat{\gamma}%
_{1}^{\left(  EFG\right)  }:=\widehat{\gamma}^{\left(  H\right)  }/\widehat
{p}$ as an asymptotically normal estimator for $\gamma_{1},$ where
$\widehat{\gamma}^{\left(  H\right)  }:=k^{-1}\sum_{i=1}^{k}\log\left(
Z_{n-i+1:n}/Z_{n-k:n}\right)  ,$ is Hill's estimator of $\gamma$ based on the
complete sample $Z_{1},...,Z_{n}$ and $\widehat{p}:=k^{-1}\sum_{i=1}^{k}%
\delta_{\left[  n-i+1:n\right]  }.$ By using the empirical process theory
tools and only by assuming the second-order condition of regular variation of
the tails of $F$ and $G,$ \cite{BMN-2015} derived, in Theorem 2.1 (assertion
$\left(  2.9\right)  ),$ a useful weak approximation to $\widehat{\gamma}%
_{1}^{\left(  EFG\right)  }$ in terms of a sequence of Brownian bridges. They
deduced in Corollary 2.1 that%
\begin{equation}
\sqrt{k}\left(  \widehat{\gamma}_{1}^{\left(  EFG\right)  }-\gamma_{1}\right)
\overset{\mathcal{D}}{\rightarrow}\mathcal{N}\left(  \dfrac{\lambda}{1-p\tau
},\frac{\gamma_{1}^{2}}{p}\right)  ,\text{ as }n\rightarrow\infty,
\label{a-n-efg}%
\end{equation}
provided that $\sqrt{k}A_{1}\left(  h\right)  \rightarrow\lambda.$ For their
part, \cite{WW2014} proposed two estimators which incidentally can be derived,
through a slight modification, from the one we will define later on. They
proved their consistency (but not the asymptotic normality) by using similar
assumptions as those of \cite{EnFG08}. These estimators are defined by
\[
\widehat{\gamma}_{1}^{\left(  W1\right)  }:=\frac{1}{n\left(  1-F_{n}^{\left(
KM\right)  }\left(  Z_{n-k:n}\right)  \right)  }\sum_{i=1}^{k}\frac
{\delta_{\left[  i:n-i+1\right]  }}{1-G_{n}^{\left(  KM\right)  }\left(
Z_{n-i+1:n}-\right)  }\log\frac{Z_{n-i+1:n}}{Z_{n-k:n}}%
\]
and%
\[
\widehat{\gamma}_{1}^{\left(  W2\right)  }:=\frac{1}{n\left(  1-F_{n}^{\left(
KM\right)  }\left(  Z_{n-k:n}\right)  \right)  }\sum_{i=1}^{k}\frac{1}%
{1-G_{n}^{\left(  KM\right)  }\left(  Z_{n-i+1:n}-\right)  }i\log
\frac{Z_{n-i+1:n}}{Z_{n-i:n}},
\]
where, for $z<Z_{n:n},$%
\begin{equation}
F_{n}^{\left(  KM\right)  }\left(  z\right)  :=1-%
%TCIMACRO{\dprod \limits_{i:Z_{i:n}\leq z}}%
%BeginExpansion
{\displaystyle\prod\limits_{i:Z_{i:n}\leq z}}
%EndExpansion
\left(  \dfrac{n-i}{n-i+1}\right)  ^{\delta_{\left[  i:n\right]  }}\text{ and
}G_{n}^{\left(  KM\right)  }\left(  z\right)  :=1-%
%TCIMACRO{\dprod \limits_{i:Z_{i:n}\leq z}}%
%BeginExpansion
{\displaystyle\prod\limits_{i:Z_{i:n}\leq z}}
%EndExpansion
\left(  \dfrac{n-i}{n-i+1}\right)  ^{1-\delta_{\left[  i:n\right]  }},
\label{KM}%
\end{equation}
are Kaplan-Meier estimators of $F$ and $G$ respectively. Thereafter, we will
see that the assumptions under which we establish the asymptotic normality of
our estimator are lighter and more familiar in the extreme value context. We
start the definition of our estimator by noting that, from Theorem 1.2.2 in
\cite{deHF06}, the first order condition of regular variation (\ref{R-F})
implies that $\lim_{t\rightarrow\infty}\int_{1}^{\infty}x^{-1}\overline
{F}(tx)/\overline{F}\left(  t\right)  dt=\gamma_{1},$which, after an
integration by parts, may be rewritten into%
\begin{equation}
\lim_{t\rightarrow\infty}\frac{1}{\overline{F}\left(  t\right)  }\int
_{t}^{\infty}\log\left(  \frac{x}{t}\right)  dF(x)=\gamma_{1}. \label{log}%
\end{equation}
By replacing $F$ by $F_{n}$ and letting $t=Z_{n-k:n},$ we obtain%
\begin{equation}
\widehat{\gamma}_{1}=\frac{1}{\overline{F}_{n}\left(  Z_{n-k:n}\right)  }%
\int_{Z_{n-k:n}}^{\infty}\log\left(  x/Z_{n-k:n}\right)  dF_{n}\left(
x\right)  , \label{gchap}%
\end{equation}
as an estimator of $\gamma_{1}.$ Before we proceed with the construction of
$\widehat{\gamma}_{1},$ we need to define a function $H^{\left(  0\right)  },$
that is somewhat similar to $H^{\left(  1\right)  },$ and its empirical
version by $H^{\left(  0\right)  }\left(  z\right)  :=\mathbf{P}\left(
Z_{1}\leq z,\delta_{1}=0\right)  =\int_{0}^{z}\overline{F}\left(  y\right)
dG\left(  y\right)  $ and $H_{n}^{\left(  0\right)  }\left(  z\right)
:=n^{-1}\sum_{i=1}^{n}\left(  1-\delta_{i}\right)  \mathbf{1}\left(  Z_{i}\leq
z\right)  ,$ $z\geq0,$ respectively. Now, note that $dF_{n}\left(  z\right)
=\exp\left\{  \int_{0}^{z}dH_{n}^{\left(  0\right)  }\left(  v\right)
/\overline{H}_{n}\left(  v-\right)  \right\}  dH_{n}^{\left(  1\right)
}\left(  z\right)  ,$ then we have%
\[
\widehat{\gamma}_{1}=\frac{\dfrac{1}{n}%
%TCIMACRO{\dsum \limits_{i=n-k+1}^{n}}%
%BeginExpansion
{\displaystyle\sum\limits_{i=n-k+1}^{n}}
%EndExpansion
\delta_{\left[  i:n\right]  }\log\dfrac{Z_{i:n}}{Z_{n-k:n}}\exp\left\{
%TCIMACRO{\dsum \limits_{j=1}^{i}}%
%BeginExpansion
{\displaystyle\sum\limits_{j=1}^{i}}
%EndExpansion
\dfrac{1-\delta_{\left[  j:n\right]  }}{n-j+1}\right\}  }{%
%TCIMACRO{\dprod \limits_{j=1}^{n-k}}%
%BeginExpansion
{\displaystyle\prod\limits_{j=1}^{n-k}}
%EndExpansion
\exp\left\{  -\dfrac{\delta_{\left[  j:n\right]  }}{n-j+1}\right\}  },
\]
which may be rewritten into%
\[
\frac{\dfrac{1}{n}%
%TCIMACRO{\dsum \limits_{i=n-k+1}^{n}}%
%BeginExpansion
{\displaystyle\sum\limits_{i=n-k+1}^{n}}
%EndExpansion
\delta_{\left[  i:n\right]  }%
%TCIMACRO{\dprod \limits_{j=1}^{i}}%
%BeginExpansion
{\displaystyle\prod\limits_{j=1}^{i}}
%EndExpansion
\exp\left\{  \dfrac{1}{n-j+1}\right\}
%TCIMACRO{\dprod \limits_{j=1}^{i}}%
%BeginExpansion
{\displaystyle\prod\limits_{j=1}^{i}}
%EndExpansion
\exp\left\{  -\dfrac{\delta_{\left[  j:n\right]  }}{n-j+1}\right\}  \log
\dfrac{Z_{i:n}}{Z_{n-k:n}}}{%
%TCIMACRO{\dprod \limits_{j=1}^{n-k}}%
%BeginExpansion
{\displaystyle\prod\limits_{j=1}^{n-k}}
%EndExpansion
\exp\left\{  -\dfrac{\delta_{\left[  j:n\right]  }}{n-j+1}\right\}  },
\]
which in turn simplifies to%
\[
\dfrac{1}{n}%
%TCIMACRO{\dsum _{i=n-k+1}^{n}}%
%BeginExpansion
{\displaystyle\sum_{i=n-k+1}^{n}}
%EndExpansion
\delta_{\left[  i:n\right]  }%
%TCIMACRO{\dprod \limits_{j=1}^{i}}%
%BeginExpansion
{\displaystyle\prod\limits_{j=1}^{i}}
%EndExpansion
\exp\left\{  \dfrac{1}{n-j+1}\right\}
%TCIMACRO{\dprod \limits_{j=n-k+1}^{i}}%
%BeginExpansion
{\displaystyle\prod\limits_{j=n-k+1}^{i}}
%EndExpansion
\exp\left\{  -\dfrac{\delta_{\left[  j:n\right]  }}{n-j+1}\right\}  \log
\dfrac{Z_{i:n}}{Z_{n-k:n}}.
\]
By changing $i$ into $n-i+1$ and $j$ into $n-j+1,$ with the necessary
modifications, we end up with the following explicit formula for our new
estimator of the EVI $\gamma_{1}:$%
\[
\widehat{\gamma}_{1}=\sum_{i=1}^{k}a_{i,n}\log\dfrac{Z_{n-i+1:n}}{Z_{n-k:n}},
\]
where
\begin{equation}
a_{i,n}:=n^{-1}\delta_{\left[  n-i+1:n\right]  }%
%TCIMACRO{\dprod \limits_{j=i}^{n}}%
%BeginExpansion
{\displaystyle\prod\limits_{j=i}^{n}}
%EndExpansion
\exp\left\{  1/j\right\}
%TCIMACRO{\dprod \limits_{j=i}^{k}}%
%BeginExpansion
{\displaystyle\prod\limits_{j=i}^{k}}
%EndExpansion
\exp\left\{  -\delta_{\left[  n-j+1:n\right]  }/j\right\}  . \label{ai}%
\end{equation}
Note that for uncensored data, we have $X=Z$ and all the $\delta^{\prime}s$
are equal to $1,$ therefore $a_{i,n}\approx k^{-1},$ for $i=1,...,k.$ Indeed,
\[
a_{i,n}=\dfrac{1}{n}%
%TCIMACRO{\dprod \limits_{j=k+1}^{n}}%
%BeginExpansion
{\displaystyle\prod\limits_{j=k+1}^{n}}
%EndExpansion
\exp\left\{  \dfrac{1}{j}\right\}  \approx\dfrac{1}{n}%
%TCIMACRO{\dprod \limits_{j=k+1}^{n}}%
%BeginExpansion
{\displaystyle\prod\limits_{j=k+1}^{n}}
%EndExpansion
\left\{  1+\dfrac{1}{j}\right\}  =\dfrac{n+1}{n}\dfrac{1}{k+1}\approx\dfrac
{1}{k},\text{ as }n\rightarrow\infty,
\]
which leads to the formula of the famous Hill estimator of the EVI
\citep[][]{Hill75}. The consistency and asymptotic normality of $\widehat
{\gamma}_{1}$ are established in the following theorem.

\begin{theorem}
\label{Theorem2}Let $F$ and $G$ be two cdf's with regularly varying tails
(\ref{R-F}) such that $\gamma_{1}<\gamma_{2}.$ Let $k=k_{n}$ be an integer
sequence such that $k\rightarrow\infty$ and $k/n\rightarrow0,$ then
$\widehat{\gamma}_{1}\rightarrow\gamma_{1}$ in probability, as $n\rightarrow
\infty.$ Assume further that the second-order condition of regular variation
(\ref{second-order}) holds, then for all large $n$%
\begin{align*}
&  \sqrt{k}\left(  \widehat{\gamma}_{1}-\gamma_{1}\right)  -\frac{\sqrt
{k}A_{1}\left(  h\right)  }{1-\tau}\\
&  =\gamma\sqrt{\frac{n}{k}}\int_{0}^{1}s^{-q-1}\left(  \mathbf{W}%
_{n,2}\left(  \frac{k}{n}s\right)  +\left(  1-\frac{q}{p}\right)
\mathbf{W}_{n,1}\left(  \frac{k}{n}s\right)  \right)  ds-\gamma_{1}\sqrt
{\frac{n}{k}}\mathbf{W}_{n,1}\left(  \frac{k}{n}\right)  +o_{\mathbf{p}%
}\left(  1\right)  ,
\end{align*}
where $\mathbf{W}_{n,1}$ and $\mathbf{W}_{n,2}$ are those defined in Theorem
\ref{Theorem1}. If in addition $\sqrt{k}A_{1}\left(  h\right)  \rightarrow
\lambda,$ then
\begin{equation}
\sqrt{k}\left(  \widehat{\gamma}_{1}-\gamma_{1}\right)  \overset{\mathcal{D}%
}{\rightarrow}\mathcal{N}\left(  \dfrac{\lambda}{1-\tau},\dfrac{p}{2p-1}%
\gamma_{1}^{2}\right)  ,\text{ as }n\rightarrow\infty. \label{bias}%
\end{equation}

\end{theorem}

\begin{remark}
\label{Rmk3}We clearly see that, when there is no censoring $\left(
\text{i.e. }p=1-q=1\right)  ,$ the Gaussian approximation and the limiting
distribution above perfectly agree with those of Hill's estimator \citep[see, e.g.,][page 76]{deHF06}.
\end{remark}

\begin{remark}
\label{Rmk4}Since $0<p<1$ and $\tau\leq0,$ then $1-\tau\geq1-p\tau>0.$ This
implies that the absolute value of the asymptotic bias of $\widehat{\gamma
}_{1}$ is smaller than or equal to that of $\widehat{\gamma}_{1}^{\left(
EFG\right)  }$ (see $\left(  \ref{bias}\right)  $ and $\left(  \ref{a-n-efg}%
\right)  ).$ In other words, the new estimator $\widehat{\gamma}_{1}$ is of
reduced bias compared to $\widehat{\gamma}_{1}^{\left(  EFG\right)  }.$
However, for any $1/2<p<1,$ we have $p/\left(  2p-1\right)  >1/p$ meaning that
the asymptotic variance of $\widehat{\gamma}_{1}$ is greater than that of
$\widehat{\gamma}_{1}^{\left(  EFG\right)  }.$ This seems logical, because it
is rare to reduce the asymptotic bias of an estimator without increasing its
asymptotic variance. This is the price to pay, see for instance \cite{peng98}
and \cite{BBWG-16}. We also note that in the complete data case, both
asymptotic biases coincide with that of Hill's estimator and so do the variances.
\end{remark}

\section{\textbf{Simulation study\label{sec4}}}

\noindent We mentioned in the introduction that \cite{WW2014} showed, by
simulation, that their estimators outperform the adapted Hill estimator
introduced by \cite{EnFG08}. Therefore, this study is intended for comparing
our estimator $\widehat{\gamma}_{1}$ to $\widehat{\gamma}_{1}^{\left(
W1\right)  }$and $\widehat{\gamma}_{1}^{\left(  W2\right)  },$ with respect to
biases and MSE's. It is carried out through two sets of censored and censoring
data, both drawn from the following Burr models: $\overline{F}\left(
x\right)  =\left(  1+x^{1/\eta_{1}}\right)  ^{-\eta_{1}/\gamma_{1}}$ and
$\overline{G}\left(  x\right)  =\left(  1+x^{1/\eta_{2}}\right)  ^{-\eta
_{2}/\gamma_{2}},$ $x\geq0,$ where $\eta_{1},\eta_{2},\gamma_{1},\gamma
_{2}>0.$ We fix $\eta_{1}=\eta_{2}=1/4$ and choose the values $0.2$ and $0.8$
for $\gamma_{1}.$ For the proportion of really observed extreme values, we
take $p=0.55,$ $0.70$ and $0.90$ for strong, moderate and weak censoring
respectively. For each couple $\left(  \gamma_{1},p\right)  ,$ we solve the
equation $p=\gamma_{2}/(\gamma_{1}+\gamma_{2})$ to get the pertaining
$\gamma_{2}$-value. We generate $2000$ independent replicates of size $n=300$
then $n=1000$ from both samples $\left(  X_{1},...,X_{n}\right)  $ and
$\left(  Y_{1},...,Y_{n}\right)  .$ Our overall results are taken as the
empirical means of the results obtained through all repetitions. We plot the
absolute biases and the MSE's as functions of the number $k$ of upper order
statistics used in the computation of the three estimators. The simulation
results are illustrated, for the respective $p$-values, in Figures
\ref{F300-55}, \ref{F300-70} and \ref{F300-90} for $n=300$ and in Figures
\ref{F1000-55}, \ref{F1000-70} and \ref{F1000-90} for $n=1000.$ On the light
of all the figures, it is clear that our estimator performs better than the
other two. Indeed, in (almost) each case, the minima of the absolute bias and
MSE of $\widehat{\gamma}_{1}$ are less than those of $\widehat{\gamma}%
_{1}^{\left(  W1\right)  }$and $\widehat{\gamma}_{1}^{\left(  W2\right)  }.$
In addition, our estimator reaches its minima a long way before $\widehat
{\gamma}_{1}^{\left(  W1\right)  }$and $\widehat{\gamma}_{1}^{\left(
W2\right)  }$ do, meaning that the number $k$ of extremes needed for the last
two estimators to be accurate is much larger than those needed for
$\widehat{\gamma}_{1}.$ In other words, the cost of our estimator in terms of
upper order statistics is very low compared to that of Worms and Worms estimators.

\begin{figure}[ptb]
\centering
\includegraphics[width=0.7\linewidth]{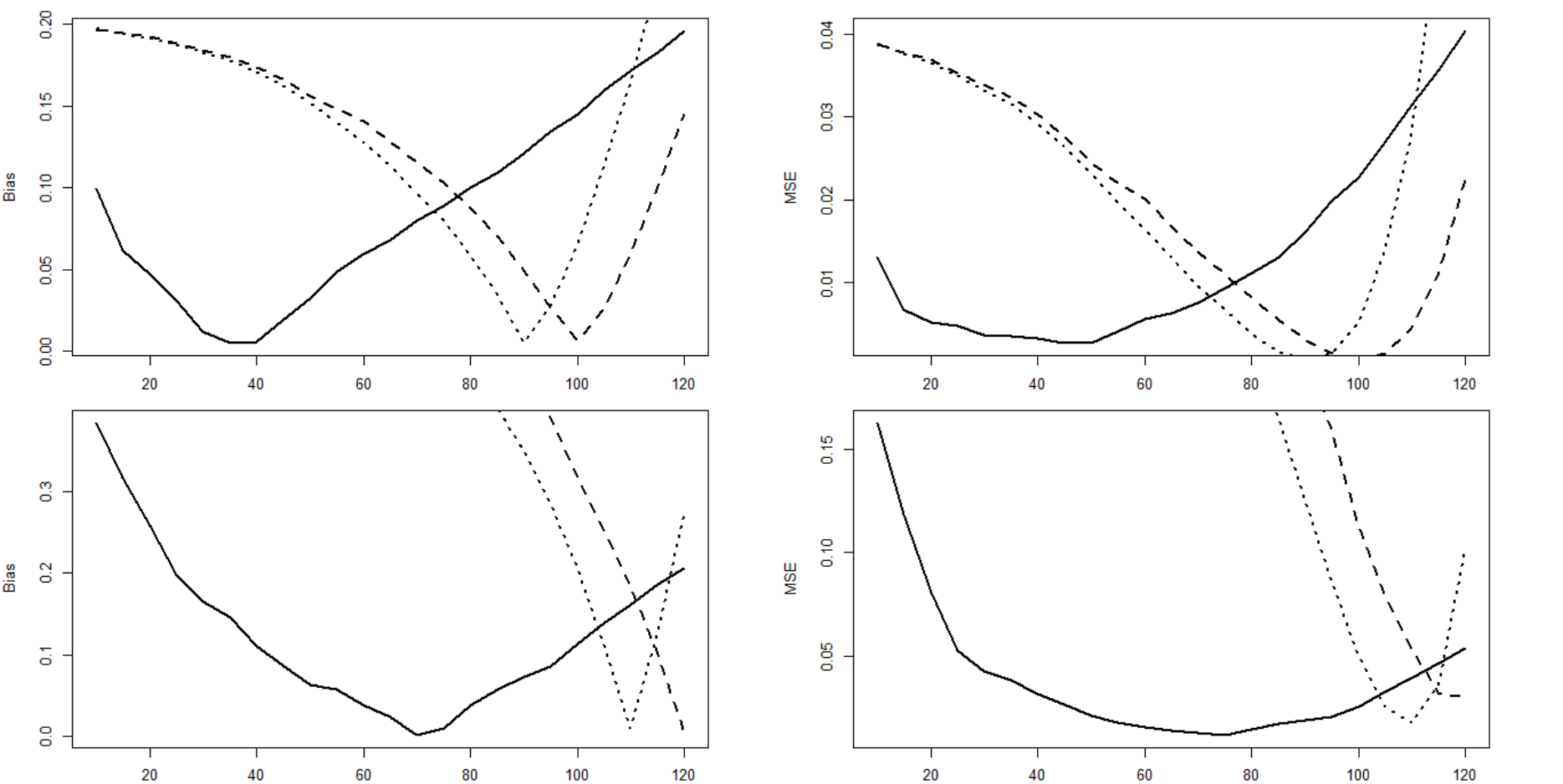}\caption{Bias (left panel) and
mse (right panel) of the estimators $\widehat{\gamma}_{1}$ (solid line),
$\widehat{\gamma}_{1}^{\left(  W1\right)  }$ (dashed line) and $\widehat
{\gamma}_{1}^{\left(  W2\right)  }$ (dotted line) of the tail index
$\gamma_{1}=0.2$ (top) and $\gamma_{1}=0.8$ (bottom) of strongly
right-censored Burr model, based on $2000$ samples of size $300.$}%
\label{F300-55}%
\end{figure}

\begin{figure}[ptb]
\centering
\includegraphics[width=0.7\linewidth]{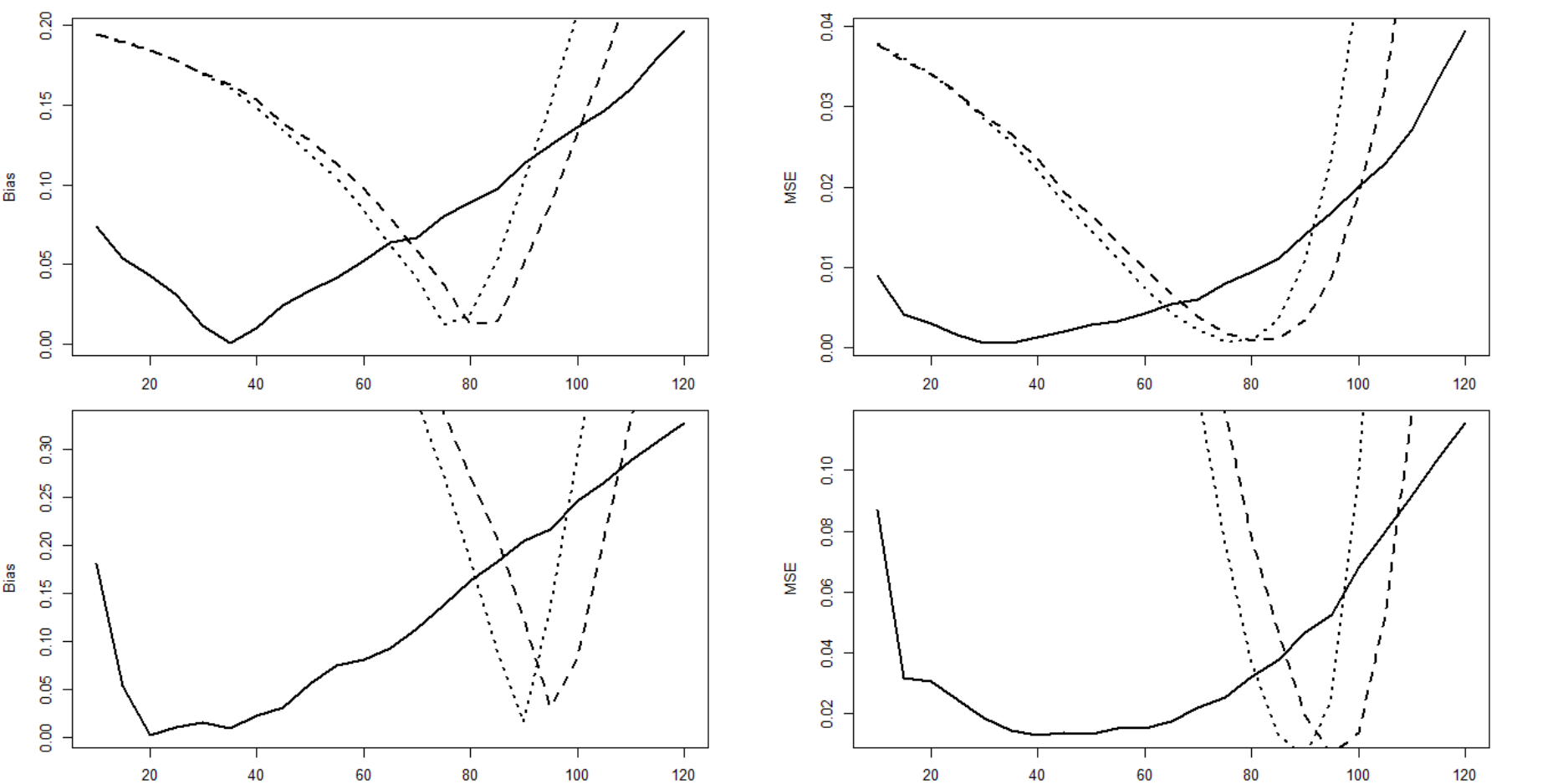}\caption{Bias (left panel) and
mse (right panel) of the estimators $\widehat{\gamma}_{1}$ (solid line),
$\widehat{\gamma}_{1}^{\left(  W1\right)  }$ (dashed line) and $\widehat
{\gamma}_{1}^{\left(  W2\right)  }$ (dotted line) of the tail index
$\gamma_{1}=0.2$ (top) and $\gamma_{1}=0.8$ (bottom) of moderately
right-censored Burr model, based on $2000$ samples of size $300.$}%
\label{F300-70}%
\end{figure}

\begin{figure}[ptb]
\centering
\includegraphics[width=0.7\linewidth]{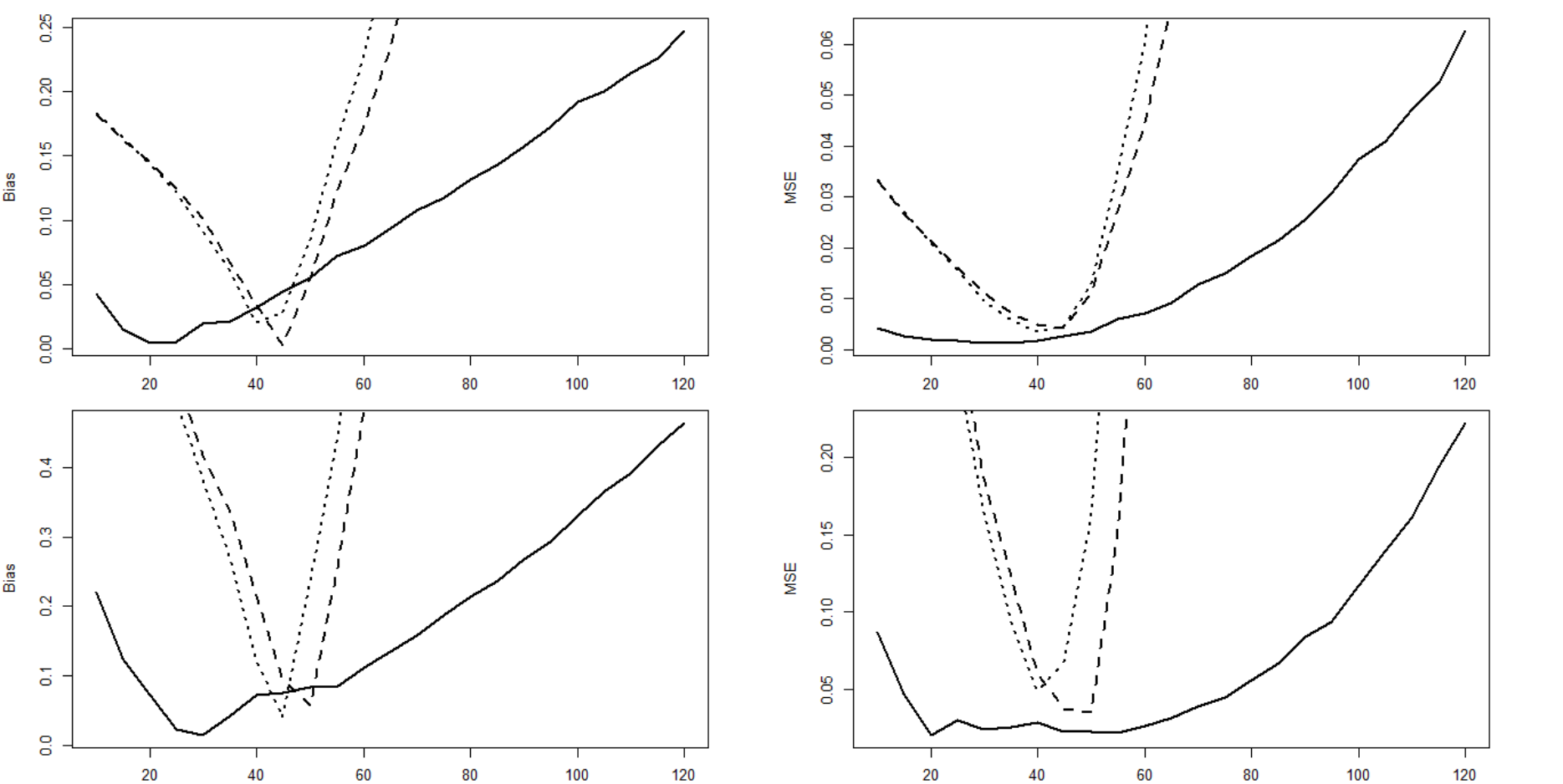}\caption{Bias (left panel) and
mse (right panel) of the estimators $\widehat{\gamma}_{1}$ (solid line),
$\widehat{\gamma}_{1}^{\left(  W1\right)  }$ (dashed line) and $\widehat
{\gamma}_{1}^{\left(  W2\right)  }$ (dotted line) of the tail index
$\gamma_{1}=0.2$ (top) and $\gamma_{1}=0.8$ (bottom) of weakly right-censored
Burr model, based on $2000$ samples of size $300.$}%
\label{F300-90}%
\end{figure}

\begin{figure}[ptb]
\centering
\includegraphics[width=0.7\linewidth]{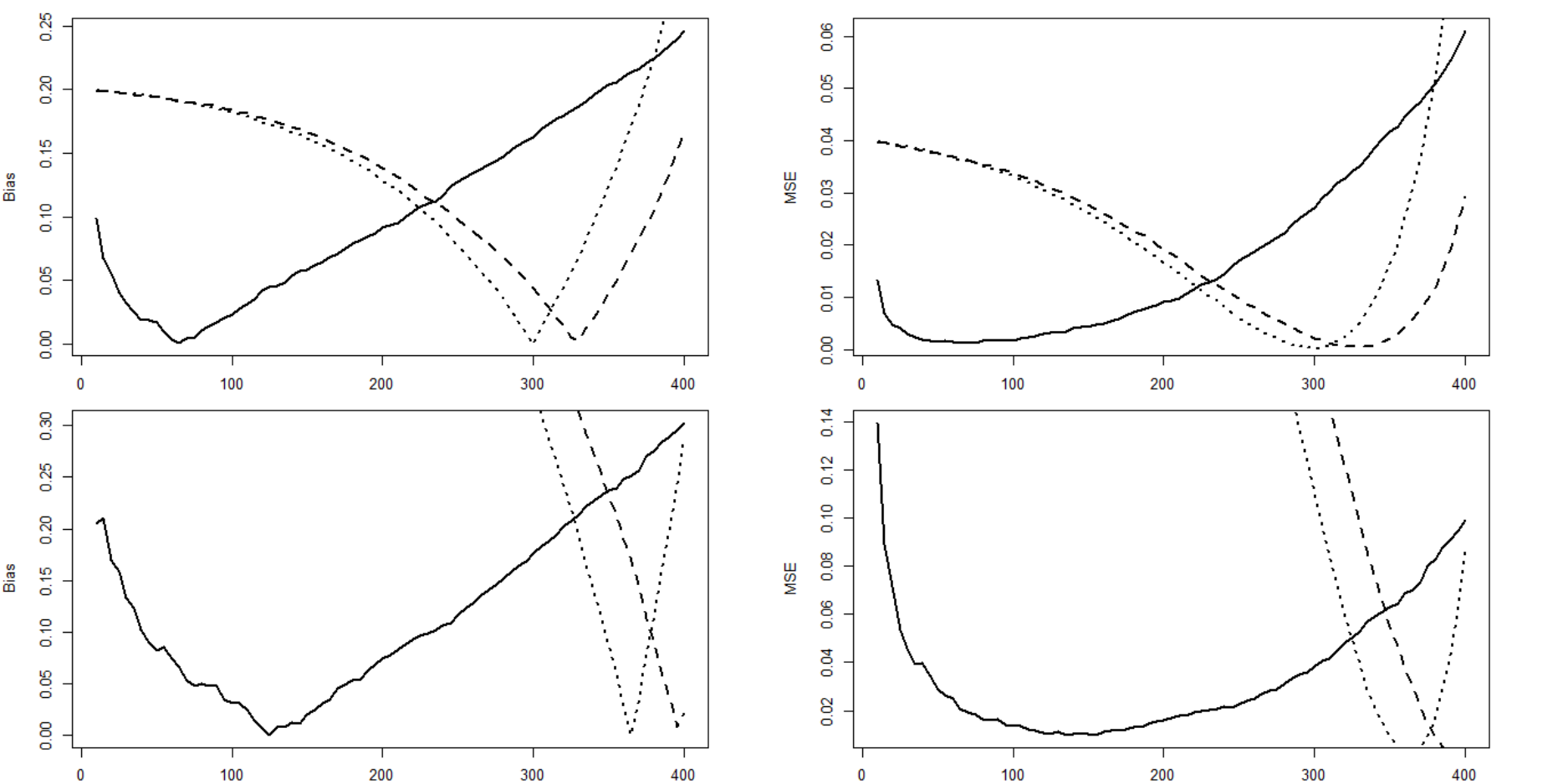}\caption{Bias (left panel) and
mse (right panel) of the estimators $\widehat{\gamma}_{1}$ (solid line),
$\widehat{\gamma}_{1}^{\left(  W1\right)  }$ (dashed line) and $\widehat
{\gamma}_{1}^{\left(  W2\right)  }$ (dotted line) of the tail index
$\gamma_{1}=0.2$ (top) and $\gamma_{1}=0.8$ (bottom) of strongly
right-censored Burr model, based on $2000$ samples of size $1000.$}%
\label{F1000-55}%
\end{figure}

\begin{figure}[ptb]
\centering
\includegraphics[width=0.7\linewidth]{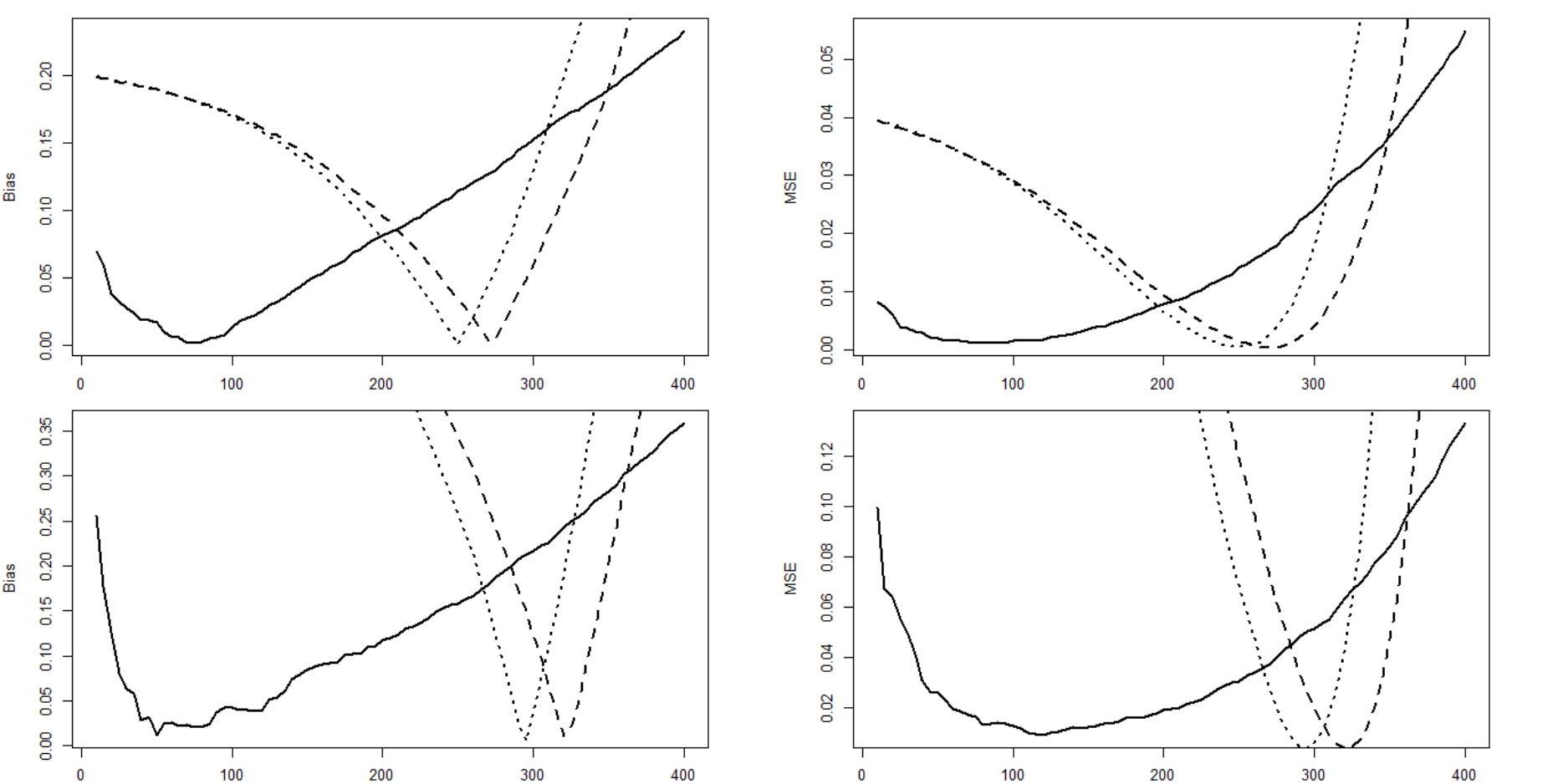}\caption{Bias (left panel) and
mse (right panel) of the estimators $\widehat{\gamma}_{1}$ (solid line),
$\widehat{\gamma}_{1}^{\left(  W1\right)  }$ (dashed line) and $\widehat
{\gamma}_{1}^{\left(  W2\right)  }$ (dotted line) of the tail index
$\gamma_{1}=0.2$ (top) and $\gamma_{1}=0.8$ (bottom) of moderately
right-censored Burr model, based on $2000$ samples of size $1000.$}%
\label{F1000-70}%
\end{figure}

\begin{figure}[ptb]
\centering
\includegraphics[width=0.7\linewidth]{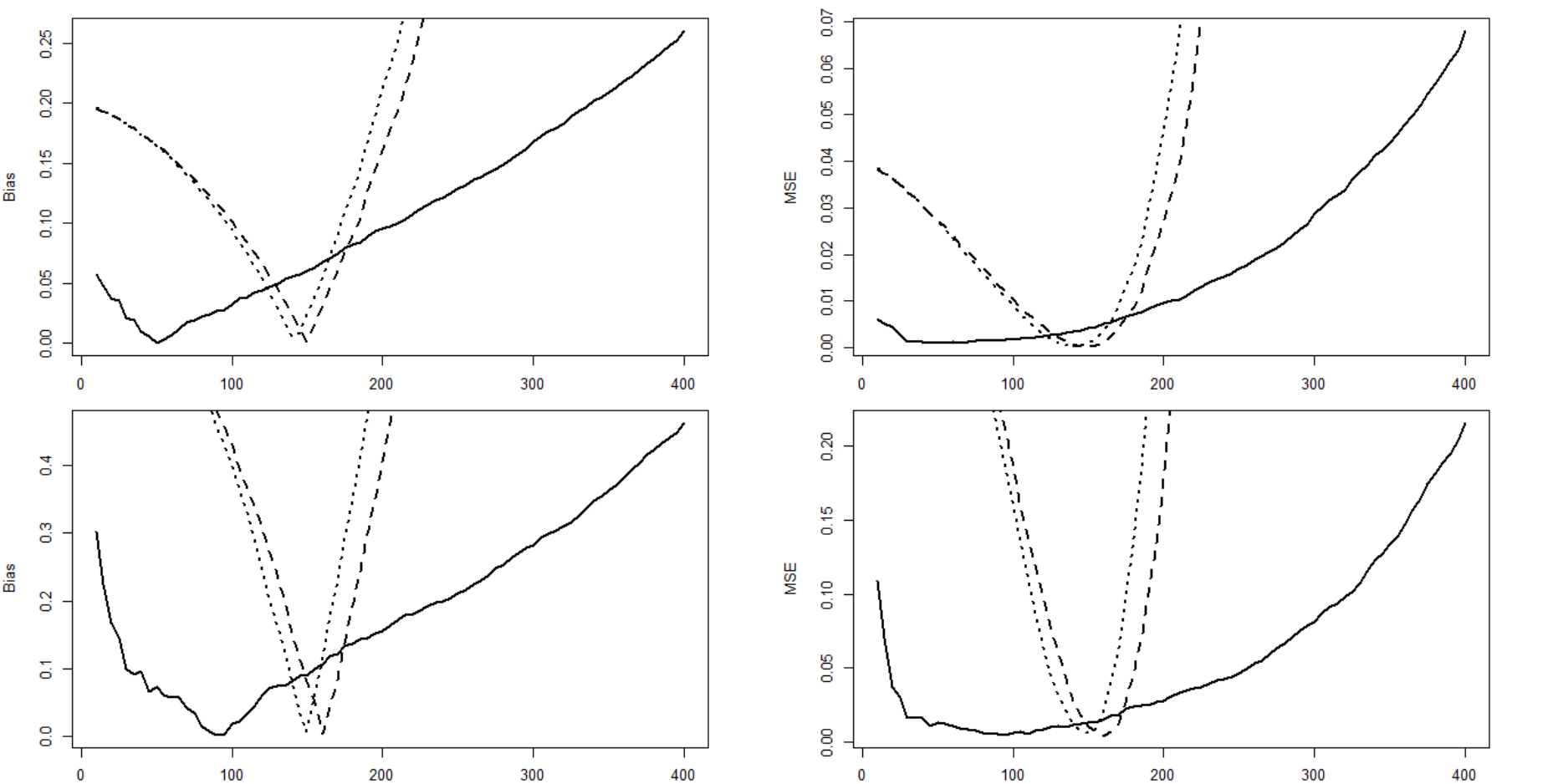}\caption{Bias (left panel) and
mse (right panel) of the estimators $\widehat{\gamma}_{1}$ (solid line),
$\widehat{\gamma}_{1}^{\left(  W1\right)  }$ (dashed line) and $\widehat
{\gamma}_{1}^{\left(  W2\right)  }$ (dotted line) of the tail index
$\gamma_{1}=0.2$ (top) and $\gamma_{1}=0.8$ (bottom) of weakly right-censored
Burr model, based on $2000$ samples of size $1000.$}%
\label{F1000-90}%
\end{figure}

\section{Proofs\textbf{\label{sec5}}}

\noindent In the sequel, we will use the following two empirical processes
$\sqrt{n}\left(  \overline{H}_{n}^{\left(  j\right)  }\left(  z\right)
-\overline{H}^{\left(  j\right)  }\left(  z\right)  \right)  ,$ $j=0,1;$
$z>0,$ which may be represented, almost surely, by two uniform empirical ones.
Indeed, let us consider the independent and identically distributed (iid)
$(0,1)$-uniform rv's $U_{i}:=\delta_{i}H^{\left(  1\right)  }\left(
Z_{i}\right)  +\left(  1-\delta_{i}\right)  \left(  \theta+H^{\left(
0\right)  }\left(  Z_{i}\right)  \right)  ,$ $i=1,...,n,$ defined in
\citep[see][]{EnKo92}. The empirical cdf and the uniform empirical process
based upon these rv's are respectively denoted by%
\begin{equation}
\mathbb{U}_{n}\left(  s\right)  :=\#\left\{  i:1\leq i\leq n,\text{ }U_{i}\leq
s\right\}  /n\text{ and }\alpha_{n}\left(  s\right)  :=\sqrt{n}\left(
\mathbb{U}_{n}\left(  s\right)  -s\right)  ,\text{ }0\leq s\leq1.\label{ecdf}%
\end{equation}
\cite{DeEn96} state that almost surely (a.s.)%
\begin{equation}
H_{n}^{\left(  1\right)  }\left(  z\right)  =\mathbb{U}_{n}\left(  H^{\left(
1\right)  }\left(  z\right)  \right)  \text{ and }H_{n}^{\left(  0\right)
}\left(  z\right)  =\mathbb{U}_{n}\left(  H^{\left(  0\right)  }\left(
z\right)  +\theta\right)  -\mathbb{U}_{n}\left(  \theta\right)  ,\label{HN1}%
\end{equation}
for $0<H^{\left(  1\right)  }\left(  z\right)  <\theta$ and $0<H^{\left(
0\right)  }\left(  z\right)  <1-\theta.$ It is easy to verify that a.s.%
\begin{equation}
\sqrt{n}\left(  \overline{H}_{n}^{\left(  1\right)  }\left(  z\right)
-\overline{H}^{\left(  1\right)  }\left(  z\right)  \right)  =\alpha
_{n}\left(  \theta\right)  -\alpha_{n}\left(  \theta-\overline{H}^{\left(
1\right)  }\left(  z\right)  \right)  ,\text{ for }0<\overline{H}^{\left(
1\right)  }\left(  z\right)  <\theta,\label{rep-H1}%
\end{equation}
and%
\begin{equation}
\sqrt{n}\left(  \overline{H}_{n}^{\left(  0\right)  }\left(  z\right)
-\overline{H}^{\left(  0\right)  }\left(  z\right)  \right)  =-\alpha
_{n}\left(  1-\overline{H}^{\left(  0\right)  }\left(  z\right)  \right)
,\text{ for }0<\overline{H}^{\left(  0\right)  }\left(  z\right)
<1-\theta.\label{rep-H0}%
\end{equation}
Our methodology strongly relies on the well-known Gaussian approximation,
given by \cite{CsCsHM86} in Corollary 2.1. It says that on the probability
space $\left(  \Omega,\mathcal{A},\mathbf{P}\right)  ,$ there exists a
sequence of Brownian bridges $\left\{  B_{n}\left(  s\right)  ;\text{ }0\leq
s\leq1\right\}  $ such that, for every $1/4<\eta<1/2,$ we have%
\begin{equation}
\sup_{1/n\leq s\leq1-1/n}\frac{\left\vert \alpha_{n}\left(  s\right)
-B_{n}\left(  s\right)  \right\vert }{\left[  s\left(  1-s\right)  \right]
^{\eta}}=O_{\mathbf{p}}\left(  n^{\eta-1/2}\right)  .\label{approx2}%
\end{equation}
For the increments $\alpha_{n}\left(  \theta\right)  -\alpha_{n}\left(
\theta-s\right)  ,$ we will need an approximation of the same type as $\left(
\ref{approx2}\right)  .$ Following similar arguments, mutatis mutandis, as
those used in the proofs of assertions (2.2) of Theorem 2.1 and (2.8) of
Theorem 2.2 in \cite{CsCsHM86}, we may show that, for every $0<\theta<1,$ we
have%
\begin{equation}
\sup_{1/n\leq s\leq\theta}\frac{\left\vert \left\{  \alpha_{n}\left(
\theta\right)  -\alpha_{n}\left(  \theta-s\right)  \right\}  -\left\{
B_{n}\left(  \theta\right)  -B_{n}\left(  \theta-s\right)  \right\}
\right\vert }{s^{\eta}}=O_{\mathbf{p}}\left(  n^{\eta-1/2}\right)
.\label{approx1}%
\end{equation}

\subsection{Proof of Theorem \ref{Theorem1}}

\noindent The proof is carried out through two steps. First, we asymptotically
represent $D_{n}\left(  x\right)  $ is terms of the empirical cdf $H_{n}%
$\textbf{ }and the empirical sub-distributions functions $H_{n}^{\left(
0\right)  }$ and\textbf{ }$H_{n}^{\left(  1\right)  }.$ Second, we rewrite it
as a functional of the following two processes%
\begin{equation}
\beta_{n}\left(  w\right)  :=\sqrt{\frac{n}{k}}\left\{  \alpha_{n}\left(
\theta\right)  -\alpha_{n}\left(  \theta-\overline{H}^{\left(  1\right)
}\left(  wZ_{n-k:n}\right)  \right)  \right\}  ,\text{ for }0<\overline
{H}^{\left(  1\right)  }\left(  wZ_{n-k:n}\right)  <\theta, \label{betan}%
\end{equation}
and%
\begin{equation}
\widetilde{\beta}_{n}\left(  w\right)  :=-\sqrt{\frac{n}{k}}\alpha_{n}\left(
1-\overline{H}^{\left(  0\right)  }\left(  wZ_{n-k:n}\right)  \right)  ,\text{
for }0<\overline{H}^{\left(  0\right)  }\left(  wZ_{n-k:n}\right)  <1-\theta,
\label{beta-tild}%
\end{equation}
in order to apply the weak approximations $\left(  \ref{approx2}\right)  $ and
$\left(  \ref{approx1}\right)  .$ We begin, as in the proof of Theorem 2.1 in
\cite{BMN2016}, by decomposing $k^{-1/2}D_{n}\left(  x\right)  $ into the sum
of%
\[
\mathbb{M}_{n1}\left(  x\right)  :=\dfrac{\overline{F}_{n}\left(
xZ_{n-k:n}\right)  -\overline{F}\left(  xZ_{n-k:n}\right)  }{\overline
{F}\left(  Z_{n-k:n}\right)  },
\]%
\[
\mathbb{M}_{n2}\left(  x\right)  :=\left(  \frac{\overline{F}\left(
Z_{n-k:n}\right)  }{\overline{F}_{n}\left(  Z_{n-k:n}\right)  }-1\right)
\frac{\overline{F}_{n}\left(  xZ_{n-k:n}\right)  -\overline{F}\left(
xZ_{n-k:n}\right)  }{\overline{F}\left(  Z_{n-k:n}\right)  },
\]%
\[
\mathbb{M}_{n3}\left(  x\right)  :=-\frac{\overline{F}\left(  xZ_{n-k:n}%
\right)  }{\overline{F}\left(  Z_{n-k:n}\right)  }\frac{\overline{F}%
_{n}\left(  Z_{n-k:n}\right)  -\overline{F}\left(  Z_{n-k:n}\right)
}{\overline{F}_{n}\left(  Z_{n-k:n}\right)  },
\]
and%
\[
\mathbb{M}_{n4}\left(  x\right)  :=\left(  \dfrac{\overline{F}\left(
xZ_{n-k:n}\right)  }{\overline{F}\left(  Z_{n-k:n}\right)  }-\dfrac
{\overline{F}\left(  xh\right)  }{\overline{F}\left(  h\right)  }\right)
+\left(  \dfrac{\overline{F}\left(  xh\right)  }{\overline{F}\left(  h\right)
}-x^{-1/\gamma_{1}}\right)  .
\]
In the following two subsections, we show that, uniformly on $x\geq p^{\gamma
},$ for any $1/4<\eta<p/2$ and small $0<\epsilon_{0}<1,$ we have%
\begin{equation}
\sum_{i=1}^{3}\sqrt{k}\mathbb{M}_{ni}\left(  x\right)  -J_{n}\left(  x\right)
=o_{\mathbf{p}}\left(  x^{\left(  2\eta-p\right)  /\gamma\pm\epsilon_{0}%
}\right)  , \label{aproxima-MI}%
\end{equation}
where $J_{n}\left(  x\right)  $ is the Gaussian process defined in Theorem
\ref{Theorem1}. It is worth mentioning that both $\mathbb{M}_{n2}\left(
x\right)  $ and $\mathbb{M}_{n3}\left(  x\right)  $ may be rewritten in terms
of $\mathbb{M}_{n1}\left(  x\right)  .$ Indeed, it is readily checked that%
\[
\mathbb{M}_{n2}\left(  x\right)  =\left(  \frac{\overline{F}\left(
Z_{n-k:n}\right)  }{\overline{F}_{n}\left(  Z_{n-k:n}\right)  }-1\right)
\mathbb{M}_{n1}\left(  x\right)  \text{ and }\mathbb{M}_{n3}\left(  x\right)
=-\dfrac{\overline{F}\left(  xZ_{n-k:n}\right)  }{\overline{F}_{n}\left(
Z_{n-k:n}\right)  }\mathbb{M}_{n1}\left(  1\right)  .
\]
Then, we only focus on the weak approximation $\sqrt{k}\mathbb{M}_{n1}\left(
x\right)  $ which will lead to that of $\sqrt{k}\mathbb{M}_{n3}\left(
x\right)  $ and the asymptotic negligibility (in probability) of $\sqrt
{k}\mathbb{M}_{n2}\left(  x\right)  .$ Finally, the approximation of $\sqrt
{k}M_{n4}\left(  x\right)  $ will result in the asymptotic bias. In
conclusion, we might say that representing $D_{n}\left(  x\right)  $ amounts
to representing $\mathbb{M}_{n1}\left(  x\right)  .$

\subsubsection{\textbf{Representation of }$\mathbb{M}_{n1}\left(  x\right)
$\textbf{ in terms of }$H_{n},$\textbf{ }$H_{n}^{\left(  0\right)  }$\textbf{
and }$H_{n}^{\left(  1\right)  }$}

We show that, for all large $n$ and $x\geq p^{\gamma}$%
\begin{equation}
\mathbb{M}_{n1}\left(  x\right)  =\mathbb{T}_{n1}\left(  x\right)
+\mathbb{T}_{n2}\left(  x\right)  +\mathbb{T}_{n3}\left(  x\right)
+R_{n}\left(  x\right)  , \label{MN1-decompos}%
\end{equation}
where%
\[
\mathbb{T}_{n1}\left(  x\right)  :=\int_{xZ_{n-k:n}}^{\infty}\frac
{\overline{F}\left(  w\right)  }{\overline{F}\left(  Z_{n-k:n}\right)  }%
\frac{d\left(  H_{n}^{\left(  1\right)  }\left(  w\right)  -H^{\left(
1\right)  }\left(  w\right)  \right)  }{\overline{H}\left(  w\right)  },
\]%
\[
\mathbb{T}_{n2}\left(  x\right)  :=\int_{xZ_{n-k:n}}^{\infty}\frac
{\overline{F}\left(  w\right)  }{\overline{F}\left(  Z_{n-k:n}\right)
}\left\{  \int_{0}^{w}\frac{d\left(  H_{n}^{\left(  0\right)  }\left(
v\right)  -H^{\left(  0\right)  }\left(  v\right)  \right)  }{\overline
{H}\left(  v\right)  }\right\}  dH^{\left(  1\right)  }\left(  w\right)  ,
\]%
\[
\mathbb{T}_{n3}\left(  x\right)  :=\int_{xZ_{n-k:n}}^{\infty}\frac
{\overline{F}\left(  w\right)  }{\overline{F}\left(  Z_{n-k:n}\right)
}\left\{  \int_{0}^{w}\frac{H_{n}\left(  v\right)  -H\left(  v\right)
}{\overline{H}^{2}\left(  v\right)  }dH^{\left(  0\right)  }\left(  v\right)
\right\}  dH^{\left(  1\right)  }\left(  w\right)  ,
\]
and $R_{n}\left(  x\right)  :=O_{\mathbf{p}}\left(  k^{-2\eta\pm\epsilon_{0}%
}\right)  x^{\left(  2\eta-p\right)  /\gamma\pm\epsilon_{0}},$ for every
$1/4<\eta<p/2$ and small $0<\epsilon_{0}<1.$ For notational simplicity, we
write $b^{\pm\epsilon_{0}}:=\max\left(  b^{\epsilon_{0}},b^{-\epsilon_{0}%
}\right)  $ and without loss of generality, we attribute $\epsilon_{0}$ to any
constant times $\epsilon_{0}$ and $b^{\pm\epsilon_{0}}$ to any linear
combinations of $b^{\pm c_{1}\epsilon_{0}}$ and $b^{\pm c_{2}\epsilon_{0}}%
,\ $for every $c_{1},c_{2}>0.$ We begin by letting $\mathcal{\ell}\left(
w\right)  :=\overline{F}\left(  w\right)  /\overline{H}\left(  w\right)  $
which may be rewritten into $\exp\left\{  \int_{0}^{w}dH^{\left(  0\right)
}\left(  v\right)  /\overline{H}\left(  v\right)  \right\}  $ whose empirical
counterpart is equal to
\[
\mathcal{\ell}_{n}\left(  w\right)  :=\exp\left\{  \int_{0}^{w}dH_{n}^{\left(
0\right)  }\left(  v\right)  /\overline{H}_{n}\left(  v\right)  \right\}  .
\]
Observe that $\overline{F}\left(  xZ_{n-k:n}\right)  =\int_{xZ_{n-k:n}%
}^{\infty}\mathcal{\ell}\left(  w\right)  dH^{\left(  1\right)  }\left(
w\right)  ,$ then by replacing $\mathcal{\ell}$ and $H^{\left(  1\right)  }$
by $\mathcal{\ell}_{n}$ and $H_{n}^{\left(  1\right)  }$ respectively, we
obtain $\overline{F}_{n}\left(  xZ_{n-k:n}\right)  =\int_{xZ_{n-k:n}}^{\infty
}\mathcal{\ell}_{n}\left(  w\right)  dH_{n}^{\left(  1\right)  }\left(
w\right)  .$ Thus%
\begin{equation}
\mathbb{M}_{n1}\left(  x\right)  =\int_{xZ_{n-k:n}}^{\infty}\frac
{\mathcal{\ell}_{n}\left(  w\right)  }{\overline{F}\left(  Z_{n-k:n}\right)
}dH_{n}^{\left(  1\right)  }\left(  w\right)  -\int_{xZ_{n-k:n}}^{\infty}%
\frac{\mathcal{\ell}\left(  w\right)  }{\overline{F}\left(  Z_{n-k:n}\right)
}dH^{\left(  1\right)  }\left(  w\right)  . \label{M}%
\end{equation}
By applying Taylor's expansion, we may rewrite $\mathcal{\ell}_{n}\left(
w\right)  -\mathcal{\ell}\left(  w\right)  $ into%
\[
\mathcal{\ell}\left(  w\right)  \left\{  \int_{0}^{w}\frac{dH_{n}^{\left(
0\right)  }\left(  v\right)  }{\overline{H}_{n}\left(  v\right)  }-\int
_{0}^{w}\frac{dH^{\left(  0\right)  }\left(  v\right)  }{\overline{H}\left(
v\right)  }\right\}  +\frac{1}{2}\widetilde{\mathcal{\ell}}_{n}\left(
w\right)  \left\{  \int_{0}^{w}\frac{dH_{n}^{\left(  0\right)  }\left(
v\right)  }{\overline{H}_{n}\left(  v\right)  }-\int_{0}^{w}\frac{dH^{\left(
0\right)  }\left(  v\right)  }{\overline{H}\left(  v\right)  }\right\}  ^{2},
\]
where $\widetilde{\mathcal{\ell}}_{n}\left(  w\right)  $ is a stochastic
intermediate value lying between $\mathcal{\ell}\left(  w\right)  $ and
$\mathcal{\ell}_{n}\left(  w\right)  .$ This allows us to decompose
$\mathbb{M}_{n1}\left(  x\right)  ,$ in $\left(  \ref{M}\right)  ,$ into the
sum of%
\[
\mathbb{T}_{n1}^{\ast}\left(  x\right)  :=\int_{xZ_{n-k:n}}^{\infty}%
\frac{\mathcal{\ell}\left(  w\right)  }{\overline{F}\left(  Z_{n-k:n}\right)
}d\left(  H_{n}^{\left(  1\right)  }\left(  w\right)  -H^{\left(  1\right)
}\left(  w\right)  \right)  ,
\]%
\[
\mathbb{T}_{n2}^{\ast}\left(  x\right)  :=\int_{xZ_{n-k:n}}^{\infty}%
\frac{\mathcal{\ell}\left(  w\right)  }{\overline{F}\left(  Z_{n-k:n}\right)
}\left\{  \int_{0}^{w}\frac{d\left(  H_{n}^{\left(  0\right)  }\left(
v\right)  -H^{\left(  0\right)  }\left(  v\right)  \right)  }{\overline
{H}\left(  v\right)  }\right\}  dH_{n}^{\left(  1\right)  }\left(  w\right)
,
\]%
\[
\mathbb{T}_{n3}^{\ast}\left(  x\right)  :=\int_{xZ_{n-k:n}}^{\infty}%
\frac{\mathcal{\ell}\left(  w\right)  }{\overline{F}\left(  Z_{n-k:n}\right)
}\left\{  \int_{0}^{w}\frac{\overline{H}\left(  v\right)  -\overline{H}%
_{n}\left(  v\right)  }{\overline{H}^{2}\left(  v\right)  }dH_{n}^{\left(
0\right)  }\left(  v\right)  \right\}  dH_{n}^{\left(  1\right)  }\left(
w\right)  ,
\]%
\[
R_{n1}\left(  x\right)  :=\int_{xZ_{n-k:n}}^{\infty}\frac{\mathcal{\ell
}\left(  w\right)  }{\overline{F}\left(  Z_{n-k:n}\right)  }\left\{  \int
_{0}^{w}\frac{\left(  \overline{H}\left(  v\right)  -\overline{H}_{n}\left(
v\right)  \right)  ^{2}}{\overline{H}_{n}\left(  v\right)  \overline{H}%
^{2}\left(  v\right)  }dH_{n}^{\left(  0\right)  }\left(  w\right)  \right\}
dH_{n}^{\left(  1\right)  }\left(  w\right)  ,
\]
and%
\[
R_{n2}\left(  x\right)  :=\dfrac{1}{2}%
%TCIMACRO{\dint _{xZ_{n-k:n}}^{\infty}}%
%BeginExpansion
{\displaystyle\int_{xZ_{n-k:n}}^{\infty}}
%EndExpansion
\dfrac{\widetilde{\mathcal{\ell}}_{n}\left(  w\right)  }{\overline{F}\left(
Z_{n-k:n}\right)  }\left\{
%TCIMACRO{\dint _{0}^{w}}%
%BeginExpansion
{\displaystyle\int_{0}^{w}}
%EndExpansion
\dfrac{dH_{n}^{\left(  0\right)  }\left(  v\right)  }{\overline{H}_{n}\left(
v\right)  }-%
%TCIMACRO{\dint _{0}^{w}}%
%BeginExpansion
{\displaystyle\int_{0}^{w}}
%EndExpansion
\dfrac{dH^{\left(  0\right)  }\left(  v\right)  }{\overline{H}\left(
v\right)  }\right\}  ^{2}dH_{n}^{\left(  1\right)  }\left(  w\right)  .
\]
First, we clearly see that $\mathbb{T}_{n1}^{\ast}\left(  x\right)
\equiv\mathbb{T}_{n1}\left(  x\right)  .$ Next, we show that $\mathbb{T}%
_{n2}^{\ast}\left(  x\right)  $ and $\mathbb{T}_{n3}^{\ast}\left(  x\right)  $
are approximations of $\mathbb{T}_{n2}\left(  x\right)  $ and $\mathbb{T}%
_{n3}\left(  x\right)  $ respectively, while $R_{ni}\left(  x\right)
=O_{\mathbf{p}}\left(  k^{-2\eta\pm\epsilon_{0}}\right)  x^{\left(
2\eta-p\right)  /\gamma\pm\epsilon_{0}},$ $i=1,2,$ as $n\rightarrow\infty,$
for any $x\geq p^{\gamma}.$ Since, from representation $\left(  \ref{HN1}%
\right)  ,$ we have $\overline{H}_{n}^{\left(  1\right)  }\left(  w\right)
=\mathbb{U}_{n}\left(  \overline{H}^{\left(  1\right)  }\left(  w\right)
\right)  $ a.s., then without loss of generality, we may write%
\[
\mathbb{T}_{n2}^{\ast}\left(  x\right)  =\int_{xZ_{n-k:n}}^{\infty}%
\frac{\mathcal{\ell}\left(  w\right)  }{\overline{F}\left(  Z_{n-k:n}\right)
}\left\{  \int_{0}^{w}\frac{d\left(  H_{n}^{\left(  0\right)  }\left(
v\right)  -H^{\left(  0\right)  }\left(  v\right)  \right)  }{\overline
{H}\left(  v\right)  }\right\}  d\mathbb{U}_{n}\left(  \overline{H}^{\left(
1\right)  }\left(  w\right)  \right)  .
\]
Let $Q^{\left(  1\right)  }\left(  u\right)  :=\inf\left\{  w:\overline
{H}^{\left(  1\right)  }\left(  w\right)  >u\right\}  ,$ $0<u<\theta,$ and set
$t=\mathbb{U}_{n}\left(  \overline{H}^{\left(  1\right)  }\left(  w\right)
\right)  $ or, in other words, $w=Q^{\left(  1\right)  }\left(  \mathbb{V}%
_{n}\left(  t\right)  \right)  ,$ where $\mathbb{V}_{n}$ denotes the empirical
quantile function pertaining to $\mathbb{U}_{n}.$ By using this change of
variables, we get%
\[
\mathbb{T}_{n2}^{\ast}=\int_{0}^{\overline{H}_{n}^{\left(  1\right)  }\left(
xZ_{n-k:n}\right)  }\frac{\mathcal{L}\left(  \mathbb{V}_{n}\left(  t\right)
\right)  }{\overline{F}\left(  Z_{n-k:n}\right)  }\left\{  \int_{0}%
^{Q^{\left(  1\right)  }\left(  \mathbb{V}_{n}\left(  t\right)  \right)
}\frac{d\left(  \overline{H}_{n}^{\left(  0\right)  }\left(  v\right)
-\overline{H}^{\left(  0\right)  }\left(  v\right)  \right)  }{\overline
{H}\left(  v\right)  }\right\}  dt,
\]
where, for notational simplicity, we set $\mathcal{L}\left(  s\right)
:=\mathcal{\ell}\left(  Q^{\left(  1\right)  }\left(  s\right)  \right)  .$
Now, we decompose $\mathbb{T}_{n2}^{\ast}\left(  x\right)  $ into the sum of%
\[
A_{n1}\left(  x\right)  :=\int_{\overline{H}^{\left(  1\right)  }\left(
xZ_{n-k:n}\right)  }^{\overline{H}_{n}^{\left(  1\right)  }\left(
xZ_{n-k:n}\right)  }\frac{\mathcal{L}\left(  \mathbb{V}_{n}\left(  t\right)
\right)  }{\overline{F}\left(  Z_{n-k:n}\right)  }\int_{0}^{Q^{\left(
1\right)  }\left(  \mathbb{V}_{n}\left(  t\right)  \right)  }\frac{d\left(
\overline{H}_{n}^{\left(  0\right)  }\left(  v\right)  -\overline{H}^{\left(
0\right)  }\left(  v\right)  \right)  }{\overline{H}\left(  v\right)  }dt,
\]%
\[
A_{n2}\left(  x\right)  :=\int_{0}^{\overline{H}^{\left(  1\right)  }\left(
xZ_{n-k:n}\right)  }\frac{\mathcal{L}\left(  \mathbb{V}_{n}\left(  t\right)
\right)  -\mathcal{L}\left(  t\right)  }{\overline{F}\left(  Z_{n-k:n}\right)
}\int_{0}^{Q^{\left(  1\right)  }\left(  \mathbb{V}_{n}\left(  t\right)
\right)  }\frac{d\left(  \overline{H}_{n}^{\left(  0\right)  }\left(
v\right)  -\overline{H}^{\left(  0\right)  }\left(  v\right)  \right)
}{\overline{H}\left(  v\right)  }dt,
\]%
\[
A_{n3}\left(  x\right)  :=\int_{0}^{\overline{H}^{\left(  1\right)  }\left(
xZ_{n-k:n}\right)  }\frac{\mathcal{L}\left(  t\right)  }{\overline{F}\left(
Z_{n-k:n}\right)  }\int_{Q^{\left(  1\right)  }\left(  t\right)  }^{Q^{\left(
1\right)  }\left(  \mathbb{V}_{n}\left(  t\right)  \right)  }\frac{d\left(
\overline{H}_{n}^{\left(  0\right)  }\left(  v\right)  -\overline{H}^{\left(
0\right)  }\left(  v\right)  \right)  }{\overline{H}\left(  v\right)  }dt,
\]
and%
\[
A_{n4}\left(  x\right)  :=\int_{0}^{\overline{H}^{\left(  1\right)  }\left(
xZ_{n-k:n}\right)  }\frac{\mathcal{L}\left(  t\right)  }{\overline{F}\left(
Z_{n-k:n}\right)  }\int_{0}^{Q^{\left(  1\right)  }\left(  t\right)  }%
\frac{d\left(  \overline{H}_{n}^{\left(  0\right)  }\left(  v\right)
-\overline{H}^{\left(  0\right)  }\left(  v\right)  \right)  }{\overline
{H}\left(  v\right)  }dt.
\]
It is clear that the change of variables $w=Q^{\left(  1\right)  }\left(
t\right)  $ yields that $A_{n4}\left(  x\right)  =\mathbb{T}_{n2}\left(
x\right)  .$ Hence, one has to show that $A_{ni}\left(  x\right)
=O_{\mathbf{p}}\left(  k^{-2\eta\pm\epsilon}\right)  x^{\left(  2\eta
-p\right)  /\gamma\pm\epsilon_{0}},$ $i=1,2,3,$ as $n\rightarrow\infty,$
uniformly on $x\geq p^{\gamma}.$ We begin by $A_{n1}\left(  x\right)  $ for
which an integration by parts gives%
\[
A_{n1}\left(  x\right)  :=\int_{\overline{H}^{\left(  1\right)  }\left(
xZ_{n-k:n}\right)  }^{\overline{H}_{n}^{\left(  1\right)  }\left(
xZ_{n-k:n}\right)  }\frac{\mathcal{L}\left(  \mathbb{V}_{n}\left(  t\right)
\right)  }{\overline{F}\left(  Z_{n-k:n}\right)  }\int_{0}^{Q^{\left(
1\right)  }\left(  \mathbb{V}_{n}\left(  t\right)  \right)  }\frac{d\left(
\overline{H}_{n}^{\left(  0\right)  }\left(  v\right)  -\overline{H}^{\left(
0\right)  }\left(  v\right)  \right)  }{\overline{H}\left(  v\right)  }dt,
\]%
\begin{align*}
A_{n1}\left(  x\right)   &  =\int_{\overline{H}^{\left(  1\right)  }\left(
xZ_{n-k:n}\right)  }^{\overline{H}_{n}^{\left(  1\right)  }\left(
xZ_{n-k:n}\right)  }\frac{\mathcal{L}\left(  \mathbb{V}_{n}\left(  t\right)
\right)  }{\overline{F}\left(  Z_{n-k:n}\right)  }\left\{  \frac{\overline
{H}_{n}^{\left(  0\right)  }\left(  Q^{\left(  1\right)  }\left(
\mathbb{V}_{n}\left(  t\right)  \right)  \right)  -\overline{H}^{\left(
0\right)  }\left(  Q^{\left(  1\right)  }\left(  \mathbb{V}_{n}\left(
t\right)  \right)  \right)  }{\overline{H}\left(  Q^{\left(  1\right)
}\left(  \mathbb{V}_{n}\left(  t\right)  \right)  \right)  }\right. \\
&  \ \ \ \ \left.  -%
%TCIMACRO{\dint _{0}^{Q^{\left(  1\right)  }\left(  \mathbb{V}_{n}\left(
%t\right)  \right)  }}%
%BeginExpansion
{\displaystyle\int_{0}^{Q^{\left(  1\right)  }\left(  \mathbb{V}_{n}\left(
t\right)  \right)  }}
%EndExpansion
\frac{\overline{H}_{n}^{\left(  0\right)  }\left(  v\right)  -\overline
{H}^{\left(  0\right)  }\left(  v\right)  }{\overline{H}\left(  v\right)
}d\overline{H}\left(  v\right)  \right\}  dt.
\end{align*}
If $I_{n}\left(  x\right)  $ denotes the interval of endpoints $\overline
{H}^{\left(  1\right)  }\left(  xZ_{n-k:n}\right)  $ and $\overline{H}%
_{n}^{\left(  1\right)  }\left(  xZ_{n-k:n}\right)  ,$ then from Lemma
$\left(  \ref{Lemma1}\right)  ,$ we have $\sup_{t\in I_{n}\left(  x\right)
}\left(  \mathcal{L}\left(  t\right)  /\mathcal{L}\left(  \mathbb{V}%
_{n}\left(  t\right)  \right)  \right)  =O_{\mathbf{p}}\left(  1\right)  ,$
uniformly on $x\geq p^{\gamma}.$ On the other hand, form representation
$\left(  \ref{HN1}\right)  ,$ we infer that
\[
\left\{  \overline{H}_{n}^{\left(  0\right)  }\left(  t\right)  ,\text{
}0<s<1-\theta\right\}  \overset{\mathcal{D}}{=}\left\{  \mathbb{U}_{n}\left(
\overline{H}^{\left(  0\right)  }\left(  t\right)  \right)  ,\text{
}0<s<1-\theta\right\}  .
\]
Then by using assertion $\left(  ii\right)  $ in Proposition $\ref{Propo1},$
we have that%
\begin{align*}
A_{n1}\left(  x\right)   &  =O_{\mathbf{p}}\left(  n^{-\eta}\right)
\int_{M_{n}^{-}\left(  x\right)  }^{M_{n}^{+}\left(  x\right)  }%
\frac{\mathcal{L}\left(  t\right)  }{\overline{F}\left(  Z_{n-k:n}\right)  }\\
&  \times\left\{  \frac{\left(  \overline{H}^{\left(  0\right)  }\left(
Q^{\left(  1\right)  }\left(  \mathbb{V}_{n}\left(  t\right)  \right)
\right)  \right)  ^{1-\eta}}{\overline{H}\left(  Q^{\left(  1\right)  }\left(
\mathbb{V}_{n}\left(  t\right)  \right)  \right)  }+\int_{0}^{Q^{\left(
1\right)  }\left(  \mathbb{V}_{n}\left(  t\right)  \right)  }\frac{\left(
\overline{H}^{\left(  0\right)  }\left(  v\right)  \right)  ^{1/2-\eta
}dH\left(  v\right)  }{\overline{H}\left(  v\right)  }\right\}  dt,
\end{align*}
where $M_{n}^{+}\left(  x\right)  $ and $M_{n}^{-}\left(  x\right)  $ denote,
the maximum and the minimum of $\overline{H}_{n}^{\left(  1\right)  }\left(
xZ_{n-k:n}\right)  $ and $\overline{H}^{\left(  1\right)  }\left(
xZ_{n-k:n}\right)  $ respectively. Observe that both $\overline{H}^{\left(
0\right)  }\circ Q^{\left(  1\right)  }$ and $\overline{H}\circ Q^{\left(
1\right)  }$ are regularly varying at infinity with the same index equal to
$1.$ Then by using assertion $(iii)$ in Proposition $\ref{Propo1},$ we readily
show that $A_{n1}\left(  x\right)  $ equals%
\[
\frac{O_{\mathbf{p}}\left(  n^{-\eta}\right)  }{\overline{F}\left(
Z_{n-k:n}\right)  }\int_{M_{n}^{-}\left(  x\right)  }^{M_{n}^{+}\left(
x\right)  }\mathcal{L}\left(  t\right)  \left\{  \frac{\left(  \overline
{H}^{\left(  0\right)  }\left(  Q^{\left(  1\right)  }\left(  t\right)
\right)  \right)  ^{1-\eta}}{\overline{H}\left(  Q^{\left(  1\right)  }\left(
t\right)  \right)  }+\int_{0}^{Q^{\left(  1\right)  }\left(  t\right)  }%
\frac{\left(  \overline{H}^{\left(  0\right)  }\left(  v\right)  \right)
^{1-\eta}dH\left(  v\right)  }{\overline{H}\left(  v\right)  }\right\}  dt.
\]
Note that $\overline{H}^{\left(  0\right)  }<\overline{H},$ then it follows,
after integration, that%
\[
A_{n1}\left(  x\right)  =\frac{O_{\mathbf{p}}\left(  n^{-\eta}\right)
}{\overline{F}\left(  Z_{n-k:n}\right)  }\int_{M_{n}^{-}\left(  x\right)
}^{M_{n}^{+}\left(  x\right)  }\frac{\mathcal{L}\left(  t\right)  }{\left[
\overline{H}\left(  Q^{\left(  1\right)  }\left(  t\right)  \right)  \right]
^{\eta}}dt,
\]
which, by a change of variables, becomes%
\[
A_{n1}\left(  x\right)  =\frac{O_{\mathbf{p}}\left(  n^{-\eta}\right)
\overline{H}^{\left(  1\right)  }\left(  Z_{n-k:n}\right)  }{\overline
{F}\left(  Z_{n-k:n}\right)  }\int_{M_{n}^{-}\left(  x\right)  /\overline
{H}^{\left(  1\right)  }\left(  Z_{n-k:n}\right)  }^{M_{n}^{+}\left(
x\right)  /\overline{H}^{\left(  1\right)  }\left(  Z_{n-k:n}\right)  }%
\frac{\mathcal{L}\left(  \overline{H}^{\left(  1\right)  }\left(
Z_{n-k:n}\right)  t\right)  }{\left[  \overline{H}\left(  Q^{\left(  1\right)
}\left(  \overline{H}^{\left(  1\right)  }\left(  Z_{n-k:n}\right)  t\right)
\right)  \right]  ^{\eta}}dt.
\]
From now on, a key result related to the regular variation concept, namely
Potter's inequalities\textbf{ }%
\citep[see, e.g., Proposition B.1.9, assertion 5  in][]{deHF06}, will be
applied quite frequently. For this reason, we need to recall this very useful
tool here. Suppose that $\Psi$ is a regularly varying function at infinity
with index $\kappa$ and let $c_{1},c_{2}$ be positive real numbers. Then there
exists $t_{0}=t_{0}\left(  c_{1},c_{2}\right)  $ such that for $t\geq t_{0},$
$tx\geq t_{0},$%
\begin{equation}
\left(  1-c_{1}\right)  x^{\kappa}\min\left(  x^{c_{2}},x^{c_{2}}\right)
<\Psi\left(  tx\right)  /\Psi\left(  t\right)  <\left(  1+c_{1}\right)
x^{\kappa}\max\left(  x^{c_{2}},x^{c_{2}}\right)  . \label{Potter}%
\end{equation}
Since $\mathcal{L}$ and $\overline{H}\circ Q^{\left(  1\right)  }$ are
regularly varying at infinity with respective indices $p-1$ and $1,$ then we
use $\left(  \ref{Potter}\right)  $ to write that, for sufficiently small
$\epsilon_{0}>0$ and for all large $n,$ we have%
\[
A_{n1}\left(  x\right)  =\frac{O_{\mathbf{p}}\left(  n^{-\eta}\right)
\overline{H}^{\left(  1\right)  }\left(  Z_{n-k:n}\right)  \mathcal{L}\left(
\overline{H}^{\left(  1\right)  }\left(  Z_{n-k:n}\right)  \right)  }{F\left(
Z_{n-k:n}\right)  \left[  \overline{H}\left(  Q^{\left(  1\right)  }\left(
\overline{H}^{\left(  1\right)  }\left(  Z_{n-k:n}\right)  \right)  \right)
\right]  ^{\eta}}\int_{M_{n}^{-}\left(  x\right)  /\overline{H}^{\left(
1\right)  }\left(  Z_{n-k:n}\right)  }^{M_{n}^{+}\left(  x\right)
/\overline{H}^{\left(  1\right)  }\left(  Z_{n-k:n}\right)  }\frac
{t^{p-1\pm\epsilon_{0}}}{t^{\eta\pm\epsilon_{0}}}dt.
\]
From assertion $\left(  i\right)  $ of Lemma 4.1 in \cite{BMN-2015}, we have
$p\overline{H}\left(  t\right)  /\overline{H}^{\left(  1\right)  }\left(
t\right)  \rightarrow1,$ as $t\rightarrow\infty,$ hence $\overline{H}^{\left(
1\right)  }\left(  Z_{n-k:n}\right)  /\overline{H}\left(  Z_{n-k:n}\right)
\overset{\mathbf{p}}{\rightarrow}p.$ On the other hand, we have $Q^{\left(
1\right)  }\left(  \overline{H}^{\left(  1\right)  }\left(  t\right)  \right)
=t$ and $\dfrac{n}{k}\overline{H}\left(  Z_{n-k:n}\right)  \overset
{\mathbf{p}}{\rightarrow}1,$ thus%
\[
\frac{\overline{H}^{\left(  1\right)  }\left(  Z_{n-k:n}\right)
\mathcal{L}\left(  \overline{H}^{\left(  1\right)  }\left(  Z_{n-k:n}\right)
\right)  }{\overline{F}\left(  Z_{n-k:n}\right)  \left[  \overline{H}\left(
Q^{\left(  1\right)  }\left(  \overline{H}^{\left(  1\right)  }\left(
Z_{n-k:n}\right)  \right)  \right)  \right]  ^{\eta}}=\left(  1+o_{\mathbf{p}%
}\left(  1\right)  \right)  \left(  n/k\right)  ^{\eta}.
\]
Then, after integration, we obtain%
\[
A_{n1}\left(  x\right)  =\frac{O_{\mathbf{p}}\left(  n^{-\eta}\right)
}{\left(  k/n\right)  ^{\eta}}\left\vert \left(  \frac{\overline{H}%
_{n}^{\left(  1\right)  }\left(  xZ_{n-k:n}\right)  }{\overline{H}^{\left(
1\right)  }\left(  Z_{n-k:n}\right)  }\right)  ^{p-\eta\pm\epsilon_{0}%
}-\left(  \frac{\overline{H}^{\left(  1\right)  }\left(  xZ_{n-k:n}\right)
}{\overline{H}^{\left(  1\right)  }\left(  Z_{n-k:n}\right)  }\right)
^{p-\eta\pm\epsilon_{0}}\right\vert .
\]
We have $\dfrac{n}{k}\overline{H}^{\left(  1\right)  }\left(  Z_{n-k:n}%
\right)  \overset{\mathbf{p}}{\rightarrow}p,$ therefore
\[
A_{n1}\left(  x\right)  =\frac{O_{\mathbf{p}}\left(  n^{-\eta}\right)
}{\left(  k/n\right)  ^{p\pm\epsilon_{0}}}\left\vert \left(  \overline{H}%
_{n}^{\left(  1\right)  }\left(  xZ_{n-k:n}\right)  \right)  ^{p-\eta
\pm\epsilon_{0}}-\left(  \overline{H}^{\left(  1\right)  }\left(
xZ_{n-k:n}\right)  \right)  ^{p-\eta\pm\epsilon_{0}}\right\vert .
\]
By applying the mean value theorem, then assertion $\left(  i\right)  $ in
Proposition $\ref{Propo1},$ we have%
\[
A_{n1}\left(  x\right)  =\frac{O_{\mathbf{p}}\left(  n^{-\eta}\right)
}{\left(  k/n\right)  ^{p\pm\epsilon_{0}}}\left\vert \overline{H}_{n}^{\left(
1\right)  }\left(  xZ_{n-k:n}\right)  -\overline{H}^{\left(  1\right)
}\left(  xZ_{n-k:n}\right)  \right\vert \left(  \overline{H}^{\left(
1\right)  }\left(  xZ_{n-k:n}\right)  \right)  ^{p-\eta-1\pm\epsilon_{0}},
\]
and by assertion $\left(  ii\right)  $ in Proposition $\ref{Propo1},$ we get%
\[
A_{n1}\left(  x\right)  =\frac{O_{\mathbf{p}}\left(  n^{-2\eta}\right)
}{\left(  k/n\right)  ^{p\pm\epsilon_{0}}}\left(  \overline{H}^{\left(
1\right)  }\left(  xZ_{n-k:n}\right)  \right)  ^{p-2\eta\pm\epsilon_{0}}.
\]
Once again, by using the fact that $\dfrac{n}{k}\overline{H}^{\left(
1\right)  }\left(  Z_{n-k:n}\right)  \overset{\mathbf{p}}{\rightarrow}p$
together with routine manipulations of Potter's inequalities $\left(
\ref{Potter}\right)  ,$ we end up with $A_{n1}\left(  x\right)  =O_{\mathbf{p}%
}\left(  k^{-2\eta\pm\epsilon_{0}}\right)  x^{\left(  2\eta-p\right)
/\gamma\pm\epsilon_{0}}.$ Let us now consider the term $A_{n2}\left(
x\right)  ,$ which may be decomposed into the sum of%
\[
A_{n2}^{\left(  1\right)  }\left(  x\right)  :=\int_{x^{-1/\gamma}%
p/n}^{\overline{H}^{\left(  1\right)  }\left(  xZ_{n-k:n}\right)  }%
\frac{\mathcal{L}\left(  \mathbb{V}_{n}\left(  t\right)  \right)
-\mathcal{L}\left(  t\right)  }{\overline{F}\left(  Z_{n-k:n}\right)  }%
\int_{0}^{Q^{\left(  1\right)  }\left(  \mathbb{V}_{n}\left(  t\right)
\right)  }\frac{d\left(  \overline{H}_{n}^{\left(  0\right)  }\left(
v\right)  -\overline{H}^{\left(  0\right)  }\left(  v\right)  \right)
}{\overline{H}\left(  v\right)  }dt,
\]
and%
\[
A_{n2}^{\left(  2\right)  }\left(  x\right)  :=\int_{0}^{x^{-1/\gamma}%
p/n}\frac{\mathcal{L}\left(  \mathbb{V}_{n}\left(  t\right)  \right)
-\mathcal{L}\left(  t\right)  }{\overline{F}\left(  Z_{n-k:n}\right)  }%
\int_{0}^{Q^{\left(  1\right)  }\left(  \mathbb{V}_{n}\left(  t\right)
\right)  }\frac{d\left(  \overline{H}_{n}^{\left(  0\right)  }\left(
v\right)  -\overline{H}^{\left(  0\right)  }\left(  v\right)  \right)
}{\overline{H}\left(  v\right)  }dt.
\]
For the term $A_{n2}^{\left(  1\right)  }\left(  x\right)  ,$ we integrate the
second integral by parts to get%
\begin{align}
A_{n2}^{\left(  1\right)  }\left(  x\right)   &  =%
%TCIMACRO{\dint _{x^{-1/\gamma}p/n}^{\overline{H}^{\left(  1\right)  }\left(
%xZ_{n-k:n}\right)  }}%
%BeginExpansion
{\displaystyle\int_{x^{-1/\gamma}p/n}^{\overline{H}^{\left(  1\right)
}\left(  xZ_{n-k:n}\right)  }}
%EndExpansion
\frac{\mathcal{L}\left(  \mathbb{V}_{n}\left(  t\right)  \right)
-\mathcal{L}\left(  t\right)  }{\overline{F}\left(  Z_{n-k:n}\right)
}\left\{  \frac{\overline{H}_{n}^{\left(  0\right)  }\left(  Q^{\left(
1\right)  }\left(  \mathbb{V}_{n}\left(  t\right)  \right)  \right)
-\overline{H}^{\left(  0\right)  }\left(  Q^{\left(  1\right)  }\left(
\mathbb{V}_{n}\left(  t\right)  \right)  \right)  }{\overline{H}\left(
Q^{\left(  1\right)  }\left(  \mathbb{V}_{n}\left(  t\right)  \right)
\right)  }\right. \nonumber\\
&  \left.  -\int_{0}^{Q^{\left(  1\right)  }\left(  \mathbb{V}_{n}\left(
t\right)  \right)  }\frac{\overline{H}_{n}^{\left(  0\right)  }\left(
v\right)  -\overline{H}^{\left(  0\right)  }\left(  v\right)  }{\left(
\overline{H}\left(  v\right)  \right)  ^{2}}d\overline{H}\left(  v\right)
\right\}  dt. \label{d}%
\end{align}
Using assertion $\left(  ii\right)  $ in Proposition $\ref{Propo1}$ yields%
\[
\left\vert A_{n2}^{\left(  1\right)  }\left(  x\right)  \right\vert
=O_{\mathbf{p}}\left(  n^{-\eta}\right)  \int_{x^{-1/\gamma}p/n}^{\overline
{H}^{\left(  1\right)  }\left(  xZ_{n-k:n}\right)  }\frac{\left\vert
\mathcal{L}\left(  \mathbb{V}_{n}\left(  t\right)  \right)  -\mathcal{L}%
\left(  t\right)  \right\vert }{\overline{F}\left(  Z_{n-k:n}\right)  }%
\frac{dt}{\left(  \overline{H}\left(  Q^{\left(  1\right)  }\left(  t\right)
\right)  \right)  ^{\eta}},
\]
which may be rewritten into%
\[
O_{\mathbf{p}}\left(  n^{-\eta}\right)  \int_{x^{-1/\gamma}p/n}^{\overline
{H}^{\left(  1\right)  }\left(  xZ_{n-k:n}\right)  }\frac{\mathcal{L}\left(
t\right)  }{\overline{F}\left(  Z_{n-k:n}\right)  }\left\vert \frac
{\mathcal{L}\left(  \mathbb{V}_{n}\left(  t\right)  \right)  }{\mathcal{L}%
\left(  t\right)  }-1\right\vert \frac{dt}{\left(  \overline{H}\left(
Q^{\left(  1\right)  }\left(  t\right)  \right)  \right)  ^{\eta}}.
\]
In the interval of endpoints $\overline{H}^{\left(  1\right)  }\left(
xZ_{n-k:n}\right)  $ and $x^{-1/\gamma}p/n,$ we have, for large $n,$
$\mathbb{V}_{n}\left(  t\right)  $ is uniformly close to zero (in
probability). On the other hand, assertion $\left(  i\right)  $ in Proposition
$\ref{Propo1},$ implies that $\left(  \mathbb{V}_{n}\left(  t\right)
/t\right)  ^{\pm}$ is uniformly stochastically bounded. Then, making use of
Potter's inequalities $\left(  \ref{Potter}\right)  $ applied to the regularly
varying function $\mathcal{L}\left(  \cdot\right)  $ (with index $-q),$ we
have%
\[
\left(  1-\epsilon_{0}\right)  \left(  \mathbb{V}_{n}\left(  t\right)
/t\right)  ^{-\epsilon_{0}-q}<\frac{\mathcal{L}\left(  \mathbb{V}_{n}\left(
t\right)  \right)  }{\mathcal{L}\left(  t\right)  }<\left(  1+\epsilon
_{0}\right)  \left(  \mathbb{V}_{n}\left(  t\right)  /t\right)  ^{\epsilon
_{0}-q}.
\]
It is clear that%
\[
\left\vert \frac{\mathcal{L}\left(  \mathbb{V}_{n}\left(  t\right)  \right)
}{\mathcal{L}\left(  t\right)  }-1\right\vert \leq\max\left\{  \left\vert
\left(  1+\epsilon_{0}\right)  \left(  \mathbb{V}_{n}\left(  t\right)
/t\right)  ^{\epsilon_{0}-q}-1\right\vert ,\left\vert \left(  1-\epsilon
_{0}\right)  \left(  \mathbb{V}_{n}\left(  t\right)  /t\right)  ^{-\epsilon
_{0}-q}-1\right\vert \right\}  .
\]
For the sake of simplicity, we rewrite this inequality into
\[
\left\vert \mathcal{L}\left(  \mathbb{V}_{n}\left(  t\right)  \right)
/\mathcal{L}\left(  t\right)  -1\right\vert \leq\left\vert \left(
1\pm\epsilon_{0}\right)  \left(  \mathbb{V}_{n}\left(  t\right)  /t\right)
^{\pm\epsilon_{0}-q}-1\right\vert
\]
which is in turn is less than of equal to $\left\vert \left(  \mathbb{V}%
_{n}\left(  t\right)  /t\right)  ^{\pm\epsilon_{0}-q}-1\right\vert
+\epsilon_{0}\left(  \mathbb{V}_{n}\left(  t\right)  /t\right)  ^{\pm
\epsilon_{0}-q}.$ In other words, we have%
\[
\left\vert \mathcal{L}\left(  \mathbb{V}_{n}\left(  t\right)  \right)
/\mathcal{L}\left(  t\right)  -1\right\vert \leq\frac{\left\vert \left(
\mathbb{V}_{n}\left(  t\right)  \right)  ^{\pm\epsilon_{0}-q}-t^{\pm
\epsilon_{0}-q}\right\vert }{t^{\pm\epsilon_{0}-q}}+\epsilon_{0}\left(
\mathbb{V}_{n}\left(  t\right)  /t\right)  ^{\pm\epsilon_{0}-q}.
\]
By making use of assertion $\left(  ii\right)  $ in Proposition $\ref{Propo1}$
(once again), we show that
\[
\left\vert \mathcal{L}\left(  \mathbb{V}_{n}\left(  t\right)  \right)
/\mathcal{L}\left(  t\right)  -1\right\vert \leq O_{\mathbf{p}}\left(
n^{-\eta}\right)  t^{\left(  \eta-1\right)  \left(  \pm\epsilon_{0}-q\right)
}+o_{\mathbf{p}}\left(  1\right)  .
\]
Therefore%
\begin{align*}
\left\vert A_{n2}^{\left(  1\right)  }\left(  x\right)  \right\vert  &  \leq
O_{\mathbf{p}}\left(  n^{-2\eta}\right)  \int_{x^{-1/\gamma}p/n}^{\overline
{H}^{\left(  1\right)  }\left(  xZ_{n-k:n}\right)  }\frac{\mathcal{L}\left(
t\right)  }{\overline{F}\left(  Z_{n-k:n}\right)  }\frac{t^{\left(
\eta-1\right)  \left(  \pm\epsilon_{0}-q\right)  }}{\left(  \overline
{H}\left(  Q^{\left(  1\right)  }\left(  t\right)  \right)  \right)  ^{\eta}%
}dt\\
&  +o_{\mathbf{p}}\left(  n^{-\eta}\right)  \int_{x^{-1/\gamma}p/n}%
^{\overline{H}^{\left(  1\right)  }\left(  xZ_{n-k:n}\right)  }\frac
{\mathcal{L}\left(  t\right)  }{\overline{F}\left(  Z_{n-k:n}\right)  }%
\frac{dt}{\left(  \overline{H}\left(  Q^{\left(  1\right)  }\left(  t\right)
\right)  \right)  ^{\eta}}.
\end{align*}
By a similar treatment as the above (we omit the details), we end up with
$A_{n2}^{\left(  1\right)  }\left(  x\right)  =O_{\mathbf{p}}\left(
k^{-2\eta\pm\epsilon_{0}}\right)  x^{\left(  2\eta-p\right)  /\gamma
\pm\epsilon_{0}}.$ For the term $A_{n2}^{\left(  2\right)  }\left(  x\right)
,$ we start by noting that $x^{-1/\gamma}p/n\leq1/n,$ for any $x\geq
p^{\gamma}.$ Since $\mathbb{V}_{n}\left(  t\right)  =U_{1:n},$ for
$0<t\leq1/n,$ it follows that
\[
A_{n2}^{\left(  2\right)  }\left(  x\right)  =\left\{  \int_{0}^{x^{-1/\gamma
}p/n}\frac{\mathcal{L}\left(  U_{1:n}\right)  -\mathcal{L}\left(  t\right)
}{\overline{F}\left(  Z_{n-k:n}\right)  }dt\right\}  \int_{0}^{Q^{\left(
1\right)  }\left(  U_{1:n}\right)  }\frac{d\left(  \overline{H}_{n}^{\left(
0\right)  }\left(  v\right)  -\overline{H}^{\left(  0\right)  }\left(
v\right)  \right)  }{\overline{H}\left(  v\right)  }.
\]
For the second integral, we use analogous arguments based on assertion
$\left(  ii\right)  $ in Proposition $\ref{Propo1},$ to write
\[
A_{n2}^{\left(  2\right)  }\left(  x\right)  =\frac{O_{\mathbf{p}}\left(
n^{-\eta}\right)  }{\overline{F}\left(  Z_{n-k:n}\right)  \left(  \overline
{H}\left(  Q^{\left(  1\right)  }\left(  U_{1:n}\right)  \right)  \right)
^{\eta}}\int_{0}^{x^{-1/\gamma}p/n}\left(  \mathcal{L}\left(  U_{1:n}\right)
+\mathcal{L}\left(  t\right)  \right)  dt.
\]
Note that the latter integral is equal to $\mathcal{L}\left(  U_{1:n}\right)
px^{-1/\gamma}/n+\int_{0}^{x^{-1/\gamma}p/n}\mathcal{L}\left(  t\right)  dt.$
By Potter's inequalities $\left(  \ref{Potter}\right)  $ on $\mathcal{L}$ and
the fact that $nU_{1:n}\overset{\mathbf{p}}{\rightarrow}1,$ it becomes
$\left(  1+o_{\mathbf{p}}\left(  1\right)  \right)  pn^{-1}\mathcal{L}\left(
1/n\right)  x^{-1/\gamma}.$\textbf{ }By replacing $\mathcal{L}$ by its
expression, we get
\[
A_{n2}^{\left(  2\right)  }\left(  x\right)  =\frac{\overline{F}\left(
Q^{\left(  1\right)  }\left(  1/n\right)  \right)  }{\overline{F}\left(
H^{\leftarrow}\left(  1-k/n\right)  \right)  }\frac{n^{-\eta-1}}{\left(
\overline{H}\left(  Q^{\left(  1\right)  }\left(  1/n\right)  \right)
\right)  ^{\eta+1}}O_{\mathbf{p}}\left(  x^{-1/\gamma}\right)  .
\]
Now, we apply Potter's inequalities $\left(  \ref{Potter}\right)  $\ to
$\overline{F},$ to write
\begin{equation}
\frac{\overline{F}\left(  Q^{\left(  1\right)  }\left(  1/n\right)  \right)
}{\overline{F}\left(  H^{\leftarrow}\left(  1-k/n\right)  \right)  }=O\left(
1\right)  \left(  \frac{H^{\leftarrow}\left(  1-k/n\right)  }{Q^{\left(
1\right)  }\left(  1/n\right)  }\right)  ^{1/\gamma_{1}\pm\epsilon_{0}},\text{
as }n\rightarrow\infty. \label{equ}%
\end{equation}
Since $\overline{H}\left(  w\right)  >\overline{H}^{1}\left(  w\right)  ,$
then $Q^{\left(  1\right)  }\left(  s\right)  >H^{\leftarrow}\left(
1-s\right)  $ and therefore
\[
\frac{H^{\leftarrow}\left(  1-k/n\right)  }{Q^{\left(  1\right)  }\left(
1/n\right)  }\leq\frac{H^{\leftarrow}\left(  1-k/n\right)  }{H^{\leftarrow
}\left(  1-1/n\right)  },
\]
which, by (once again) using Potter's inequalities $\left(  \ref{Potter}%
\right)  $ to $H^{\leftarrow}\left(  1-s\right)  ,$ equals $O\left(
k^{-\gamma\pm\epsilon_{0}}\right)  .$ Then the right-hand side of $\left(
\ref{equ}\right)  $ is asymptotically equal to $O\left(  k^{-\gamma/\gamma
_{1}\pm\epsilon_{0}}\right)  .$ On the other hand, we have $\left(
\overline{H}\left(  Q^{\left(  1\right)  }\left(  1/n\right)  \right)
\right)  ^{\eta+1}>\left(  \overline{H}^{\left(  1\right)  }\left(  Q^{\left(
1\right)  }\left(  1/n\right)  \right)  \right)  ^{\eta+1}=n^{-\eta-1}%
,$\textbf{ }it follows that\textbf{ }$A_{n2}^{\left(  2\right)  }\left(
x\right)  =O_{\mathbf{p}}\left(  k^{-p\pm\epsilon_{0}}\right)  x^{-1/\gamma}%
,$\textbf{ }where\textbf{ }$p=\gamma/\gamma_{1}$ is assumed to be greater than
$1/2.$ Consequently, we have\textbf{ }$A_{n2}^{\left(  2\right)  }\left(
x\right)  =O_{\mathbf{p}}\left(  k^{-2\eta\pm\epsilon_{0}}\right)  x^{\left(
2\eta-p\right)  /\gamma\pm\epsilon_{0}},$\textbf{ }for any $1/4<\eta\leq p/2,$
and thus $A_{n2}\left(  x\right)  =O_{\mathbf{p}}\left(  k^{-2\eta\pm
\epsilon_{0}}\right)  x^{\left(  2\eta-p\right)  /\gamma\pm\epsilon_{0}}$ as
well.\textbf{ }By similar arguments, we show that $A_{n3}\left(  x\right)  $
asymptotically equals the same quantity, therefore we omit the details.
Finally, we may write that $\mathbb{T}_{n2}^{\ast}\left(  x\right)
=\mathbb{T}_{n2}\left(  x\right)  +O_{\mathbf{p}}\left(  k^{-2\eta\pm
\epsilon_{0}}\right)  x^{\left(  2\eta-p\right)  /\gamma\pm\epsilon_{0}}.$ For
the term $\mathbb{T}_{n3}^{\ast}\left(  x\right)  $ we proceed as we did for
$\mathbb{T}_{n2}^{\ast}\left(  x\right)  $ but with more tedious
manipulations. By two steps, we have to get rid of integrands $dH_{n}^{\left(
0\right)  }\left(  v\right)  $ and $dH_{n}^{\left(  1\right)  }\left(
w\right)  $ and replace them by their theoretical counterparts $dH^{\left(
0\right)  }\left(  v\right)  $ and $dH^{\left(  1\right)  }\left(  w\right)
.$ First, we define, for $0<t<1-\theta,$%
\[
Q^{\left(  0\right)  }\left(  t\right)  :=\inf\left\{  v:\overline{H}^{\left(
0\right)  }\left(  v\right)  >t\right\}  \text{ and }Q_{n}^{\left(  0\right)
}\left(  t\right)  :=\inf\left\{  v:\overline{H}_{n}^{\left(  0\right)
}\left(  v\right)  >t\right\}  .
\]
By the change of variables $v=Q_{n}^{\left(  0\right)  }\left(  t\right)  $
and similar arguments to those used for the terms $A_{ni}\left(  x\right)  ,$
we show that
\begin{align*}
\mathbb{T}_{n3}^{\ast}\left(  x\right)   &  =\int_{xZ_{n-k:n}}^{\infty}%
\frac{\mathcal{\ell}\left(  w\right)  }{\overline{F}\left(  Z_{n-k:n}\right)
}\left\{  \int_{0}^{w}\frac{\overline{H}\left(  v\right)  -\overline{H}%
_{n}\left(  v\right)  }{\overline{H}^{2}\left(  v\right)  }dH^{\left(
0\right)  }\left(  v\right)  \right\}  dH_{n}^{\left(  1\right)  }\left(
w\right) \\
&  +O_{\mathbf{p}}\left(  k^{-2\eta\pm\epsilon_{0}}\right)  x^{\left(
2\eta-p\right)  /\gamma\pm\epsilon_{0}}.
\end{align*}
Second, we use the change of variables $w=Q_{n}^{\left(  1\right)  }\left(
s\right)  $ and proceed as above to get $\mathbb{T}_{n3}^{\ast}\left(
x\right)  =\mathbb{T}_{n3}\left(  x\right)  +O_{\mathbf{p}}\left(
k^{-2\eta\pm\epsilon_{0}}\right)  x^{\left(  2\eta-p\right)  /\gamma
\pm\epsilon_{0}}.$ At this stage, we proved that $\mathbb{T}_{n1}\left(
x\right)  $ is exactly $\mathbb{T}_{n1}^{\ast}\left(  x\right)  $ and that
$\mathbb{T}_{n2}\left(  x\right)  $ and\textbf{ }$\mathbb{T}_{n3}\left(
x\right)  $ are approximated by $\mathbb{T}_{n2}^{\ast}\left(  x\right)  $ and
$\mathbb{T}_{n3}^{\ast}\left(  x\right)  $ respectively. For the first
remainder term $R_{n1}\left(  x\right)  ,$ it suffices to follow the same
procedure to show that, uniformly on $x\geq p^{\gamma},$ $\sqrt{k}%
R_{n1}\left(  x\right)  =O_{\mathbf{p}}\left(  k^{-2\eta\pm\epsilon_{0}%
}\right)  x^{\left(  2\eta-p\right)  /\gamma\pm\epsilon_{0}}.$ As for the last
term $R_{n2}\left(  x\right)  ,$ we use similar technics to the decomposition
\begin{align*}
&
%TCIMACRO{\dint _{0}^{w}}%
%BeginExpansion
{\displaystyle\int_{0}^{w}}
%EndExpansion
\dfrac{dH_{n}^{\left(  0\right)  }\left(  v\right)  }{\overline{H}_{n}\left(
v\right)  }-%
%TCIMACRO{\dint _{0}^{w}}%
%BeginExpansion
{\displaystyle\int_{0}^{w}}
%EndExpansion
\dfrac{dH^{\left(  0\right)  }\left(  v\right)  }{\overline{H}\left(
v\right)  }\\
&  =%
%TCIMACRO{\dint _{0}^{w}}%
%BeginExpansion
{\displaystyle\int_{0}^{w}}
%EndExpansion
\dfrac{d\left(  H_{n}^{\left(  0\right)  }\left(  v\right)  -H^{\left(
0\right)  }\left(  v\right)  \right)  }{\overline{H}_{n}\left(  v\right)  }+%
%TCIMACRO{\dint _{0}^{w}}%
%BeginExpansion
{\displaystyle\int_{0}^{w}}
%EndExpansion
\frac{\overline{H}_{n}\left(  v\right)  -\overline{H}\left(  v\right)
}{\overline{H}_{n}\left(  v\right)  \overline{H}\left(  v\right)  }dH^{\left(
0\right)  }\left(  v\right)  ,
\end{align*}
to show that $\sqrt{k}R_{n1}\left(  x\right)  =O_{\mathbf{p}}\left(
k^{-2\eta\pm\epsilon_{0}}\right)  x^{\left(  2\eta-p\right)  /\gamma
\pm\epsilon_{0}}$ as well, therefore we omit details. Now, the weak
approximation $\left(  \ref{MN1-decompos}\right)  $ of $\mathbb{M}_{n1}\left(
x\right)  $ is well established.

\subsubsection{\textbf{Gaussian approximation to }$D_{n}\left(  x\right)  $}

We start by representing each $\mathbb{T}_{nj}\left(  x\right)  ,$ $j=1,2,3,$
in terms of the processes $\left(  \ref{betan}\right)  $ and $\left(
\ref{beta-tild}\right)  .$ Note that $\mathbb{T}_{n1}\left(  x\right)  $ may
be rewritten into%
\[
\mathbb{T}_{n1}\left(  x\right)  =-\int_{xZ_{n-k:n}}^{\infty}\frac
{\overline{F}\left(  w\right)  }{\overline{F}\left(  Z_{n-k:n}\right)  }%
\frac{d\left(  \overline{H}_{n}^{\left(  1\right)  }\left(  w\right)
-\overline{H}^{\left(  1\right)  }\left(  w\right)  \right)  }{\overline
{H}\left(  w\right)  },
\]
which, by integration by parts, becomes%
\[
\mathbb{T}_{n1}\left(  x\right)  =-\left[  \frac{\overline{F}\left(  w\right)
}{\overline{F}\left(  Z_{n-k:n}\right)  }\frac{\overline{H}_{n}^{\left(
1\right)  }\left(  w\right)  -\overline{H}^{\left(  1\right)  }\left(
w\right)  }{\overline{H}\left(  w\right)  }\right]  _{xZ_{n-k:n}}^{\infty
}+\int_{xZ_{n-k:n}}^{\infty}\frac{\overline{H}_{n}^{\left(  1\right)  }\left(
w\right)  -\overline{H}^{\left(  1\right)  }\left(  w\right)  }{\overline
{H}\left(  w\right)  }\frac{d\overline{F}\left(  w\right)  }{\overline
{F}\left(  Z_{n-k:n}\right)  }.
\]
Observe that, for $u\geq Z_{n:n},$ $\overline{H}_{n}^{\left(  1\right)
}\left(  u\right)  =0$ and, from Lemma 4.1 in \cite{BMN-2015}, $\overline
{H}^{\left(  1\right)  }\left(  w\right)  /\overline{H}\left(  w\right)  $
tends to $p$ as $w\rightarrow\infty.$ It follows that%
\[
\lim_{w\rightarrow\infty}\overline{F}\left(  w\right)  \frac{\overline{H}%
_{n}^{\left(  1\right)  }\left(  w\right)  -\overline{H}^{\left(  1\right)
}\left(  w\right)  }{\overline{H}\left(  w\right)  }=-\lim_{w\rightarrow
\infty}\overline{F}\left(  w\right)  \frac{\overline{H}^{\left(  1\right)
}\left(  w\right)  }{\overline{H}\left(  w\right)  }=0.
\]
Thus, after a change of variables, we have%
\begin{align*}
\mathbb{T}_{n1}\left(  x\right)   &  =\frac{\overline{F}\left(  xZ_{n-k:n}%
\right)  }{\overline{F}\left(  Z_{n-k:n}\right)  }\frac{\overline{H}%
_{n}^{\left(  1\right)  }\left(  xZ_{n-k:n}\right)  -\overline{H}^{\left(
1\right)  }\left(  xZ_{n-k:n}\right)  }{\overline{H}\left(  xZ_{n-k:n}\right)
}\\
&  \ \ \ \ \ \ \ \ \ \ +\int_{x}^{\infty}\frac{\overline{H}_{n}^{\left(
1\right)  }\left(  wZ_{n-k:n}\right)  -\overline{H}^{\left(  1\right)
}\left(  wZ_{n-k:n}\right)  }{\overline{H}\left(  wZ_{n-k:n}\right)  }%
\frac{d\overline{F}\left(  wZ_{n-k:n}\right)  }{\overline{F}\left(
Z_{n-k:n}\right)  }.
\end{align*}
Making the change of variables $w=F^{\leftarrow}\left(  1-s\overline{F}\left(
Z_{n-k:n}\right)  \right)  /Z_{n-k:n}=:\xi_{n}\left(  s\right)  $ and using
$\left(  \ref{rep-H1}\right)  $ and $\left(  \ref{betan}\right)  $ yield%
\[
\mathbb{T}_{n1}\left(  x\right)  =\frac{\sqrt{k}}{n}\frac{\overline{F}\left(
xZ_{n-k:n}\right)  }{\overline{F}\left(  Z_{n-k:n}\right)  }\frac{\beta
_{n}\left(  x\right)  }{\overline{H}\left(  xZ_{n-k:n}\right)  }+\frac
{\sqrt{k}}{n}\int_{0}^{\frac{\overline{F}\left(  xZ_{n-k:n}\right)
}{\overline{F}\left(  Z_{n-k:n}\right)  }}\frac{\beta_{n}\left(  \xi
_{n}\left(  s\right)  \right)  }{\overline{H}\left(  \xi_{n}\left(  s\right)
Z_{n-k:n}\right)  }ds.
\]
Next, we show that, for $\epsilon_{0}>0$ sufficiently small, we have%
\begin{equation}
\sqrt{k}\mathbb{T}_{n1}\left(  x\right)  =x^{-1/\gamma_{2}}\beta_{n}\left(
x\right)  +\frac{q}{\gamma}\int_{x}^{\infty}w^{1/\gamma_{2}-1}\beta_{n}\left(
w\right)  dw+o_{\mathbf{p}}\left(  x^{\left(  2\eta-p\right)  /\gamma
\pm\epsilon_{0}}\right)  . \label{TN1x}%
\end{equation}
To this end, let us decompose $\mathbb{T}_{n1}\left(  x\right)  $ into the sum
of
\[
\mathbb{T}_{n1}^{\left(  1\right)  }\left(  x\right)  :=\frac{\sqrt{k}}%
{n}\frac{\overline{F}\left(  xZ_{n-k:n}\right)  }{\overline{F}\left(
Z_{n-k:n}\right)  }\frac{\beta_{n}\left(  x\right)  }{\overline{H}\left(
xZ_{n-k:n}\right)  },\text{ }\mathbb{T}_{n1}^{\left(  2\right)  }\left(
x\right)  :=\frac{\sqrt{k}}{n}\int_{x^{-1/\gamma_{1}}}^{\frac{\overline
{F}\left(  xZ_{n-k:n}\right)  }{\overline{F}\left(  Z_{n-k:n}\right)  }}%
\frac{\beta_{n}\left(  \xi_{n}\left(  s\right)  \right)  }{\overline{H}\left(
\xi_{n}\left(  s\right)  Z_{n-k:n}\right)  }ds,
\]%
\[
\mathbb{T}_{n1}^{\left(  3\right)  }\left(  x\right)  :=\frac{\sqrt{k}}{n}%
\int_{0}^{x^{-1/\gamma_{1}}}\left\{  \frac{\overline{H}\left(  s^{-\gamma_{1}%
}h\right)  }{\overline{H}\left(  \xi_{n}\left(  s\right)  Z_{n-k:n}\right)
}-1\right\}  \frac{\beta_{n}\left(  \xi_{n}\left(  s\right)  \right)
}{\overline{H}\left(  s^{-\gamma_{1}}h\right)  }ds,
\]%
\[
\mathbb{T}_{n1}^{\left(  4\right)  }\left(  x\right)  :=\frac{\sqrt{k}}{n}%
\int_{0}^{x^{-1/\gamma_{1}}}\frac{\beta_{n}\left(  \xi_{n}\left(  s\right)
\right)  -\beta_{n}\left(  s^{-\gamma_{1}}\right)  }{\overline{H}\left(
s^{-\gamma_{1}}h\right)  }ds,
\]%
\[
\mathbb{T}_{n1}^{\left(  5\right)  }\left(  x\right)  :=\frac{\sqrt{k}}{n}%
\int_{0}^{x^{-1/\gamma_{1}}}\left\{  \frac{s^{\gamma_{1}/\gamma}\overline
{H}\left(  h\right)  }{\overline{H}\left(  s^{-\gamma_{1}}h\right)
}-1\right\}  \frac{\beta_{n}\left(  s^{-\gamma_{1}}\right)  }{s^{\gamma
_{1}/\gamma}\overline{H}\left(  h\right)  }ds,\text{ }\mathbb{T}_{n1}^{\left(
6\right)  }\left(  x\right)  :=\dfrac{\sqrt{k}}{n}%
%TCIMACRO{\dint _{0}^{x^{-1/\gamma_{1}}}}%
%BeginExpansion
{\displaystyle\int_{0}^{x^{-1/\gamma_{1}}}}
%EndExpansion
\dfrac{\beta_{n}\left(  s^{-\gamma_{1}}\right)  }{s^{\gamma_{1}/\gamma
}\overline{H}\left(  h\right)  }ds.
\]
We shall show that $\sqrt{k}\mathbb{T}_{n1}^{\left(  i\right)  }\left(
x\right)  =o_{\mathbf{p}}\left(  x^{\left(  2\eta-p\right)  /\gamma\pm
\epsilon_{0}}\right)  ,$ $i=2,...,5,$ uniformly on $x\geq p^{\gamma},$ while
$\sqrt{k}\mathbb{T}_{n1}^{\left(  1\right)  }\left(  x\right)  $ and $\sqrt
{k}\mathbb{T}_{n1}^{\left(  6\right)  }\left(  x\right)  $ are approximations
to the first and second terms in $\left(  \ref{TN1x}\right)  $ respectively.
Let us begin by $\sqrt{k}\mathbb{T}_{n1}^{\left(  2\right)  }\left(  x\right)
$ and write%
\[
\sqrt{k}\left\vert \mathbb{T}_{n1}^{\left(  2\right)  }\left(  x\right)
\right\vert \leq\frac{k}{n}\int_{0}^{c_{n}\left(  x\right)  }\frac{\left\vert
\beta_{n}\left(  \xi_{n}\left(  s\right)  \right)  \right\vert }{\overline
{H}\left(  \xi_{n}\left(  s\right)  Z_{n-k:n}\right)  }ds,
\]
where $c_{n}\left(  x\right)  :=\max\left(  \overline{F}\left(  xZ_{n-k:n}%
\right)  /\overline{F}\left(  Z_{n-k:n}\right)  ,x^{-1/\gamma_{1}}\right)  .$
Next, we provide a lower bound to $\overline{H}\left(  \xi_{n}\left(
s\right)  Z_{n-k:n}\right)  $ by applying Potter's inequalities $\left(
\ref{Potter}\right)  $\ to $\overline{F}.$ Since $Z_{n-k:n}\rightarrow\infty$
a.s., then from the right-hand side of $\left(  \ref{Potter}\right)  $\ we
have $\overline{F}\left(  xZ_{n-k:n}\right)  /\overline{F}\left(
Z_{n-k:n}\right)  <2x^{-1/\gamma_{1}\pm\epsilon_{0}}$ a.s., which implies that
$c_{n}\left(  x\right)  <2x^{-1/\gamma_{1}\pm\epsilon_{0}}$ a.s., for any
$x\geq p^{\gamma},$ as well. Let us now rewrite the sequence $\xi_{n}\left(
s\right)  $ into%
\begin{equation}
\xi_{n}\left(  s\right)  =\frac{F^{\mathbb{\leftarrow}}\left(  1-s\overline
{F}\left(  Z_{n-k:n}\right)  \right)  }{F^{\mathbb{\leftarrow}}\left(
1-\overline{F}\left(  Z_{n-k:n}\right)  \right)  },\text{ }0<s<1, \label{ksi}%
\end{equation}
and use the left-hand side of Potter's inequalities $\left(  \ref{Potter}%
\right)  $\ for the quantile function $u\rightarrow F^{\mathbb{\leftarrow}%
}\left(  1-u\right)  ,$ to get $\xi_{n}\left(  s\right)  \geq2^{-1}%
s^{-\gamma_{1}\pm\epsilon_{0}}$ a.s., for any $x^{-1/\gamma_{1}}%
<s<\overline{F}\left(  xZ_{n-k:n}\right)  /\overline{F}\left(  Z_{n-k:n}%
\right)  .$ This implies that $\xi_{n}\left(  s\right)  \geq2^{-1}%
x^{1\pm\epsilon_{0}}\geq2^{-1}p^{\gamma\pm\epsilon_{0}}>0.$ Then $\xi
_{n}\left(  s\right)  Z_{n-k:n}\rightarrow\infty$ a.s. and therefore, from the
left-hand side of $\left(  \ref{Potter}\right)  $\ (applied to $\overline
{H}),$ we have $\overline{H}\left(  Z_{n-k:n}\right)  /\overline{H}\left(
\xi_{n}\left(  s\right)  Z_{n-k:n}\right)  =O\left(  \left(  \xi_{n}\left(
s\right)  \right)  ^{1/\gamma\pm\epsilon_{0}}\right)  $ a.s. uniformly on $s.$
This allows us to write that%
\[
\sqrt{k}\mathbb{T}_{n1}^{\left(  2\right)  }\left(  x\right)  =O\left(
1\right)  \frac{k/n}{\overline{H}\left(  Z_{n-k:n}\right)  }\int
_{0}^{2x^{-1/\gamma_{1}\pm\epsilon_{0}}}\left(  \xi_{n}\left(  s\right)
\right)  ^{1/\gamma\pm\epsilon_{0}}\left\vert \beta_{n}\left(  \xi_{n}\left(
s\right)  \right)  \right\vert ds,\text{ a.s.}%
\]
By combining Corollary 2.2.2 with Proposition B.1.10 in \cite{deHF06}, we have
$\left(  n/k\right)  \overline{H}\left(  Z_{n-k:n}\right)  \overset
{\mathbf{p}}{\rightarrow}1,$ hence%
\[
\sqrt{k}\mathbb{T}_{n1}^{\left(  2\right)  }\left(  x\right)  =O_{\mathbf{p}%
}\left(  1\right)  \int_{0}^{2x^{-1/\gamma_{1}\pm\epsilon_{0}}}\left(  \xi
_{n}\left(  s\right)  \right)  ^{1/\gamma\pm\epsilon_{0}}\left\vert \beta
_{n}\left(  \xi_{n}\left(  s\right)  \right)  \right\vert ds.
\]
From $\left(  \ref{Potter}\right)  ,$ we infer that $0<s<2x_{0}^{-1/\gamma
_{1}\pm\epsilon_{0}}=:s_{0}.$ On the other hand, in view of assertion $\left(
ii\right)  $ of Lemma \ref{Lemma2}, we have $\sup\nolimits_{0<s<s_{0}}\left(
\xi_{n}\left(  s\right)  \right)  ^{\left(  1-\eta\right)  /\gamma}\left\vert
\beta_{n}\left(  \xi_{n}\left(  s\right)  \right)  \right\vert =o_{\mathbf{p}%
}\left(  1\right)  ,$ therefore%
\[
\sqrt{k}\mathbb{T}_{n1}^{\left(  2\right)  }\left(  x\right)  =o_{\mathbf{p}%
}\left(  1\right)  \int_{0}^{2x^{-1/\gamma_{1}\pm\epsilon_{0}}}\left(  \xi
_{n}\left(  s\right)  \right)  ^{\eta/\gamma\pm\epsilon_{0}}ds.
\]
Note that $\left(  1-\eta\right)  /\gamma\pm\epsilon_{0}>0,$ then by using the
right-hand side of $\left(  \ref{Potter}\right)  $ (applied to
$F^{\mathbb{\leftarrow}}\left(  1-\cdot\right)  ),$ we get%
\[
\sqrt{k}\mathbb{T}_{n1}^{\left(  2\right)  }\left(  x\right)  =o_{\mathbf{p}%
}\left(  1\right)  \int_{0}^{2x^{-1/\gamma_{1}\pm\epsilon_{0}}}s^{-\eta
\gamma_{1}/\gamma\pm\epsilon_{0}}ds,
\]
which equals $o_{\mathbf{p}}\left(  x^{\eta/\gamma-1/\gamma_{1}\pm\epsilon
_{0}}\right)  .$ Recall that $\gamma_{1}=\gamma/p,$ then it is easy to verify
that $\sqrt{k}\mathbb{T}_{n1}^{\left(  2\right)  }\left(  x\right)
=o_{\mathbf{p}}\left(  x^{\left(  2\eta-p\right)  /\gamma\pm\epsilon_{0}%
}\right)  .$ By using similar arguments, we also show that $\sqrt{k}%
\mathbb{T}_{n1}^{\left(  i\right)  }\left(  x\right)  =o_{\mathbf{p}}\left(
x^{\left(  2\eta-p\right)  /\gamma\pm\epsilon_{0}}\right)  ,$ $i=3,5,$
therefore we omit the details. For the term $\mathbb{T}_{n4}\left(  x\right)
,$ we have%
\[
\sqrt{k}\left\vert \mathbb{T}_{n1}^{\left(  4\right)  }\left(  x\right)
\right\vert \leq\frac{k}{n}\int_{0}^{x^{-1/\gamma_{1}}}\frac{\left\vert
\beta_{n}\left(  \xi_{n}\left(  s\right)  \right)  -\beta_{n}\left(
s^{-\gamma_{1}}\right)  \right\vert }{\overline{H}\left(  s^{-\gamma_{1}%
}h\right)  }ds.
\]
In view of $\left(  \ref{Potter}\right)  $ with the fact that $\overline
{H}\left(  h\right)  =k/n,$ we have%
\[
\sqrt{k}\mathbb{T}_{n1}^{\left(  4\right)  }\left(  x\right)  =O_{\mathbf{p}%
}\left(  1\right)  \int_{0}^{x^{-1/\gamma_{1}}}s^{-\gamma_{1}/\gamma
\pm\epsilon_{0}}\left\vert \beta_{n}\left(  \xi_{n}\left(  s\right)  \right)
-\beta_{n}\left(  s^{-\gamma_{1}}\right)  \right\vert ds.
\]
From assertion $\left(  ii\right)  $ of Lemma \ref{Lemma2}, we have $\beta
_{n}\left(  \xi_{n}\left(  s\right)  -s^{-\gamma_{1}}\right)  =o_{\mathbf{p}%
}\left(  s^{\left(  1-\eta\right)  \gamma_{1}/\gamma}\right)  ,$ uniformly on
$0<s<x^{-1/\gamma_{1}},$ then after elementary calculation, we end up with
$\sqrt{k}\mathbb{T}_{n1}^{\left(  4\right)  }\left(  x\right)  =o_{\mathbf{p}%
}\left(  x^{\left(  2\eta-p\right)  /\gamma\pm\epsilon_{0}}\right)  .$ As for
the first term $\mathbb{T}_{n1}^{\left(  1\right)  }\left(  x\right)  ,$ it
suffices to use Potter's inequalities $\left(  \ref{Potter}\right)  $ (for
$\overline{F}$ and $\overline{H})$ to get
\[
\sqrt{k}\mathbb{T}_{n1}^{\left(  1\right)  }\left(  x\right)  =x^{-1/\gamma
_{2}}\beta_{n}\left(  x\right)  +o_{\mathbf{p}}\left(  x^{\left(
2\eta-p\right)  /\gamma\pm\epsilon_{0}}\right)  .
\]
Finally, for $\mathbb{T}_{n1}^{\left(  6\right)  }\left(  x\right)  $ we
observe that $\sqrt{k}\mathbb{T}_{n1}^{\left(  6\right)  }\left(  x\right)
=\int_{0}^{x^{-1/\gamma_{1}}}s^{\gamma_{1}/\gamma}\beta_{n}\left(
s^{-\gamma_{1}}\right)  ds,$ which by a change of variables meets the second
term in $\left(  \ref{TN1x}\right)  .$ Let us now consider the term
$\mathbb{T}_{n2}\left(  x\right)  .$ First, notice that%
\[
\int_{0}^{w}\frac{d\left(  H_{n}^{\left(  0\right)  }\left(  v\right)
-H^{\left(  0\right)  }\left(  v\right)  \right)  }{\overline{H}\left(
v\right)  }=-\int_{0}^{w}\frac{d\left(  \overline{H}_{n}^{\left(  0\right)
}\left(  v\right)  -\overline{H}^{\left(  0\right)  }\left(  v\right)
\right)  }{\overline{H}\left(  v\right)  },
\]
which, after an integration by parts, becomes%
\[
\left(  \overline{H}_{n}^{\left(  0\right)  }\left(  0\right)  -\overline
{H}^{\left(  0\right)  }\left(  0\right)  \right)  -\frac{\overline{H}%
_{n}^{\left(  0\right)  }\left(  w\right)  -\overline{H}^{\left(  0\right)
}\left(  w\right)  }{\overline{H}\left(  w\right)  }-\int_{0}^{w}\left(
\overline{H}_{n}^{\left(  0\right)  }\left(  v\right)  -\overline{H}^{\left(
0\right)  }\left(  v\right)  \right)  \frac{d\overline{H}\left(  v\right)
}{\overline{H}^{2}\left(  v\right)  }.
\]
It follows that $\mathbb{T}_{n2}\left(  x\right)  $ may be written into the
sum of%
\[
\mathbb{T}_{n2}^{\left(  1\right)  }\left(  x\right)  :=-\left(  \overline
{H}_{n}^{\left(  0\right)  }\left(  0\right)  -\overline{H}^{\left(  0\right)
}\left(  0\right)  \right)  \int_{xZ_{n-k:n}}^{\infty}\frac{\overline
{F}\left(  w\right)  }{\overline{F}\left(  Z_{n-k:n}\right)  }\frac
{d\overline{H}^{\left(  1\right)  }\left(  w\right)  }{\overline{H}\left(
w\right)  },
\]%
\[
\mathbb{T}_{n2}^{\left(  2\right)  }\left(  x\right)  :=\int_{xZ_{n-k:n}%
}^{\infty}\frac{\overline{H}_{n}^{\left(  0\right)  }\left(  w\right)
-\overline{H}^{\left(  0\right)  }\left(  w\right)  }{\overline{H}\left(
w\right)  }\frac{\overline{F}\left(  w\right)  }{\overline{F}\left(
Z_{n-k:n}\right)  }\frac{d\overline{H}^{\left(  1\right)  }\left(  w\right)
}{\overline{H}\left(  w\right)  },
\]
and%
\[
\mathbb{T}_{n2}^{\left(  3\right)  }\left(  x\right)  :=%
%TCIMACRO{\dint _{xZ_{n-k:n}}^{\infty}}%
%BeginExpansion
{\displaystyle\int_{xZ_{n-k:n}}^{\infty}}
%EndExpansion
\dfrac{\overline{F}\left(  w\right)  }{\overline{F}\left(  Z_{n-k:n}\right)
}\left\{
%TCIMACRO{\dint _{0}^{w}}%
%BeginExpansion
{\displaystyle\int_{0}^{w}}
%EndExpansion
\left(  \overline{H}_{n}^{\left(  0\right)  }\left(  v\right)  -\overline
{H}^{\left(  0\right)  }\left(  v\right)  \right)  \dfrac{d\overline{H}\left(
v\right)  }{\overline{H}^{2}\left(  v\right)  }\right\}  \dfrac{d\overline
{H}^{\left(  1\right)  }\left(  w\right)  }{\overline{H}\left(  w\right)  }.
\]
For the first term, we replace $d\overline{H}^{\left(  1\right)  }\left(
w\right)  $ and $\overline{H}\left(  w\right)  $ by $\overline{G}\left(
w\right)  d\overline{F}\left(  w\right)  $ and $\overline{G}\left(  w\right)
\overline{F}\left(  w\right)  $ respectively and we get%
\[
\mathbb{T}_{n2}^{\left(  1\right)  }\left(  x\right)  =\left(  \overline
{H}_{n}^{\left(  0\right)  }\left(  0\right)  -\overline{H}^{\left(  0\right)
}\left(  0\right)  \right)  \frac{\overline{F}\left(  xZ_{n-k:n}\right)
}{\overline{F}\left(  Z_{n-k:n}\right)  }.
\]
By using the routine manipulations of Potter's inequalities $\left(
\ref{Potter}\right)  $\ (applied to $\overline{F}),$ we obtain $\mathbb{T}%
_{n2}^{\left(  1\right)  }\left(  x\right)  =\left(  \overline{H}_{n}^{\left(
0\right)  }\left(  0\right)  -\overline{H}^{\left(  0\right)  }\left(
0\right)  \right)  O_{\mathbf{p}}\left(  x^{-1/\gamma_{1}\pm\epsilon_{0}%
}\right)  .$ On the other hand, by the central limit theorem, we have
$\overline{H}_{n}^{\left(  0\right)  }\left(  0\right)  -\overline{H}^{\left(
0\right)  }\left(  0\right)  =O_{\mathbf{p}}\left(  n^{-1/2}\right)  ,$ as
$n\rightarrow\infty,$ it follows that $\sqrt{k}\mathbb{T}_{n2}^{\left(
1\right)  }\left(  x\right)  =\sqrt{k/n}O_{\mathbf{p}}\left(  x^{-1/\gamma
_{1}\pm\epsilon_{0}}\right)  .$ Since $x^{-1/\gamma_{1}\pm\epsilon_{0}%
}=O\left(  x^{\left(  2\eta-p\right)  /\gamma\pm\epsilon_{0}}\right)  $ and
$k/n\rightarrow0,$ then $\sqrt{k}\mathbb{T}_{n2}^{\left(  1\right)  }\left(
x\right)  =o_{\mathbf{p}}\left(  x^{\left(  2\eta-p\right)  /\gamma\pm
\epsilon_{0}}\right)  .$ It is easy to verify that $\mathbb{T}_{n2}^{\left(
2\right)  }\left(  x\right)  $ may be rewritten into%
\[
\sqrt{k}\mathbb{T}_{n2}^{\left(  2\right)  }\left(  x\right)  =\int
_{x}^{\infty}\frac{\widetilde{\beta}_{n}\left(  w\right)  }{\overline
{H}\left(  wZ_{n-k:n}\right)  }\frac{d\overline{F}\left(  wZ_{n-k:n}\right)
}{\overline{F}\left(  Z_{n-k:n}\right)  }.
\]
By using similar arguments as those used for $\mathbb{T}_{n1}\left(  x\right)
,$ we show that
\[
\sqrt{k}\mathbb{T}_{n2}^{\left(  2\right)  }\left(  x\right)  =-\frac
{p}{\gamma}\int_{x}^{\infty}w^{1/\gamma_{2}-1}\widetilde{\beta}_{n}\left(
w\right)  dw+o_{\mathbf{p}}\left(  x^{\left(  2\eta-p\right)  /\gamma
\pm\epsilon_{0}}\right)  ,
\]
therefore we omit the details. As for the third term $\mathbb{T}_{n2}^{\left(
3\right)  }\left(  x\right)  ,$ we have%
\[
\mathbb{T}_{n2}^{\left(  3\right)  }\left(  x\right)  =\int_{x}^{\infty
}\left\{  \int_{0}^{wZ_{n-k:n}}\left(  \overline{H}_{n}^{\left(  0\right)
}\left(  v\right)  -\overline{H}^{\left(  0\right)  }\left(  v\right)
\right)  \frac{d\overline{H}\left(  v\right)  }{\overline{H}^{2}\left(
v\right)  }\right\}  d\frac{\overline{F}\left(  wZ_{n-k:n}\right)  }%
{\overline{F}\left(  Z_{n-k:n}\right)  }.
\]
After a change of variables and an integration by parts in the first integral,
we apply Potter's inequalities $\left(  \ref{Potter}\right)  $ to
$\overline{F},$ to show that%
\[
\sqrt{k}\mathbb{T}_{n2}^{\left(  3\right)  }\left(  x\right)  =-\frac
{p}{\gamma}\int_{x}^{\infty}w^{-1/\gamma_{1}-1}\left\{  \int_{0}^{wZ_{n-k:n}%
}\sqrt{k}\left(  \overline{H}_{n}^{\left(  0\right)  }\left(  v\right)
-\overline{H}^{\left(  0\right)  }\left(  v\right)  \right)  \frac
{d\overline{H}\left(  v\right)  }{\overline{H}^{2}\left(  v\right)  }\right\}
dw+\sqrt{k}R_{n3}\left(  x\right)  ,
\]
where%
\[
\sqrt{k}R_{n3}\left(  x\right)  :=o_{\mathbb{P}}\left(  x^{-1/\gamma_{1}%
\pm\epsilon_{0}}\right)  \int_{x}^{\infty}w^{-1/\gamma_{1}-1}\left\{  \int
_{0}^{wZ_{n-k:n}}\sqrt{k}\left\vert \overline{H}_{n}^{\left(  0\right)
}\left(  v\right)  -\overline{H}^{\left(  0\right)  }\left(  v\right)
\right\vert \frac{dH\left(  v\right)  }{\overline{H}^{2}\left(  v\right)
}\right\}  dw.
\]
By using once again assertion $\left(  ii\right)  $ in Proposition
\ref{Propo1} with the fact that $\overline{H}^{\left(  0\right)  }%
<\overline{H}$ together with Potter's inequalities $\left(  \ref{Potter}%
\right)  $\ routine manipulations, we readily show that $\sqrt{k}R_{n3}\left(
x\right)  =o_{\mathbf{p}}\left(  x^{\left(  2\eta-p\right)  /\gamma\pm
\epsilon_{0}}\right)  ,$ uniformly on $x\geq p^{\gamma},$ therefore we omit
the details. Recall that $\gamma_{1}=\gamma/p$ and let us rewrite the first
term in $\mathbb{T}_{n2}^{\left(  3\right)  }\left(  x\right)  $ into%
\[
\int_{x}^{\infty}\left\{  \int_{0}^{wZ_{n-k:n}}\sqrt{k}\left(  \overline
{H}_{n}^{\left(  0\right)  }\left(  v\right)  -\overline{H}^{\left(  0\right)
}\left(  v\right)  \right)  \frac{d\overline{H}\left(  v\right)  }%
{\overline{H}^{2}\left(  v\right)  }\right\}  dw^{-1/\gamma_{1}},
\]
which, by an integration by parts, becomes%
\begin{align*}
&  -\int_{0}^{xZ_{n-k:n}}\sqrt{k}\left(  \overline{H}_{n}^{\left(  0\right)
}\left(  v\right)  -\overline{H}^{\left(  0\right)  }\left(  v\right)
\right)  \frac{d\overline{H}\left(  v\right)  }{\overline{H}^{2}\left(
v\right)  }\\
&  -\int_{x}^{\infty}w^{-1/\gamma_{1}}\sqrt{k}\left(  \overline{H}%
_{n}^{\left(  0\right)  }\left(  wZ_{n-k:n}\right)  -\overline{H}^{\left(
0\right)  }\left(  wZ_{n-k:n}\right)  \right)  \frac{d\overline{H}\left(
wZ_{n-k:n}\right)  }{\overline{H}^{2}\left(  wZ_{n-k:n}\right)  }.
\end{align*}
The first term above may be rewritten into $-\dfrac{k}{n}%
%TCIMACRO{\dint _{0}^{x}}%
%BeginExpansion
{\displaystyle\int_{0}^{x}}
%EndExpansion
\widetilde{\beta}_{n}\left(  w\right)  d\overline{H}\left(  wZ_{n-k:n}\right)
/\overline{H}^{2}\left(  wZ_{n-k:n}\right)  $ and by similar arguments as
those used for $\mathbb{T}_{n1}\left(  x\right)  ,$ we show that the second
term equals $\dfrac{1}{\gamma}\int_{x}^{\infty}w^{1/\gamma_{2}-1}%
\widetilde{\beta}_{n}\left(  w\right)  dw+o_{\mathbf{p}}\left(  x^{\left(
2\eta-p\right)  /\gamma\pm\epsilon_{0}}\right)  $ and thus we have%
\[
\sqrt{k}\mathbb{T}_{n2}^{\left(  3\right)  }\left(  x\right)  =-\frac{k}%
{n}\int_{0}^{x}\widetilde{\beta}_{n}\left(  w\right)  \frac{d\overline
{H}\left(  wZ_{n-k:n}\right)  }{\overline{H}^{2}\left(  wZ_{n-k:n}\right)
}+\frac{1}{\gamma}\int_{x}^{\infty}w^{1/\gamma_{2}-1}\widetilde{\beta}%
_{n}\left(  w\right)  dw+o_{\mathbf{p}}\left(  x^{\left(  2\eta-p\right)
/\gamma\pm\epsilon_{0}}\right)  .
\]
Consequently, we have%
\[
\sqrt{k}\mathbb{T}_{n2}\left(  x\right)  =-x^{-1/\gamma_{1}}\frac{k}{n}%
\int_{0}^{x}\widetilde{\beta}_{n}\left(  w\right)  \frac{d\overline{H}\left(
wZ_{n-k:n}\right)  }{\overline{H}^{2}\left(  wZ_{n-k:n}\right)  }+\frac
{q}{\gamma}\int_{x}^{\infty}w^{1/\gamma_{2}-1}\widetilde{\beta}_{n}\left(
w\right)  dw+o_{\mathbf{p}}\left(  x^{\left(  2\eta-p\right)  /\gamma
\pm\epsilon_{0}}\right)  .
\]
For the term $\mathbb{T}_{n3}\left(  x\right)  ,$ routine manipulations lead
to%
\begin{align*}
\sqrt{k}\mathbb{T}_{n3}\left(  x\right)   &  =x^{-1/\gamma_{1}}\frac{k}{n}%
\int_{0}^{x}\left(  \beta_{n}\left(  w\right)  +\widetilde{\beta}_{n}\left(
w\right)  \right)  \frac{d\overline{H}^{\left(  0\right)  }\left(
wZ_{n-k:n}\right)  }{\overline{H}^{2}\left(  wZ_{n-k:n}\right)  }\\
&  -\frac{q}{\gamma}\int_{x}^{\infty}w^{-1/\gamma_{2}-1}\left(  \beta
_{n}\left(  w\right)  +\widetilde{\beta}_{n}\left(  w\right)  \right)
dw+\sqrt{k}R_{n4}\left(  x\right)  ,
\end{align*}
where%
\begin{align*}
\sqrt{k}R_{n4}\left(  x\right)   &  :=o_{\mathbf{p}}\left(  1\right)
x^{-1/\gamma_{1}\pm\epsilon_{0}}\dfrac{k}{n}%
%TCIMACRO{\dint _{0}^{x}}%
%BeginExpansion
{\displaystyle\int_{0}^{x}}
%EndExpansion
\left(  \left\vert \beta_{n}\left(  w\right)  \right\vert +\left\vert
\widetilde{\beta}_{n}\left(  w\right)  \right\vert \right)  \dfrac{dH^{\left(
0\right)  }\left(  wZ_{n-k:n}\right)  }{\overline{H}^{2}\left(  wZ_{n-k:n}%
\right)  }\\
&  +\frac{q}{\gamma}\int_{x}^{\infty}w^{-1/\gamma_{2}-1}\left(  \left\vert
\beta_{n}\left(  w\right)  \right\vert +\left\vert \widetilde{\beta}%
_{n}\left(  w\right)  \right\vert \right)  dw.
\end{align*}
By a similar treatment as that of $\sqrt{k}R_{n3}\left(  x\right)  ,$ we get
$\sqrt{k}R_{n4}\left(  x\right)  =o_{\mathbf{p}}\left(  x^{\left(
2\eta-p\right)  /\gamma\pm\epsilon_{0}}\right)  .$ By substituting the results
obtained above, for the terms $\mathbb{T}_{nj}\left(  x\right)  ,$ $j=1,2,3,$
in equation $\left(  \ref{MN1-decompos}\right)  ,$ we end up with%
\begin{align*}
\sqrt{k}\mathbb{M}_{n1}\left(  x\right)   &  =x^{1/\gamma_{2}}\beta_{n}\left(
x\right)  +x^{-1/\gamma_{1}}\dfrac{k}{n}\left\{
%TCIMACRO{\dint _{0}^{x}}%
%BeginExpansion
{\displaystyle\int_{0}^{x}}
%EndExpansion
\beta_{n}\left(  w\right)  \dfrac{d\overline{H}^{\left(  0\right)  }\left(
wZ_{n-k:n}\right)  }{\overline{H}^{2}\left(  wZ_{n-k:n}\right)  }\right. \\
&  \left.  -\int_{0}^{x}\widetilde{\beta}_{n}\left(  w\right)  \frac
{d\overline{H}^{\left(  1\right)  }\left(  wZ_{n-k:n}\right)  }{\overline
{H}^{2}\left(  wZ_{n-k:n}\right)  }\right\}  +o_{\mathbf{p}}\left(  x^{\left(
2\eta-p\right)  /\gamma\pm\epsilon_{0}}\right)  .
\end{align*}
The asymptotic negligibility (in probability) of $\sqrt{k}\mathbb{M}%
_{n2}\left(  x\right)  $ is readily obtained. Indeed, note that we have
$\sqrt{k}\mathbb{M}_{n1}\left(  x\right)  =O_{\mathbf{p}}\left(  x^{\left(
2\eta-p\right)  /\gamma\pm\epsilon_{0}}\right)  $ and from Theorem 2\textbf{
}in \cite{Cs96}, we infer that $\overline{F}\left(  Z_{n-k:n}\right)
/\overline{F}_{n}\left(  Z_{n-k:n}\right)  -1=O_{\mathbf{p}}\left(
k^{-1/2}\right)  .$ This means that $\sqrt{k}\mathbb{M}_{n2}\left(  x\right)
=o_{\mathbf{p}}\left(  x^{\left(  2\eta-p\right)  /\gamma\pm\epsilon_{0}%
}\right)  $ (because $k\rightarrow\infty).$ Recall now that%
\[
\sqrt{k}\mathbb{M}_{n3}\left(  x\right)  =-\frac{\overline{F}\left(
xZ_{n-k:n}\right)  }{\overline{F}\left(  Z_{n-k:n}\right)  }\sqrt{k}%
\mathbb{M}_{n1}\left(  1\right)  ,
\]
which, by applying Potter's inequalities $\left(  \ref{Potter}\right)  $ to
$\overline{F},$ yields that $\sqrt{k}\mathbb{M}_{n3}\left(  x\right)  $ is
equal to%
\begin{align*}
&  -x^{1/\gamma_{1}}\beta_{n}\left(  x\right)  -x^{-1/\gamma_{1}}\frac{k}%
{n}\left\{  \int_{0}^{1}\beta_{n}\left(  w\right)  \frac{d\overline
{H}^{\left(  0\right)  }\left(  wZ_{n-k:n}\right)  }{\overline{H}^{2}\left(
wZ_{n-k:n}\right)  }\right. \\
&  \ \ \ \ \ \ \ \ \ \ \ \ \ \left.  -\int_{0}^{x}\widetilde{\beta}_{n}\left(
w\right)  \frac{d\overline{H}^{\left(  1\right)  }\left(  wZ_{n-k:n}\right)
}{\overline{H}^{2}\left(  wZ_{n-k:n}\right)  }\right\}  +o_{\mathbf{p}}\left(
x^{\left(  2\eta-p\right)  /\gamma\pm\epsilon_{0}}\right)  .
\end{align*}
Therefore%
\[
\sum_{i=1}^{3}\sqrt{k}\mathbb{M}_{ni}\left(  x\right)  =x^{1/\gamma_{2}}%
\beta_{n}\left(  x\right)  -x^{-1/\gamma_{1}}\beta_{n}\left(  1\right)
+\mathcal{D}_{n}\left(  \beta_{n},\widetilde{\beta}_{n};x\right)
+o_{\mathbf{p}}\left(  x^{\left(  2\eta-p\right)  /\gamma\pm\epsilon_{0}%
}\right)  ,
\]
where%
\[
\mathcal{D}_{n}\left(  \beta_{n},\widetilde{\beta}_{n};x\right)
:=x^{-1/\gamma_{1}}\frac{k}{n}\left\{  \int_{1}^{x}\beta_{n}\left(  w\right)
\frac{d\overline{H}^{\left(  0\right)  }\left(  wZ_{n-k:n}\right)  }%
{\overline{H}^{2}\left(  wZ_{n-k:n}\right)  }-\int_{1}^{x}\widetilde{\beta
}_{n}\left(  w\right)  \frac{d\overline{H}^{\left(  1\right)  }\left(
wZ_{n-k:n}\right)  }{\overline{H}^{2}\left(  wZ_{n-k:n}\right)  }\right\}  ,
\]
which, by routine manipulations as the above, is shown to be equal to%
\[
x^{-1/\gamma_{1}}\left\{  \gamma_{1}^{-1}\int_{1}^{x}w^{1/\gamma-1}%
\widetilde{\beta}_{n}\left(  w\right)  dw-\gamma_{1}^{-1}\int_{1}%
^{x}w^{1/\gamma-1}\beta_{n}\left(  w\right)  dw\right\}  +o_{\mathbf{p}%
}\left(  x^{\left(  2\eta-p\right)  /\gamma\pm\epsilon_{0}}\right)  .
\]
Therefore, we have%
\begin{align*}
\sum_{i=1}^{3}\sqrt{k}\mathbb{M}_{ni}\left(  x\right)   &  =x^{1/\gamma_{2}%
}\beta_{n}\left(  x\right)  +o_{\mathbf{p}}\left(  x^{\left(  2\eta-p\right)
/\gamma\pm\epsilon_{0}}\right) \\
&  -x^{-1/\gamma_{1}}\left\{  \beta_{n}\left(  1\right)  -\gamma_{1}^{-1}%
\int_{1}^{x}w^{1/\gamma-1}\widetilde{\beta}_{n}\left(  w\right)  dw-\gamma
_{1}^{-1}\int_{1}^{x}w^{1/\gamma-1}\beta_{n}\left(  w\right)  dw\right\}  .
\end{align*}
We are now in position to apply the well-known Gaussian approximation $\left(
\ref{approx2}\right)  $ to get%
\begin{align*}
&  \sum_{i=1}^{3}\sqrt{k}\mathbb{M}_{ni}\left(  x\right)  =x^{1/\gamma_{2}%
}\sqrt{\frac{n}{k}}\mathbf{B}_{n}\left(  x\right)  -x^{-1/\gamma_{1}}%
\sqrt{\frac{n}{k}}\mathbf{B}_{n}\left(  1\right) \\
&  +x^{-1/\gamma_{1}}\left\{  \int_{1}^{x}\frac{\mathbf{B}_{n}^{\ast}\left(
w\right)  }{\overline{H}^{2}\left(  w\right)  }dH^{1}\left(  w\right)
-\int_{0}^{1}\frac{\mathbf{B}_{n}\left(  w\right)  }{\overline{H}^{2}\left(
w\right)  }dH\left(  w\right)  \right\}  +o_{\mathbf{p}}\left(  x^{\left(
2\eta-p\right)  /\gamma\pm\epsilon_{0}}\right)  ,
\end{align*}
where $\mathbf{B}_{n}\left(  w\right)  $ and $\mathbf{B}_{n}^{\ast}\left(
w\right)  $ are two Gaussian processes defined by%
\[
\mathbf{B}_{n}\left(  w\right)  :=B_{n}\left(  \theta\right)  -B_{n}\left(
\theta-\overline{H}^{\left(  1\right)  }\left(  wZ_{n-k:n}\right)  \right)
\text{ and }\widetilde{\mathbf{B}}_{n}\left(  w\right)  :=-B_{n}\left(
1-\overline{H}^{\left(  0\right)  }\left(  wZ_{n-k:n}\right)  \right)  .
\]
By similar arguments as those used in Lemma 5.2 in \cite{BMN-2015}, we end up
with
\begin{align*}
&
%TCIMACRO{\dsum _{i=1}^{3}}%
%BeginExpansion
{\displaystyle\sum_{i=1}^{3}}
%EndExpansion
\sqrt{k}\mathbb{M}_{ni}\left(  x\right) \\
&  =x^{1/\gamma_{2}}\sqrt{\dfrac{n}{k}}\mathbb{B}_{n}\left(  \dfrac{k}%
{n}x^{-1/\gamma}\right)  -x^{-1/\gamma_{1}}\sqrt{\dfrac{n}{k}}\mathbb{B}%
_{n}\left(  \dfrac{k}{n}\right) \\
&  +\frac{x^{-1/\gamma_{1}}}{\gamma}\sqrt{\frac{n}{k}}\int_{1}^{x}%
u^{1/\gamma-1}\left(  p\widetilde{\mathbb{B}}_{n}\left(  \frac{k}%
{n}u^{-1/\gamma}\right)  -q\mathbb{B}_{n}\left(  \frac{k}{n}u^{-1/\gamma
}\right)  \right)  du+o_{\mathbf{p}}\left(  x^{\left(  2\eta-p\right)
/\gamma\pm\epsilon_{0}}\right)  ,
\end{align*}
where $\mathbb{B}_{n}\left(  s\right)  $ and $\widetilde{\mathbb{B}}%
_{n}\left(  s\right)  ,$ $0<s<1,$ are sequences of centred Gaussian processes
defined by $\mathbb{B}_{n}\left(  s\right)  :=B_{n}\left(  \theta\right)
-B_{n}\left(  \theta-ps\right)  $ and $\widetilde{\mathbb{B}}_{n}\left(
s\right)  :=-B_{n}\left(  1-qs\right)  .$ Let $\left\{  W_{n}\left(  t\right)
;\text{ }0\leq t\leq1\right\}  $ be a sequence of Weiner processes defined on
$\left(  \Omega,\mathcal{A},\mathbb{P}\right)  $ so that
\begin{equation}
\left\{  B_{n}\left(  t\right)  ;0\leq t\leq1\right\}  \overset{\mathcal{D}%
}{=}\left\{  W_{n}\left(  t\right)  -tW_{n}\left(  1\right)  ;0\leq
t\leq1\right\}  . \label{W}%
\end{equation}
It is easy to verify that $\sum_{i=1}^{3}\sqrt{k}\mathbb{M}_{ni}\left(
x\right)  =J_{n}\left(  x\right)  +o_{\mathbf{p}}\left(  x^{\left(
2\eta-p\right)  /\gamma\pm\epsilon_{0}}\right)  ,$ which is exactly $\left(
\ref{aproxima-MI}\right)  .$ Finally, we take care of the term $\mathbb{M}%
_{n4}\left(  x\right)  .$ To this end, we apply the uniform inequality of
second-order regularly varying functions to $\overline{F}$ (see, e.g., the
bottom of page 161 in \cite{deHF06}), to write%
\[
\frac{\overline{F}\left(  xZ_{n-k:n}\right)  }{\overline{F}\left(
Z_{n-k:n}\right)  }-x^{-1/\gamma_{1}}=\left(  1+o_{\mathbf{p}}\left(
1\right)  \right)  A_{1}\left(  Z_{n-k:n}\right)  x^{-1/\gamma_{1}}%
\dfrac{x^{\tau/\gamma_{1}}-1}{\tau/\gamma_{1}},
\]
and%
\[
\frac{\overline{F}\left(  xh\right)  }{\overline{F}\left(  h\right)
}-x^{-1/\gamma_{1}}=\left(  1+o\left(  1\right)  \right)  A_{1}\left(
h\right)  x^{-1/\gamma_{1}}\dfrac{x^{\tau/\gamma_{1}}-1}{\tau/\gamma_{1}},
\]
uniformly on $x\geq p^{\gamma}.$ Since $A_{1}$ is regularly varying at
infinity (with index $\tau\gamma_{1})$ and $Z_{n-k:n}/h\overset{\mathbf{p}%
}{\rightarrow}1,$ it follows that $A_{1}\left(  Z_{n-k:n}\right)  =\left(
1+o_{\mathbf{p}}\left(  1\right)  \right)  A_{1}\left(  h\right)  ,,$
therefore
\[
\sqrt{k}\mathbb{M}_{n4}\left(  x\right)  =\left(  1+o_{\mathbf{p}}\left(
1\right)  \right)  x^{-1/\gamma_{1}}\dfrac{x^{\tau/\gamma_{1}}-1}{\tau
/\gamma_{1}}\sqrt{k}A_{1}\left(  h\right)  ,\text{ as }n\rightarrow\infty.
\]
By assumption we have $\sqrt{k}A_{1}\left(  h\right)  =O\left(  1\right)  ,$
then%
\[
D_{n}\left(  x\right)  -J_{n}\left(  x\right)  -x^{-1/\gamma_{1}}%
\dfrac{x^{\tau/\gamma_{1}}-1}{\tau/\gamma_{1}}\sqrt{k}A_{1}\left(  h\right)
=o_{\mathbf{p}}\left(  x^{\left(  2\eta-p\right)  /\gamma\pm\epsilon_{0}%
}\right)  .
\]
Let $\eta_{0}>0$ such that $1/4<\eta<\eta_{0}<p/2,$ then
\[
x^{\left(  p-2\eta_{0}\right)  /\gamma}\left\{  D_{n}\left(  x\right)
-J_{n}\left(  x\right)  -x^{-1/\gamma_{1}}\dfrac{x^{\tau/\gamma_{1}}-1}%
{\tau/\gamma_{1}}\sqrt{k}A_{1}\left(  h\right)  \right\}  =o_{\mathbf{p}%
}\left(  x^{2\left(  \eta-\eta_{0}\right)  /\gamma\pm\epsilon_{0}}\right)  .
\]
Now, we choose $\epsilon_{0}$ sufficiently small so that $\left(  \eta
-\eta_{0}\right)  /\gamma+\epsilon_{0}<0.$ Since $x\geq p^{\gamma}>0,$ then
$o_{\mathbf{p}}\left(  x^{2\left(  \eta-\eta_{0}\right)  /\gamma\pm
\epsilon_{0}}\right)  =o_{\mathbf{p}}\left(  1\right)  ,$ hence for
$1/4<\eta_{0}<p/4$%
\[
x^{\left(  p-2\eta_{0}\right)  /\gamma}\left\{  D_{n}\left(  x\right)
-J_{n}\left(  x\right)  -x^{-1/\gamma_{1}}\dfrac{x^{\tau/\gamma_{1}}-1}%
{\tau/\gamma_{1}}\sqrt{k}A_{1}\left(  h\right)  \right\}  =o_{\mathbf{p}%
}\left(  1\right)  ,
\]
uniformly on $x\geq p^{\gamma}.$ To achieve the proof it suffices to replace
$\gamma$ by $p\gamma_{1}$ and choose $\epsilon=p-2\eta_{0}$ so that
$0<\epsilon<1/2,$ as sought. \hfill$\Box$

\subsection{Proof of Theorem \ref{Theorem2}}

For the consistency of $\widehat{\gamma}_{1},$ we make an integration by parts
and a change of variables in equation $\left(  \ref{gchap}\right)  $ to get%
\[
\widehat{\gamma}_{1}=\int_{1}^{\infty}x^{-1}\frac{\overline{F}_{n}\left(
xZ_{n-k:n}\right)  }{\overline{F}_{n}\left(  Z_{n-k:n}\right)  }dx,
\]
which may be decomposed into the sum of $I_{1n}:=\int_{1}^{\infty}%
x^{-1}\overline{F}\left(  xZ_{n-k:n}\right)  /\overline{F}\left(
Z_{n-k:n}\right)  dx$ and $I_{2n}:=\int_{1}^{\infty}x^{-1}\sum_{i=1}%
^{3}\mathbb{M}_{ni}\left(  x\right)  dx.$ By using the regular variation of
$\overline{F}$ $\left(  \ref{R-F}\right)  $ and the corresponding Potter's
inequalities $\left(  \ref{Potter}\right)  ,$ we get $I_{1n}\overset
{\mathbf{P}}{\rightarrow}\gamma_{1}$ as $n\rightarrow\infty.$ Then, we just
need to show that $I_{2n}$ tends to zero in probability. From $\left(
\ref{aproxima-MI}\right)  $ we have%
\[
I_{2n}=\frac{1}{\sqrt{k}}\int_{1}^{\infty}x^{-1}J_{n}\left(  x\right)
dx+\frac{o_{\mathbf{p}}\left(  1\right)  }{\sqrt{k}}\int_{1}^{\infty
}x^{-1-\epsilon/\left(  p\gamma_{1}\right)  }dx,
\]
where the second integral above is finite and therefore the second term of
$I_{2n}$ is negligible in probability. On the other hand, we have $\int
_{1}^{\infty}x^{-1}J_{n}\left(  x\right)  dx=\int_{1}^{\infty}x^{-1}\left(
J_{1n}\left(  x\right)  +J_{2n}\left(  x\right)  \right)  dx,$ where
$J_{1n}\left(  x\right)  $ and $J_{2n}\left(  x\right)  $ are the two centred
Gaussian processes given in Theorem \ref{Theorem1}. After some elementary but
tedious manipulations of integral calculus, we obtain%
\begin{align}
&  \int_{1}^{\infty}x^{-1}J_{n}\left(  x\right)  dx\label{rep}\\
&  =\gamma\sqrt{\frac{n}{k}}\int_{0}^{1}s^{-q-1}\left(  \mathbf{W}%
_{n,2}\left(  \frac{k}{n}s\right)  +\left(  1-\frac{q}{p}\right)
\mathbf{W}_{n,1}\left(  \frac{k}{n}s\right)  \right)  ds-\gamma_{1}\sqrt
{\frac{n}{k}}\mathbf{W}_{n,1}\left(  \frac{k}{n}\right)  .\nonumber
\end{align}
Since $\left\{  W_{n}\left(  s\right)  ,\text{ }0\leq s\leq1\right\}  $ is a
sequence Weiner processes, we may readily show that%
\[
\sqrt{\dfrac{n}{k}}\mathbf{E}\left\vert \mathbf{W}_{n,1}\left(  \dfrac{k}%
{n}s\right)  \right\vert \leq\left(  ps\right)  ^{1/2}\text{ and }\sqrt
{\dfrac{n}{k}}\mathbf{E}\left\vert \mathbf{W}_{n,2}\left(  \dfrac{k}%
{n}s\right)  \right\vert \leq\left(  qs\right)  ^{1/2}.
\]
But $\gamma_{1}<\gamma_{2},$ hence $0<q<1/2$ and thus it is easy to verify
that $\mathbf{E}\left\vert \int_{1}^{\infty}x^{-1}J_{n}\left(  x\right)
dx\right\vert <\infty.$ This yields that $I_{2n}\overset{_{\mathbf{P}}%
}{\rightarrow}0$ when $n\rightarrow\infty$ (because $1/\sqrt{k}\rightarrow0),$
as sought. As for the Gaussian representation result, we write $\sqrt
{k}\left(  \widehat{\gamma}_{1}-\gamma_{1}\right)  =\int_{1}^{\infty}%
x^{-1}D_{n}\left(  x\right)  dx,$ then by applying Theorem $\ref{Theorem1}$
together with the representation $\left(  \ref{rep}\right)  $ and the
assumption, we have $\sqrt{k}A_{1}\left(  h\right)  \rightarrow\lambda,$ we
get $\sqrt{k}\left(  \widehat{\gamma}_{1}-\gamma_{1}\right)  =C_{1n}%
+C_{2n}+C_{3n}+\dfrac{\lambda}{1-\tau}+o_{\mathbf{p}}\left(  1\right)  ,$
where
\[
C_{1n}:=\gamma\sqrt{\frac{n}{k}}\int_{0}^{1}s^{-q-1}\mathbf{W}_{n,2}\left(
\frac{k}{n}s\right)  ds,\text{ }C_{2n}:=\gamma\left(  1-\frac{q}{p}\right)
\sqrt{\frac{n}{k}}\int_{0}^{1}s^{-q-1}\mathbf{W}_{n,1}\left(  \frac{k}%
{n}s\right)  ds
\]
and $C_{3n}:=-\gamma_{1}\sqrt{\dfrac{n}{k}}\mathbf{W}_{n,1}\left(  \dfrac
{k}{n}\right)  .$ The computation of the limit of $\mathbf{E}\left[
C_{1n}+C_{2n}+C_{3n}\right]  ^{2}$ gives%
\[
\frac{2q\gamma^{2}}{2q^{2}-3q+1}+\left(  1-\frac{q}{p}\right)  ^{2}%
\frac{2p\gamma^{2}}{2q^{2}-3q+1}+\frac{\gamma^{2}}{p}-2\frac{\gamma^{2}}%
{p}\left(  1-\frac{q}{p}\right)  =\frac{\gamma^{2}}{p\left(  2p-1\right)  },
\]
which by substituting $p\gamma_{1}$ for $\gamma$ completes the proof.\hfill
$\Box\medskip$

\noindent\textbf{Concluding notes\medskip}

\noindent On the basis of Nelson-Aalen nonparametric estimator, we introduced
a product-limit process for the tail of a heavy-tailed distribution of
randomly right-censored data. The Gaussian approximation of this process
proved to be a very useful tool in achieving the asymptotic normality of the
estimators of tail indices and related statistics. Furthermore, we defined a
Hill-type estimator for the extreme value index and determined its limiting
Gaussian distribution. Intensive simulations show that the latter outperforms
the already existing ones, with respect to bias and MSE. It is noteworthy that
the asymptotic behavior of the newly proposed estimator is only assessed under
the second-order condition of regular variation of the underlying distribution
tail, in contrast to the unfamiliar assumptions of \cite{WW2014} and
\cite{EnFG08}. This represents the main big improvement brought in this paper.
Our approach will have fruitful consequences and open interesting paths in the
statistical analysis of extremes with incomplete data. The generalization of
this approach to the whole range of maximum domains of attraction would make a
very good topic for a future work.

\section{\textbf{Appendix\label{sec6}}}

\noindent The following Proposition provides useful results with regards to
the uniform empirical and quantile functions $\mathbb{U}_{n}\left(  t\right)
$ and $\mathbb{V}_{n}\left(  t\right)  $ respectively.

\begin{proposition}
\label{Propo1}$\left(  i\right)  $ For $n\geq1$ and $0<a<1,$ we have%
\[
\sup_{a/n\leq t\leq1}\left(  t/\mathbb{U}_{n}\left(  t\right)  \right)
^{\pm1}=O_{\mathbf{p}}\left(  1\right)  =\sup_{a/n\leq t\leq1}\left(
t/\mathbb{V}_{n}\left(  t\right)  \right)  ^{\pm1}.
\]
$\left(  ii\right)  $ For $n\geq1,$ $0<b<1$ and $0<\eta<1/2,$ we have%
\[
\sup_{0<t\leq1}n^{\eta}\left\vert \mathbb{U}_{n}\left(  t\right)
-t\right\vert /t^{1-\eta}=O_{\mathbf{p}}\left(  1\right)  =\sup_{b/n\leq
t\leq1}n^{\eta}\left\vert \mathbb{V}_{n}\left(  t\right)  -t\right\vert
/t^{1-\eta}.
\]
$\left(  iii\right)  $ Let $\varphi$ be a regularly varying function at
infinity with index $\alpha$ and let $0<a_{n}<b_{n}<1$ be such that
$na_{n}=O\left(  1\right)  $ and $b_{n}\downarrow0.$ Then%
\[
\sup_{a_{n}\leq t\leq b_{n}}\left(  \varphi\left(  t\right)  /\varphi\left(
\mathbb{V}_{n}\left(  t\right)  \right)  \right)  =O_{\mathbf{p}}\left(
1\right)  .
\]

\end{proposition}

\begin{proof}
The proofs of assertion $\left(  i\right)  $ and the first result of assertion
$\left(  ii\right)  $ may be found in \cite{SW86} in pages 415 (assertions 5-8
) and 425 (assertion 16) respectively. The second result of assertion $\left(
ii\right)  $ is proved by using the first results of both assertions $\left(
i\right)  $ and $\left(  ii\right)  .$ For the third assertion $\left(
iii\right)  ,$ it suffices to apply Potter's inequalities $\left(
\ref{Potter}\right)  $ to function $\varphi$ together with assertion $\left(
i\right)  $ corresponding to $\mathbb{V}_{n}\left(  t\right)  .$
\end{proof}

\begin{lemma}
\label{Lemma1}Let $I_{n}\left(  x\right)  $ be the interval of endpoints
$\overline{H}^{\left(  1\right)  }\left(  xZ_{n-k:n}\right)  $ and
$\overline{H}_{n}^{\left(  1\right)  }\left(  xZ_{n-k:n}\right)  .$ Then%
\[
\sup_{x\geq p^{\gamma}}\sup_{t\in I_{n}\left(  x\right)  }\left(
\mathcal{L}\left(  t\right)  /\mathcal{L}\left(  \mathbb{V}_{n}\left(
t\right)  \right)  \right)  =O_{\mathbf{p}}\left(  1\right)  ,\text{ as
}n\rightarrow\infty,
\]
where $\mathcal{L}\left(  s\right)  :=\overline{F}\left(  Q^{\left(  1\right)
}\left(  s\right)  \right)  /\overline{H}\left(  Q^{\left(  1\right)  }\left(
s\right)  \right)  ,$ with $Q^{\left(  1\right)  }$ being the generalized
inverse of $\overline{H}^{\left(  1\right)  }.$
\end{lemma}

\begin{proof}
First, we show that, for any small $0<\xi<1,$ there exists a constant $c>0,$
such that $I_{n}\left(  x\right)  $ is included in the half-open interval
$\left[  \left(  1-\xi\right)  px^{-1/\gamma}/n,ck/n\right)  ,$ with
probability close to $1,$ as $n\rightarrow\infty.$ Indeed, following similar
arguments as those used in the proof of part $(i)$ in Lemma 4.1 in
\cite{BMN-2015}, we infer that for all large $n$
\begin{equation}
\overline{H}^{\left(  1\right)  }\left(  xZ_{n-k:n}\right)  =\left(
1+o_{\mathbf{p}}\left(  1\right)  \right)  px^{-1/\gamma}\overline{H}\left(
Z_{n-k:n}\right)  , \label{equi}%
\end{equation}
uniformly on $x\geq p^{\gamma}.$ We have $\overline{H}\left(  Z_{n-k:n}%
\right)  =\left(  1+o_{\mathbf{p}}\left(  1\right)  \right)  k/n,$ which
implies that
\[
\mathbf{P}\left\{  \left(  1-\xi\right)  px^{-1/\gamma}\dfrac{k}{n}%
<\overline{H}^{\left(  1\right)  }\left(  xZ_{n-k:n}\right)  <\left(
1+\xi\right)  px^{-1/\gamma}\dfrac{k}{n}\right\}  \rightarrow1,\text{ as
}n\rightarrow\infty.
\]
Since $k>1$ and $x\geq p^{\gamma},$ then
\begin{equation}
\mathbf{P}\left\{  \dfrac{\left(  1-\xi\right)  px^{-1/\gamma}}{n}%
<\overline{H}^{\left(  1\right)  }\left(  xZ_{n-k:n}\right)  <\left(
1+\xi\right)  \dfrac{k}{n}\right\}  \rightarrow1,\text{ as }n\rightarrow
\infty. \label{P1}%
\end{equation}
On the other hand, it is obvious that for any $x\geq p^{\gamma},$
$\mathbf{P}\left(  \overline{H}_{n}^{\left(  1\right)  }\left(  xZ_{n-k:n}%
\right)  \geq1/n\right)  =1$ and $\overline{H}_{n}^{\left(  1\right)  }\left(
xZ_{n-k:n}\right)  =\mathbb{U}_{n}\left(  \overline{H}^{\left(  1\right)
}\left(  xZ_{n-k:n}\right)  \right)  ,$ a.s. Then, in view of assertion $(i)$
in Proposition $\left(  \ref{Propo1}\right)  ,$ we have $\overline{H}%
_{n}^{\left(  1\right)  }\left(  xZ_{n-k:n}\right)  =O_{\mathbf{p}}\left(
\overline{H}^{\left(  1\right)  }\left(  xZ_{n-k:n}\right)  \right)  ,$ it
follows, from $\left(  \ref{equi}\right)  ,$ that $\overline{H}_{n}^{\left(
1\right)  }\left(  xZ_{n-k:n}\right)  =O_{\mathbf{p}}\left(  k/n\right)  ,$
uniformly on $x\geq p^{\gamma}.$ This means that there exists $d>0$ such that
\begin{equation}
\mathbf{P}\left\{  \dfrac{\left(  1-\xi\right)  px^{-1/\gamma}}{n}%
\leq\overline{H}_{n}^{\left(  1\right)  }\left(  xZ_{n-k:n}\right)
<dk/n\right\}  \rightarrow1,\text{ as }n\rightarrow\infty. \label{P2}%
\end{equation}
Therefore, from $\left(  \ref{P1}\right)  $ and $\left(  \ref{P2}\right)  ,$
that $\mathbf{P}\left\{  I_{n}\left(  x\right)  \subset\left[  \left(
1-\xi\right)  px^{-1/\gamma}/n,ck/n\right)  \right\}  $ tends to $1$ as
$n\rightarrow\infty,$ for any $x\geq p^{\gamma},$ where $c:=\max\left(
d,1+\xi\right)  ,$ as sought. Next, let $a_{n}\left(  x\right)  :=\left(
1-\xi\right)  px^{-1/\gamma}/n$ and $b_{n}:=ck/n\downarrow0.$ It is clear that
$0<a_{n}\left(  x\right)  <b_{n}<1$ and $na_{n}\left(  x\right)  <1-\xi$ that
is $na_{n}\left(  x\right)  =O\left(  1\right)  $ for any $x\geq p^{\gamma}.$
Then, by using assertion $(iii)$ of Proposition $\ref{Propo1},$ we get
$\sup_{t\in I_{n}\left(  x\right)  }\left(  \mathcal{L}\left(  t\right)
/\mathcal{L}\left(  \mathbb{V}_{n}\left(  t\right)  \right)  \right)
=O_{\mathbf{p}}\left(  1\right)  ,$ as $n\rightarrow\infty,$ uniformly on
$x\geq p^{\gamma}.$
\end{proof}

\begin{lemma}
\label{Lemma2}Let $\beta_{n}$ and $\widetilde{\beta}_{n}$ be the two empirical
processes respectively defined in $\left(  \ref{betan}\right)  $ and $\left(
\ref{beta-tild}\right)  .$ Then, for all large $n$ and any $1/4<\eta<p/2,$ we
have%
\[
\left(  i\right)  \text{ }\sup_{w>1}w^{\left(  1-\eta\right)  /\gamma
}\left\vert \beta_{n}\left(  w\right)  \right\vert =o_{\mathbf{p}}\left(
1\right)  =\sup_{w>1}w^{\left(  1-\eta\right)  /\gamma}\left\vert
\widetilde{\beta}_{n}\left(  w\right)  \right\vert .
\]
Moreover, for any small $\epsilon_{0}>0,$ we have uniformly on
$0<s<x^{-1/\gamma_{1}},$%
\[
\left(  ii\right)  \text{ }\beta_{n}\left(  \xi_{n}\left(  s\right)  \right)
-\beta_{n}\left(  s^{-\gamma_{1}}\right)  =o_{\mathbf{p}}\left(  s^{\left(
1-\eta\right)  \gamma_{1}/\gamma\pm\epsilon_{0}}\right)  ,
\]
where $\xi_{n}\left(  s\right)  $ is of the form $\left(  \ref{ksi}\right)  .$
\end{lemma}

\begin{proof}
Recall that%
\[
\beta_{n}\left(  w\right)  =\sqrt{\frac{n}{k}}\left\{  \alpha_{n}\left(
\theta\right)  -\alpha_{n}\left(  \theta-\overline{H}^{\left(  1\right)
}\left(  wZ_{n-k:n}\right)  \right)  \right\}  ,\text{ for }0<\overline
{H}^{\left(  1\right)  }\left(  wZ_{n-k:n}\right)  <\theta,
\]
and note that $\left\{  \alpha_{n}\left(  \theta\right)  -\alpha_{n}\left(
\theta-t\right)  ,\text{ }0<t<\theta\right\}  \overset{\mathcal{D}}{=}\left\{
\alpha_{n}\left(  t\right)  ,\text{ }0<t<\theta\right\}  .$ Then without loss
of generality, we may write $\beta_{n}\left(  w\right)  =\sqrt{\dfrac{n}{k}%
}\alpha_{n}\left(  \overline{H}^{\left(  1\right)  }\left(  wZ_{n-k:n}\right)
\right)  ,$ for any $w>1.$ It is easy to verify that
\[
w^{\left(  1-\eta\right)  /\gamma}\left\vert \beta_{n}\left(  w\right)
\right\vert =\frac{n^{1-\eta}}{\sqrt{k}}\left(  \frac{\overline{H}^{\left(
1\right)  }\left(  wZ_{n-k:n}\right)  }{w^{-1/\gamma}}\right)  ^{\eta-1}%
\frac{n^{\eta}\left\vert \mathbb{U}_{n}\left(  \overline{H}^{\left(  1\right)
}\left(  wZ_{n-k:n}\right)  \right)  -\overline{H}^{\left(  1\right)  }\left(
wZ_{n-k:n}\right)  \right\vert }{\left(  \overline{H}^{\left(  1\right)
}\left(  wZ_{n-k:n}\right)  \right)  ^{\eta-1}}.
\]
Making use of Proposition $\ref{Propo1}$ (assertion $\left(  ii\right)  ),$ we
infer that
\[
\frac{n^{\eta}\left\vert \mathbb{U}_{n}\left(  \overline{H}^{\left(  1\right)
}\left(  wZ_{n-k:n}\right)  \right)  -\overline{H}^{\left(  1\right)  }\left(
wZ_{n-k:n}\right)  \right\vert }{\left(  \overline{H}^{\left(  1\right)
}\left(  wZ_{n-k:n}\right)  \right)  ^{\eta-1}}=O_{\mathbf{p}}\left(
1\right)  ,
\]
uniformly on $w>1,$ for any $1/4<\eta<p/2.$ It follows that%
\[
w^{\left(  1-\eta\right)  /\gamma}\left\vert \beta_{n}\left(  w\right)
\right\vert =O_{\mathbf{p}}\left(  \frac{n^{1-\eta}}{\sqrt{k}}\right)  \left(
\frac{\overline{H}^{\left(  1\right)  }\left(  wZ_{n-k:n}\right)
}{w^{-1/\gamma}}\right)  ^{\eta-1}.
\]
On the other hand, by using Potter's inequalities $\left(  \ref{Potter}%
\right)  $ to both $\overline{F}$ and $\overline{G},$ with similar arguments
as those used in the proof of Lemma 4.1 (assertion $(i))$ in \cite{BMN-2015},
we may readily show that $\lim_{z\rightarrow\infty}\sup_{w>1}\left\vert
\overline{H}^{\left(  1\right)  }\left(  zw\right)  /\overline{H}\left(
z\right)  -pw^{-1/\gamma}\right\vert =0.$ Since $\dfrac{n}{k}\overline
{H}\left(  Z_{n-k:n}\right)  \overset{\mathbf{p}}{\rightarrow}1,$ then
$w^{\left(  1-\eta\right)  /\gamma}\left\vert \beta_{n}\left(  w\right)
\right\vert =O_{\mathbf{p}}\left(  \left(  k/n\right)  ^{\eta}k^{1/2-2\eta
}\right)  $ uniformly on $w>1.$ We have $k\rightarrow\infty,$ $k/n\rightarrow
0$ and $1/2-2\eta<0,$ for any $1/4<\eta<p/2,$ therefore $w^{\left(
1-\eta\right)  /\gamma}\left\vert \beta_{n}\left(  w\right)  \right\vert
=o_{\mathbf{p}}\left(  1\right)  ,$ uniformly on $w>1,$ leading to the first
result of assertion $\left(  i\right)  .$ The proof of the second result
follows similar arguments. For assertion $\left(  ii\right)  ,$ we first note
that%
\[
\left\{  \beta_{n}\left(  \xi_{n}\left(  s\right)  \right)  -\beta_{n}\left(
s^{-\gamma_{1}}\right)  ,0<s<1\right\}  \overset{\mathcal{D}}{=}\left\{
\beta_{n}\left(  \xi_{n}\left(  s\right)  -s^{-\gamma_{1}}\right)
,0<s<1\right\}  .
\]
By using assertion $\left(  i\right)  ,$ we write $\beta_{n}\left(  \xi
_{n}\left(  s\right)  -s^{-\gamma_{1}}\right)  =o_{\mathbf{p}}\left(
\left\vert \xi_{n}\left(  s\right)  -s^{-\gamma_{1}}\right\vert ^{\left(
\eta-1\right)  /\gamma}\right)  ,$ uniformly on $0<s<x^{-1/\gamma_{1}}$ and by
applying (once again) Potter's inequalities $\left(  \ref{Potter}\right)  $ to
$F^{\mathbb{\leftarrow}}\left(  1-\cdot\right)  ,$ we infer that $\xi
_{n}\left(  s\right)  -s^{-\gamma_{1}}=o_{\mathbf{p}}\left(  s^{-1/\gamma
_{1}\pm\epsilon_{0}}\right)  .$ It follows that%
\[
\beta_{n}\left(  \xi_{n}\left(  s\right)  -s^{-\gamma_{1}}\right)
=o_{\mathbf{p}}\left(  s^{\left(  1-\eta\right)  \gamma_{1}/\gamma\pm
\epsilon_{0}}\right)  ,
\]
which completes the proof.
\end{proof}


\begin{thebibliography}{99999999999999999999999999999999999999}                                                           %


\bibitem[Aalen(1976)]{Aalen}Aalen, O., 1976. Nonparametric inference in
connection with multiple decrement models. \textit{Scand. J. Statist.}
\textbf{3}, 15-27.

\bibitem[Beirlant \textit{et al}.(2007)]{BeGDFF07}Beirlant, J., Guillou, A.,
Dierckx, G. and Fils-Villetard, A., 2007. Estimation of the extreme value
index and extreme quantiles under random censoring. \textit{Extremes}
\textbf{10}, 151-174.

\bibitem[Beirlant \textit{et al}.(2016)]{BBWG-16}Beirlant, J., Bardoutsos, A.,
de Wet, T. and Gijbels, I., 2016. Bias reduced tail estimation for censored
Pareto type distributions. \textit{Statist. Probab. Lett.} \textbf{109, }78-88.

\bibitem[Benchaira \textit{et al.}(2016)]{BMN2016}Benchaira, S., Meraghni, D.
and Necir, A. , 2016. Tail product-limit process for truncated data with
application to extreme value index estimation.\ \textit{Extremes }%
\textbf{19}\textit{, }219-251.

\bibitem[Bingham \textit{et} \textit{al.}(1987)]{Bin87}Bingham, N.H., Goldie,
C.M. and Teugels, J.L., 1987. Regular Variation. Cambridge University Press.

\bibitem[Brahimi \textit{et} \textit{al.}(2015)]{BMN-2015}Brahimi, B.,
Meraghni, D. and Necir, A., 2015. Approximations to the tail index estimator
of a heavy-tailed distribution under random censoring and application.
\textit{Math. Methods Statist.} \textbf{24,} 266-279.

\bibitem[Cs\"{o}rg\H{o} \textit{et al}.(1986)]{CsCsHM86}Cs\"{o}rg\H{o}, M.,
Cs\"{o}rg\H{o}, S., Horv\'{a}th, L. and Mason, D.M., 1986. Weighted empirical
and quantile processes. \textit{Ann. Probab.} \textbf{14}, 31-85.

\bibitem[Cs\"{o}rg\H{o}(1996)]{Cs96}Cs\"{o}rg\H{o}, S., 1996. Universal
Gaussian approximations under random censorship. \textit{Ann. Statist.}
\textbf{24}, 2744-2778.

\bibitem[Deheuvels and Einmahl(1996)]{DeEn96}Deheuvels, P. and Einmahl,
J.H.J., 1996. On the strong limiting behavior of local functionals of
empirical processes based upon censored data. \textit{Ann. Probab.}
\textbf{24}, 504-525.

\bibitem[Efron(1967)]{Efron}Efron, B., 1967. The two-sample problem with
censored data. \textit{Proceedings of the Fifth Berkeley Symposium on
Mathematical Statistics }\textbf{4,} 831-552.

\bibitem[Einmahl \textit{et al}.(2008)]{EnFG08}Einmahl, J.H.J.,
Fils-Villetard, A. and Guillou, A., 2008. Statistics of extremes under random
censoring. \textit{Bernoulli} \textbf{14}, 207-227.

\bibitem[Einmahl and Koning(1992)]{EnKo92}Einmahl, J.H.J. and Koning, A.J.,
1992. Limit theorems for a general weighted process under random censoring.
\textit{Canad. J. Statist.} \textbf{20}, 77-89.

\bibitem[Fleming and Harrington(1984)]{FH84}Fleming, T.R. and Harrington, D.
P., 1984. Nonparametric estimation of the survival distribution in censored
data. Comm. \textit{Statist. A-Theory Methods} \textbf{13,} 2469-2486.

\bibitem[Gardes and Stupfler(2015)]{GS2015}Gardes, L. and Stupfler, G., 2015.
Estimating extreme quantiles under random truncation. \textit{TEST}
\textbf{24}, 207-227.

\bibitem[de Haan and Stadtm\"{u}ller(1996)]{deHS96}de Haan, L. and
Stadtm\"{u}ller, U., 1996. Generalized regular variation of second order.
\textit{J.} \textit{Australian Math. Soc.} (Series A) \textbf{61}, 381-395.

\bibitem[de Haan and Ferreira(2006)]{deHF06}de Haan, L. and Ferreira, A.,
2006. Extreme Value Theory: An Introduction. \textit{Springer}.

\bibitem[Hill(1975)]{Hill75}Hill, B.M., 1975. A simple general approach to
inference about the tail of a distribution. \textit{Ann. Statist}. \textbf{3}, 1163-1174.

\bibitem[Huang and Strawderman(2006)]{HS6}Huang, X. and Strawderman, R. L,
2006. A note on the Breslow survival estimator. \textit{J. Nonparametr. Stat.}
\textbf{18}, 45-56.

\bibitem[Kaplan and Meier(1958)]{KM58}Kaplan, E.L. and Meier, P., 1958.
Nonparametric estimation from incomplete observations. \textit{J. Amer.
Statist. Assoc.} \textbf{53,} 457-481.

\bibitem[Ndao \textit{et al.}(2014, 2016)]{NDD2014}Ndao, P., Diop, A. and
Dupuy, J.-F., 2014. Nonparametric estimation of the conditional tail index and
extreme quantiles under random censoring. \textit{Comput. Statist. Data Anal.}
\textbf{79}, 63--79.

\bibitem[Ndao \textit{et al.}(2015)]{NDD2015}Ndao, P., Diop, A. and Dupuy,
J.-F., 2016. Nonparametric estimation of the conditional extreme-value index
with random covariates and censoring. \textit{J. Statist. Plann. Inference
}\textbf{168}, 20-37.

\bibitem[Nelson(1972)]{Nelson}Nelson,W., 1972. Theory and applications of
hazard plotting for censored failure data. \textit{Techno-metrics}
\textbf{14}, 945-966.

\bibitem[Peng(1998)]{peng98}Peng, L., 1998. Asymptotically unbiased estimators
for the extreme-value index. \textit{Statist. Probab. Lett.} \textbf{38}, 107-115.

\bibitem[Reiss and Thomas(2007)]{ReTo07}Reiss, R.-D. and Thomas, M., 2007.
Statistical Analysis of Extreme Values with Applications to Insurance,
Finance, Hydrology and Other Fields, 3rd ed. \textit{Birkh\"{a}user Verlag,
}Basel, Boston, Berlin.

\bibitem[Resnick(2006)]{Res06}Resnick, S., 2006. Heavy-Tail Phenomena:
Probabilistic and Statistical Modeling. Springer.

\bibitem[Shorack and Wellner(1986)]{SW86}Shorack, G.R. and Wellner, J.A.,
1986. Empirical Processes with Applications to Statistics. \textit{Wiley.}

\bibitem[Worms and Worms(2014)]{WW2014}Worms, J. and Worms, R., 2014. New
estimators of the extreme value index under random right censoring, for
heavy-tailed distributions. \textit{Extremes} \textbf{17}, 337-358.
\end{thebibliography}
\end{document}